\documentclass{amsart}
\usepackage[utf8]{inputenc}
\usepackage{amssymb}
\usepackage{graphicx}
\newtheorem{theorem}{Theorem}[section]
\newtheorem{lemma}[theorem]{Lemma}
\newtheorem{corollary}[theorem]{Corollary}
\newtheorem{proposition}[theorem]{Proposition}
\usepackage{comment}
\usepackage{xcolor}
\usepackage{caption}
\usepackage{subcaption}
\usepackage{booktabs}
\theoremstyle{definition}

\newtheorem{assumption}[theorem]{Assumption}
\usepackage{multirow}
\usepackage{hyperref}
\theoremstyle{remark}
\newtheorem{remark}[theorem]{Remark}
\newcommand{\V}{\mathbb{V}}

\newcommand{\R}{\mathbb{R}}

\newcommand{\N}{\mathbb{N}}

\newcommand{\E}{{\mathbb{E}}}

\newcommand{\F}{{\mathcal{F}}} 
\renewcommand{\P}{{\mathbb{P}}} 

\usepackage{algorithm}
\usepackage{algpseudocode}

\newcommand{\Tau}{\mathrm{T}}

\title[On the Computational Cost of SGLD]{On the Computational Cost of Stochastic Gradient Langevin Dynamics}

\author{Mateusz B. Majka\textsuperscript{1}, 
Tigran Nagapetyan, 
\L ukasz Szpruch\textsuperscript{2,3}, 
Yue Wu\textsuperscript{4}, 
Danqi~Zhuang\textsuperscript{4}}

\date{}

\begin{document}
\maketitle
\footnotetext[1]{School of Mathematical and Computer Sciences, Heriot-Watt University and Maxwell Institute for Mathematical Sciences, Edinburgh, UK}
\footnotetext[2]{School of Mathematics, University of Edinburgh, UK}
\footnotetext[3]{The Alan Turing Institute, London, UK}
\footnotetext[4]{Department of Mathematics and Statistics, University of Strathclyde, Glasgow, UK}

\begin{abstract}
    Stochastic Gradient Langevin Dynamics (SGLD) reduces the cost of Langevin-based sampling by replacing full-dataset drift evaluations with mini-batch approximations, but the resulting subsampling error may offset this computational saving. We study this trade-off for stochastic differential equations with finite-sum drifts and compare the computational cost of SGLD with that of the Euler-Maruyama (EM) method. For a prescribed mean-square accuracy $\varepsilon^2$, we derive complexity estimates that explicitly track the dependence on the dataset size $m$, mini-batch size $s$, and accuracy parameter $\varepsilon$. The resulting comparison reveals distinct parameter regimes in which either method is preferable. In particular, EM can have lower leading-order cost only in a small-data, aggressive-subsampling regime, whereas SGLD is favoured over most of the remaining parameter space. In the practically relevant regime $s \ll m$, the transition between the two methods occurs at the scale $m \asymp \varepsilon^{-1}$. We complement the theoretical analysis with numerical experiments based on a Gaussian Bayesian inference model, which examine the predicted cost regimes together with the underlying discretisation error and variance estimates.
\end{abstract}

\section{Introduction}

Stochastic Gradient Langevin Dynamics (SGLD) is a popular algorithm for approximate sampling from probability distributions. It can be used instead of classical alternatives such as the Unadjusted Langevin Algorithm (ULA) in situations where reducing the computational cost is desirable due to large data sets involved. SGLD achieves this by subsampling only a small number of data points from the large data set that is needed to evaluate the complete drift in ULA.

While subsampling reduces the computational cost, it increases the variance and hence it is not clear whether the overall computational complexity (i.e., the cost required to achieve a prescribed accuracy measured in the mean squared error) is reduced or increased. The benefits of using subsampling (i.e., choosing SGLD over ULA) seem to be problem-dependent and there are examples in the literature where either method turns out to be a better choice \cite{PillaiSmithZaman2026, truecost, BrosseDurmusMoulines2018}.

There has been considerable interest in studying the advantages and disadvantages of SGLD compared to ULA \cite{PillaiSmithZaman2026, LuYeZhou2025,BakerFearnheadFoxNemeth2019,NemethFearnhead2021,AicherPutchaNemethFearnhead2025, DalalyanKaragulyan2019, truecost, BrosseDurmusMoulines2018, TehThieryVollmer2016, WellingTeh2011}. The aspects of SGLD that have been studied include, among others, the bias \cite{DalalyanKaragulyan2019, BrosseDurmusMoulines2018, VollmerZygalakisTeh2016, LiWang2025, ZhangAkyildizDamoulasSabanis2023}, the mean-squared error analysis \cite{LuYeZhou2025} and the computational cost \cite{truecost}. Various approaches to the variance reduction for stochastic gradient methods have also been explored extensively \cite{BakerFearnheadFoxNemeth2019, Chatterji2018, KinoshitaSuzuki2022, ChenLuXu2022}.

In the present paper, similarly to \cite{truecost}, we focus on analysing the computational cost of SGLD and comparing it to the corresponding cost of ULA. The goal of this paper, compared to \cite{truecost}, is to provide a more detailed insight into the difference in the computational cost between SGLD and ULA, and analyse how these costs depend on the total number of data points $m$, the number of subsamples $s$, and the desired accuracy $\varepsilon$. In particular, we will identify the exact regimes (described in terms of $m$, $s$ and $\varepsilon$) in which SGLD is beneficial over ULA, or vice versa. Our theoretical findings will be confirmed by numerical experiments on examples related to Bayesian inference.

It is important to acknowledge that stochastic gradient methods have been used for many numerical schemes more advanced than the simple SGLD model considered in this paper, such as kinetic Langevin models \cite{ChadaLeimkuhlerPaulinWhalley2023, LeimkuhlerPaulinWhalley2024, PaulinWhalley2026}, Hamiltonian Monte Carlo \cite{AkyildizSabanis2024, ChakMonmarche2026, ChenDingCarin2015, ChenFoxGuestrin2014, ZouXuGu2019}, tamed ULA \cite{LimNeufeldSabanisZhang2024, LovasLytrasRasonyiSabanis2023} or Multi-level Monte Carlo \cite{GilesMajkaSzpruchVollmerZygalakis2020, MajkaSabateVidalesSzpruch2023}, to name just a few. However, since the analysis of the computational cost in the basic SGLD case is already challenging, in the present work we focus exclusively on the most straightforward model.

In this work we consider sampling problems, where the potential function naturally exhibits a finite-sum structure. The main example comes from Bayesian inference, where a priori uncertainty
in a parameter $x$ is modeled using a probability density $\pi_{0}(x)$
called the prior. Suppose that for a
fixed parameter $x$ the data $\left\{ y_{i}\right\} _{i=1,\dots,m}$
are i.i.d. with density ${\pi(y|x)}$. Here $m$ is the number of data points, and in typical applications it is very large \cite{GilesMajkaSzpruchVollmerZygalakis2020, BrosseDurmusMoulines2018, LuYeZhou2025}. By Bayes'
rule the posterior is given by
\[
\pi(x)\propto\pi_{0}(x)\prod_{i=1}^{m}{\pi(y_{i}|x)} \,.
\]
It is well-known that this distribution is invariant for the Langevin equation $ \mathrm{d}X_t = \nabla U(x) \,\mathrm{d}t + \sqrt{2}  \,\mathrm{d}W_t$ with 
\begin{equation}\label{eq: BayesianDrift}
\nabla U(x)=\nabla\log{\pi_{0}}(x)+\sum_{i=1}^{m}\nabla\log{\pi(y_{i}|x)}.
\end{equation}
See e.g.\ \cite[Section 5.1]{GilesMajkaSzpruchVollmerZygalakis2020}] and the references therein for further discussion of this setting. This motivates focusing our analysis on SDEs with drifts of a finite-sum form.

However, for notational convenience, and in order to keep our analysis as general as possible, instead of directly working with the drift \eqref{eq: BayesianDrift}, we will follow the notation from \cite{MajkaMijatovicSzpruch2020} and consider more general SDEs of the form $dX_t = a(X_t)\,\mathrm{d}t + \beta \,\mathrm{d}W_t$, where $\beta>0$ and the drift is given by
\begin{equation}\label{eq: introDrift}
a(x) = \frac{1}{m} \sum_{i=1}^m b(x,\xi_i) 
\end{equation}
for a general function $b: \mathbb{R}^d \times \mathbb{R}^n \to \mathbb{R}^d$ and parameters $\xi_i \in \mathbb{R}^n$
(see Section \ref{section:general} for details). Evidently, \eqref{eq: introDrift} is a direct generalization of \eqref{eq: BayesianDrift} (note that we introduce the factor $\frac{1}{m}$ in \eqref{eq: introDrift} for convenience and its role will be explained in Section \ref{section:general}). For this class of SDEs, we will consider their standard Euler-Maruyama (EM) discretisations (which for the specific gradient drift \eqref{eq: BayesianDrift} would correspond to the well-known ULA scheme), as well as their SGLD counterparts, which, instead of using all $m$ data points to evaluate the drift in each step of the algorithm, rely only on a smaller subsample of size $s \ll m$. 

\subsection*{Comparison to \cite{truecost}}

Our paper is clearly motivated by \cite{truecost} and attempts to achieve the same goal of comparing the computational cost of SGLD against that of ULA (or, in the more general, non-Langevin setting \eqref{eq: introDrift}, the cost of the subsampled EM scheme against that of the "full" EM scheme). However, compared to \cite{truecost}, we provide a more careful mathematical analysis of cost regimes for both ULA and SGLD, we re-evaluate the regimes and obtain more accurate formulas, and, importantly, we verify the correctness of our regimes on numerical examples (which was not done in \cite{truecost}). We believe that our analysis presents a substantial improvement on \cite{truecost} and our conclusions are more accurate than what can be found in \cite{truecost}. One of our main motivations for why this improved analysis of the cost of SGLD was needed, was the work on Multi-level Monte Carlo methods with stochastic gradients \cite{GilesMajkaSzpruchVollmerZygalakis2020, MajkaSabateVidalesSzpruch2023}, where a careful analysis of the computational cost of numerical methods related to SGLD plays a substantial role. 

Note that a crucial difference between our paper and \cite{truecost} is that we work with drifts rescaled by $\frac{1}{m}$ rather than directly with the unscaled version \eqref{eq: BayesianDrift}. It is an arbitrary choice in the sense that, after  rescaling the noise in the SDE by $\frac{1}{\sqrt{m}}$, the invariant distribution remains unchanged (see Section \ref{sec:validation} for a more detailed discussion). However, as we will demonstrate, this choice leads to a much cleaner analysis and more accurate formulas for cost regimes, since this allows us to  work with model parameters independent of $m$ (and in particular our step-sizes are independent of $m$).

Furthermore, note that \cite{truecost} concluded that the computational cost of SGLD is independent of the mini-batch (subsample) size. In the present paper, we show that the conclusion in \cite{truecost} was in fact inaccurate, and it was just an artifact of the method of analysis adopted in \cite{truecost} (i.e., the bounds obtained in \cite{truecost} were not sharp enough). By improving the analysis from \cite{truecost}, we conclude that the computational cost of SGLD is in fact heavily dependent on the mini-batch size and we carefully describe all the possible cost regimes (in terms of $m$, $s$ and $\varepsilon$).

\subsection*{Paper structure} The remainder of this paper is organised as follows. In Section \ref{section:general} we introduce our setting and assumptions, and we formulate our main results describing the computational cost of the standard Euler-Maruyama (EM) method and the SGLD scheme. We also discuss in detail how to compare the two costs depending on the values of $m$, $s$ and $\varepsilon$. In Section \ref{section:numerical cost} we introduce the setting of our numerical experiments and we present a verification of the theoretical cost from Section \ref{section:general} via suitable numerical examples. In Section \ref{sec:discussion} we present our detailed theoretical analysis of the cost (which includes formulating several auxiliary results about the corresponding SDE and its numerical approximations). We also present the proofs of the main theoretical results from Section \ref{section:general}. In Section \ref{section: additional numerical} we present additional numerical experiments verifying some of the auxiliary results from Section \ref{sec:discussion}, such as the bounds on the variance of the EM and SGLD methods. Several proofs of the auxiliary results from Section \ref{sec:discussion} are deferred to the Appendices \ref{sec:proofs} and \ref{sec:numerical}.

\section{Main results}\label{section:general}

\subsection{Preliminaries and assumptions}

Let $(W_t)_{t \geq 0}$ be the standard Brownian motion in $\mathbb{R}^d$ on a filtered probability space $(\Omega_W, \mathcal{F}^W, \{\mathcal{F}^W_t\}_{t\geq 0},\mathbb{P}_W)$. We are interested in the following $\mathbb{R}^d$-valued Markovian SDE over $[0,\infty)$
\begin{equation}\label{generalSDE}
dX_t = a(X_t)\,\mathrm{d}t + \beta \,\mathrm{d}W_t \,, 
\end{equation}
with initial condition $X_0 \in L^2(\Omega_W)$, where $a : \mathbb{R}^d \to \mathbb{R}^d$ and $\beta > 0$.

Motivated by the applications outlined in the previous section, we will focus on drifts $a$ of the following special form
\begin{equation}\label{drift}
a(x) = \frac{1}{m} \sum_{i=1}^m b(x,\xi_i) \,,
\end{equation}
where 
$b: \mathbb{R}^d \times \mathbb{R}^n \to \mathbb{R}^d$ is a function parametrized by the data-points $(\xi_i)_{i=1}^{m} \subset \mathbb{R}^n$ and $m \geq 1$ is the size of the dataset. For further discussion on SDEs with drifts of this type in the context of subsampling, see e.g.\ \cite[Example 2.15]{MajkaMijatovicSzpruch2020} or \cite[Section 5.1]{GilesMajkaSzpruchVollmerZygalakis2020}. For the drift defined in \eqref{drift}, we impose the following assumptions:

\begin{assumption}[Lipschitz continuity]\label{ass:sol1}
	There exists an $L > 0$ such that for all $x$, $y \in \mathbb{R}^d$ and all $\xi \in \mathbb{R}^n$ we have
	\begin{equation*}
	|b(x, \xi) - b(y, \xi)| \le L \left|x-y\right|.
	\end{equation*}
\end{assumption}
\begin{assumption}[Dissipativity]\label{ass:sol2}
	There exists a $K > 0$ such that for all $x$, $y \in \mathbb{R}^d$ and all $\xi \in \mathbb{R}^n$ we have
	\begin{equation*}
	\langle x - y , b(x, \xi) - b(y, \xi) \rangle \le - K\left|x-y\right|^2. 
	\end{equation*}
\end{assumption}
Note that under these assumptions, the drift $a$ satisfies
\begin{equation}\label{eq:a1}
|a(x) - a(y)| \le L\left|x-y\right|,
\end{equation}
and 
\begin{equation}\label{eq:a2}
\langle x-y , a(x) - a(y) \rangle \le - K\left|x-y\right|^2
\end{equation}
for all $x$, $y \in \mathbb{R}^d$, with the same constants $L$ and $K$ as above, i.e., with constants independent of the number of data points $m$.

\begin{remark}
	Note that if the $\frac{1}{m}$ factor in \eqref{drift} were absent, then under Assumptions \ref{ass:sol1} and \ref{ass:sol2}, conditions \eqref{eq:a1} and \eqref{eq:a2} would be satisfied with the constants $Lm$ and $Km$, instead of $L$ and $K$, respectively. As we will see in the sequel, including this factor will make it easier for us to track the dependence of various constants on $m$. 
\end{remark}
Note that Eqn. \eqref{eq:a2} leads to 
\begin{equation}\label{eqn:inequalitya}
\langle x, a(x)\rangle
\le -\frac{K}{2}|x|^{2}
+ \frac{1}{2K}|a(0)|^{2},
\end{equation}
where the proof of this observation is postponed to Appendix \ref{sec:proofs}. 

For some results we will also need the following additional assumption. Denote the canonical basis for $\mathbb{R}^d$ as $\{e_i\}_{i=1}^d$. Consider the drift $a:\mathbb{R}^d\to \mathbb{R}^d$ in \eqref{generalSDE} which can be rewritten as $a = (a_1, \ldots , a_d)$ with $a_i:\mathbb{R}^d\to \mathbb{R}$ and defined as $a_i(x)\mapsto \langle a(x), e_i\rangle$.

\begin{assumption}\label{ass:additional}
	Assume that the drift $a = (a_1, \ldots , a_d)$ in \eqref{generalSDE} is twice-differentiable and for each  $i \in \{ 1, \ldots , d \}$, there exist constants $C_{a_i^{(1)}}$ and $C_{a_i^{(2)}}$ such that for all $x \in \mathbb{R}^d$ $$|\nabla a_i(x)| \le C_{a_i^{(1)}} \qquad \text{ and } \qquad  \| \nabla^2 a_i(x) \|_{\operatorname{op}} \le C_{a_i^{(2)}},$$
	where $\| \nabla^2 a_i(x) \|_{\operatorname{op}}$ is the operator norm of the Hessian matrix of $a_i$.
\end{assumption}

\subsection{Numerical schemes}

\leavevmode

We will now introduce the notation for the Euler-Maruyama (EM) and the SGLD schemes, which are the focal point of this paper. For a fixed stepsize $h> 0$, let the time grid 
\begin{equation}\label{eqn:timegrid}
\Tau_h:=\{t_k=kh, k \in \mathbb{N}\cup \{0\}\}  
\end{equation}
be given. On $\Tau_h$ we define the EM scheme by considering the following discretisation of \eqref{generalSDE}:
\begin{equation}\label{discretisation}
\theta_{k+1} = \theta_k + ha(\theta_k) + \beta \sqrt{h} Z_{k+1} \,,
\end{equation}
where $(Z_k)_{k=1}^{\infty}$ are i.i.d.\ and $Z_k \sim \mathcal{N}(0,I_d)$. In Section \ref{section:SMC}, we will take $Z_{k+1} = \frac{1}{\sqrt{h}} \left( W_{t_{k+1}} - W_{t_k} \right)$. Clearly $\theta_{k+1}$ is $\mathcal{F}^W_{t_{k+1}}$-adapted.
\subsubsection{Subsampling}
\leavevmode

A general "stochastic gradient" setting can be introduced by replacing the exact drift $a$ in \eqref{discretisation} with its unbiased estimator. Below we introduce the notation for a general class of estimators of drifts $a$ of the form \eqref{drift}. Let
$\hat{b} : \mathbb{R}^d \times \mathbb{R}^{\hat{n}} \to \mathbb{R}^d$ and consider a sequence $(U_k)_{k=0}^{\infty}$ of $\mathbb{R}^{\hat{n}}$-valued i.i.d.\ random variables on a probability space $(\Omega_U, \mathcal{F}^U, \{\mathcal{F}^U_k\}_{k \in \mathbb{N}},\mathbb{P}_U)$, independent of $(Z_k)_{k=1}^{\infty}$, such that for all $x \in \mathbb{R}^d$ and all $k \geq 0$ we have
\begin{equation*}
\mathbb{E}_U[\hat{b}(x,U_k)]  =\frac{1}{m} \sum_{i=1}^m b(x,\xi_i)= a(x).
\end{equation*} 
Note that $\xi_i \in \mathbb{R}^n$ and $U_k \in \mathbb{R}^{\hat{n}}$ but typically $\hat{n}\neq n$ (cf.\ the examples below). The most commonly used examples are the subsampling with and without replacement models.

\medskip
\noindent\textbf{\emph{Subsampling with replacement}}
\leavevmode

For the drift $a$ introduced in \eqref{drift} we choose the estimator  
\begin{equation}\label{estimator}
\hat{b}(x,U_k) := \frac{1}{s} \sum_{i=1}^{s} b(x, \xi_{U_k^i}),
\end{equation} 
where, for each $k \geq 0$, $U_k = (U^i_k)_{i=1}^s\subset \mathbb{R}^s$ with $U^i_k$ being i.i.d.\ and $U^i_k \sim \operatorname{Unif}(\{ 1, \ldots, m \})$. In this case, $\hat{n}=s$, $s \le m$.

\medskip
\noindent\textbf{\emph{Subsampling without replacement}}
\leavevmode

For a pre-determined integer $s\leq m$, we choose
\begin{equation}\label{eq:bwor}
\hat{b}(x,U_k) := \frac{1}{s} \sum_{i=1}^{m} b(x,\xi_i) U_k^i \,,   
\end{equation}
where, for each $k \geq 1$,  the random variables $(U_k^i)_{i=1}^m\subset \mathbb{R}^m$ are correlated and satisfying 
$$U_k^i\in \{0,1\} \text{ and }\sum_{i=1}^m U_k^i=s.$$
It can be easily derived that
\begin{equation}\label{eqn:subsapling_pb1}
\mathbb{P}(U_k^i = 1) = \frac{s}{m}, \quad \mathbb{P}(U_k^i = 0) = 1 - \frac{s}{m},
\end{equation}
and 
\begin{equation}\label{eqn:subsapling_pb2}
\mathbb{P}(U_k^i = 1 , U_k^j = 1) = \binom{m-2}{s-2} / \binom{m}{s}
\end{equation}
for any $i \neq j$. However, $U_k^i$ are independent for different values of $k$. In this case,  $\hat{n}=m$.

For alternative methods of subsampling, see e.g.\ \cite{ShawWhalley2025}.

In this paper, we consider the SGLD chain $(Y_k)_{k=0}^{\infty}$ generated from subsampling without replacement, given by 
\begin{equation}\label{discretisationSGLD}
Y_{k+1} = Y_k + h\hat{b}_s(Y_k, U_k) + \beta \sqrt{h} Z_{k+1} \,,
\end{equation}
where we denote the drift by 
$\hat{b}_s$
to stress its dependence on the number $s$ of subsamples (i.e., the mini-batch size). This resulting discrete-time random
process $(Y_{k})_{k}$ is defined on the product probability space 
\begin{align}
\label{eq:Omega}
(\Omega, \F, \P) := (\Omega_{W}\times\Omega_{U}, \F^W\otimes \F^{U},
\P_W \otimes \P_{U}),
\end{align}
adapted to a 
discrete-time filtration $(\F_k)_{k}$
on $(\Omega, \F, \P)$ defined as  
\begin{align}
\label{eq:discretefiltration}
\F_k := \F^W_{t_k} \otimes \F^{U}_{k}, \quad \text{ for } k \in \N \cup \{0\}.
\end{align}

Note that we abuse the terminology by using the term SGLD even in cases where the corresponding SDE \eqref{generalSDE} does not exhibit a Langevin structure (i.e., when the drift \eqref{drift} does not have a gradient form).

\subsection{Results on computational cost}

We are now ready to discuss our main results.

Our primary goal is to analyse the computational cost of the standard Monte Carlo
estimator based on the Euler--Maruyama scheme and of the corresponding SGLD
estimator. We aim to approximate
$$
\pi(f) := \int_{\mathbb{R}^d} f(x)\,\pi(dx),
$$
where $\pi$ is the invariant measure of \eqref{generalSDE}. For a prescribed
accuracy $\varepsilon > 0$, we control the mean square error by decomposing it into
bias and variance contributions. To quantify the complexity of both methods, we
decompose the total computational cost as
\begin{align*}
\textbf{Total computational cost}
&= \text{number of paths} \\
&\quad\times \text{simulation time per path} \\
&\quad\times \text{number of steps per unit time} \\
&\quad\times \text{sampling cost per path}.
\end{align*}
We then determine how these quantities should be chosen in order to guarantee
mean square error of order $\varepsilon^2$, and deduce the resulting dependence
of the total cost on $\varepsilon$, $m$, and, in the SGLD case, the batch size $s$.

Throughout the cost analysis, for non-negative quantities
$A=A(\varepsilon,m,s)$ and $B=B(\varepsilon,m,s)$, we write
\[
A \lesssim B
\]
if there exists a constant $C>0$ independent of $\varepsilon$, $m$, and $s$, such that
\[
A \leq C B,
\]
and we write
\[
A \asymp B
\]
if both $A\lesssim B$ and $B\lesssim A$. 
The implicit constant $C$ may depend on fixed problem parameters such as
$d$, $K$, $L$, $L_f$, the initial condition, and the regularity
constants appearing in Assumption~\ref{ass:additional}, which are all independent of $\varepsilon$, $m$, and $s$ in our setting.

\begin{theorem}[Cost of the Euler--Maruyama Monte Carlo estimator]
	\label{lem:cost_em}
	Assume that $f$ is Lipschitz with Lipschitz constant $L_f$, and let Assumption \ref{ass:sol1}, Assumption \ref{ass:sol2} and Assumption \ref{ass:additional} hold. Consider the estimator
	$$
	\widehat{\pi}^{\mathrm{EM}}_{k,N_p}(f)
	:=
	\frac{1}{N_p}\sum_{i=1}^{N_p} f(\theta_k^{\,i}),
	$$
	where $\bigl(\theta_k^{\,i}\bigr)_{k \geq 0}$, $i=1,\dots,N_p$, are i.i.d.\ copies of the EM scheme
	\eqref{discretisation} for SDE \eqref{generalSDE} with $\beta=m^{-1/2}$. Then one can choose $N_p$, $h$, and $k$ so that
	$$
	\mathbb{E}_W\Bigl[\bigl|\pi(f)-\widehat{\pi}^{\mathrm{EM}}_{k,N_p}(f)\bigr|^2\Bigr]
	\leq \varepsilon^2,
	$$
	with total computational cost satisfying
	$$
	\mathrm{Cost}(\mathrm{EM})
	\asymp
	\max\!\bigl(\varepsilon^{-2}m^{-1},1\bigr)\,
	\log(\varepsilon^{-1})\,
	\varepsilon^{-1}\,m.
	$$
	In particular, this yields the two regimes
	\begin{equation}\label{eq:emcost}
	\mathrm{Cost}(\mathrm{EM})
	\asymp
	\begin{cases}
	\varepsilon^{-3}\log(\varepsilon^{-1}), & \text{if } m < \varepsilon^{-2},\\[1ex]
	m\,\varepsilon^{-1}\log(\varepsilon^{-1}), & \text{if } m \geq \varepsilon^{-2}.
	\end{cases} 
	\end{equation}
\end{theorem}

The proof can be found in Section \ref{sec:discussion}.

We also have a corresponding result for the SGLD algorithm.

\begin{theorem}[Cost of the SGLD estimator]
	\label{lem:cost_sgld}
	Assume that $f$ is Lipschitz with Lipschitz constant $L_f$, and that Assumption \ref{ass:sol1}, Assumption \ref{ass:sol2} and Assumption \ref{ass:additional} hold. Consider the estimator
	$$
	\widehat{\pi}^{\mathrm{SGLD}}_{k,N_p}(f)
	:=
	\frac{1}{N_p}\sum_{i=1}^{N_p} f(Y_k^{\,i}),
	$$
	where $\bigl(Y_k^{\,i}\bigr)_{k \geq 0}$, $i=1,\dots,N_p$, are i.i.d.\ copies of the SGLD chain
	\eqref{discretisationSGLD} for SDE \eqref{generalSDE} with $\beta = m^{-1/2}$. Then one can choose $N_p$, $h$, and $k$ so that
	$$
	\mathbb{E}\Bigl[\bigl|\pi(f)-\widehat{\pi}^{\mathrm{SGLD}}_{k,N_p}(f)\bigr|^2\Bigr]
	\leq \varepsilon^2,
	$$
	with total computational cost satisfying
	$$
	\mathrm{Cost}(\mathrm{SGLD})
	\asymp
	\max\!\bigl(\varepsilon^{-2}m^{-1},1\bigr)\,
	\log(\varepsilon^{-1})\,
	\max\!\left(\varepsilon^{-1},\,\varepsilon^{-2}\frac{m-s}{ms}\right)\,s.
	$$
	Consequently, the cost takes the following four forms:
	\begin{equation}\label{eq:sgldcost}
	\mathrm{Cost}(\mathrm{SGLD})
	\asymp
	\begin{cases}
	\varepsilon^{-3}\log(\varepsilon^{-1})\,\dfrac{s}{m},
	& \text{if } m < \varepsilon^{-2} \text{ and } \dfrac{m-s}{ms} < \varepsilon,\\[3ex]
	\varepsilon^{-4}\log(\varepsilon^{-1})\,\dfrac{m-s}{m^2},
	& \text{if } m < \varepsilon^{-2} \text{ and } \dfrac{m-s}{ms} \geq \varepsilon,\\[3ex]
	\varepsilon^{-1}\log(\varepsilon^{-1})\,s,
	& \text{if } m \geq \varepsilon^{-2} \text{ and } \dfrac{m-s}{ms} < \varepsilon,\\[3ex]
	\varepsilon^{-2}\log(\varepsilon^{-1})\,\dfrac{m-s}{m},
	& \text{if } m \geq \varepsilon^{-2} \text{ and } \dfrac{m-s}{ms} \geq \varepsilon.
	\end{cases} 
	\end{equation}
\end{theorem}

The proof can be found in Section \ref{sec:discussion}.

\begin{remark}[Interpretation of the cost regimes]
	The comparison between the Euler--Maruyama Monte Carlo estimator and the SGLD
	estimator is inherently regime-dependent. Indeed, by Theorem~\ref{lem:cost_em} and
	Theorem~\ref{lem:cost_sgld}, the leading-order cost depends on the relative size of
	$m$, $\varepsilon^{-2}$, and the subsampling factor $(m-s)/(ms)$. Consequently,
	there is no uniform dominance of one method over the other across the whole
	parameter space. We also stress that this comparison has been established under
	the standing assumption that $f$ is Lipschitz; for observables with weaker
	regularity or polynomial growth, the variance and bias estimates may scale
	differently, and the resulting complexity comparison may therefore change. This
	point is examined further in the numerical section. 
\end{remark}

\begin{remark}[When subsampling yields a genuine gain]\label{rem:sgld-gain}
	If $m\geq\varepsilon^{-2}$, then SGLD is no more expensive than EM throughout the whole range $1\leq s\leq m$ at the level of the leading-order cost bounds, and is strictly cheaper whenever genuine subsampling produces a strict gain.
	
	Indeed, if
	$(m-s)/ms<\varepsilon$,
	then
	\[
	\operatorname{Cost}(\mathrm{EM})
	\asymp
	m\varepsilon^{-1}\log(\varepsilon^{-1}),
	\qquad
	\operatorname{Cost}(\mathrm{SGLD})
	\asymp
	s\varepsilon^{-1}\log(\varepsilon^{-1}),
	\]
	so SGLD is cheaper than EM by a factor of order $s/m$.
	
	If instead
	$(m-s)/ms\geq \varepsilon$,
	then
	\[
	\operatorname{Cost}(\mathrm{SGLD})
	\asymp
	\varepsilon^{-2}\log(\varepsilon^{-1})\frac{m-s}{m},
	\]
	and
	\[
	\frac{\operatorname{Cost}(\mathrm{SGLD})}
	{\operatorname{Cost}(\mathrm{EM})}
	\asymp
	\frac{m-s}{\varepsilon m^2}
	\leq
	\frac{1}{\varepsilon m}
	\leq
	\varepsilon,
	\]
	where the last inequality is due to $m\geq\varepsilon^{-2}$.
\end{remark}

\begin{remark}[When standard EM can be preferable]\label{rem:euler-preferable}
	EM can be cheaper only in a small-data, aggressive-subsampling
	regime. Indeed, if $m<\varepsilon^{-2}$, then
	\[
	\operatorname{Cost}(\mathrm{EM})
	\asymp
	\varepsilon^{-3}\log(\varepsilon^{-1}).
	\]
	When additionally $(m-s)/ms \geq \varepsilon$, the SGLD cost is
	\[
	\operatorname{Cost}(\mathrm{SGLD})
	\asymp
	\varepsilon^{-4}\log(\varepsilon^{-1})
	\frac{m-s}{m^2}.
	\]
	Consequently,
	\[
	\operatorname{Cost}(\mathrm{EM})
	<
	\operatorname{Cost}(\mathrm{SGLD})
	\]
	if and only if
	\[
	m-s>\varepsilon m^2.
	\]
	Equivalently, EM is cheaper at leading order precisely when
	\[
	m<\varepsilon^{-1},
	\qquad
	\frac{s}{m}<1-\varepsilon m.
	\]
	Outside this region, SGLD has the smaller leading-order cost, with the
	two costs being comparable on the boundary
	\[
	\frac{s}{m}=1-\varepsilon m.
	\]
	This comparison concerns only the leading-order expressions; the
	finite-parameter transition may be shifted by the constants suppressed
	by the notation $\asymp$.
\end{remark}

\begin{figure}[H]
	\centering
	\begin{subfigure}{0.48\textwidth}
		\includegraphics[width=\textwidth]{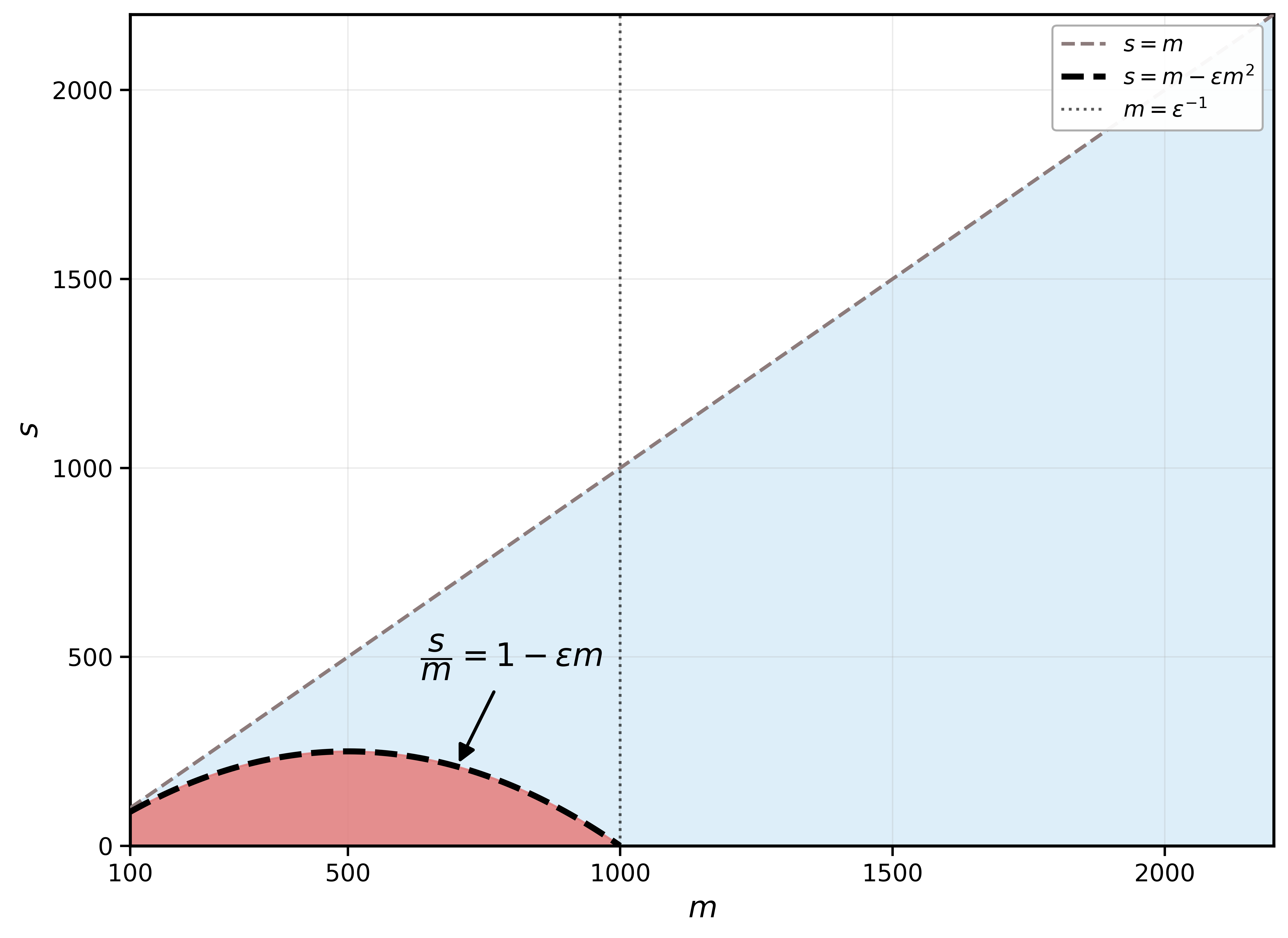}
		\caption{$\varepsilon = 0.001$}
	\end{subfigure}
	\centering
	\begin{subfigure}{0.48\textwidth}
		\includegraphics[width=\textwidth]{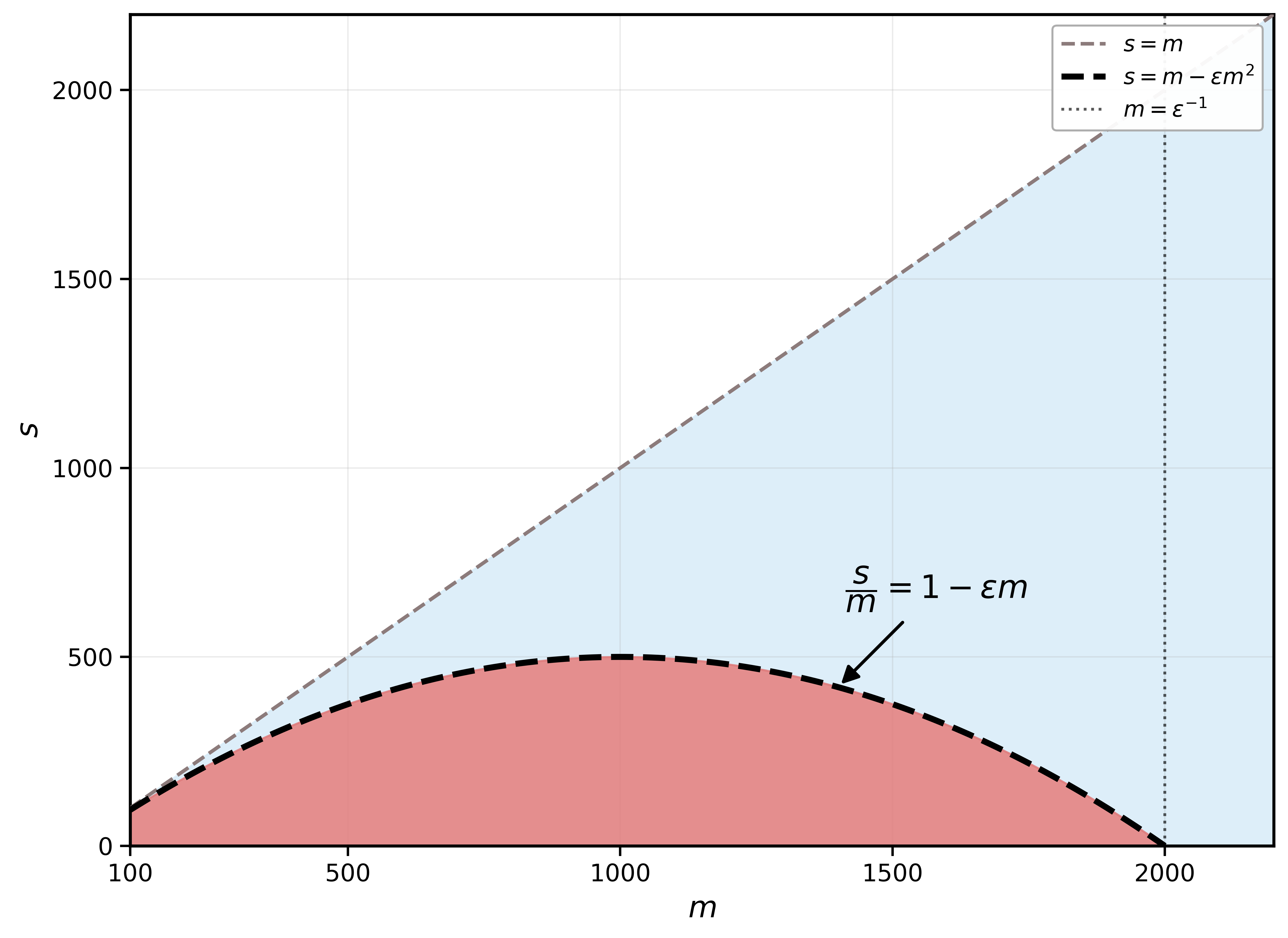}
		\caption{$\varepsilon = 0.0005$}
	\end{subfigure}
	\caption{ Comparison of the computational cost between the SGLD estimator and the standard EM under various accuracy requirements. The red region indicates the parameter pairs $(m,s)$ for which 
		$\mathrm{Cost}(\mathrm{EM}) < \mathrm{Cost}(\mathrm{SGLD})$, while the blue region denotes the opposite.} \label{fig:CostRegion}
\end{figure}

Figure~\ref{fig:CostRegion} illustrates the cost comparison established in
Remarks~\ref{rem:sgld-gain} and~\ref{rem:euler-preferable} and shows that SGLD has the smaller leading-order cost over most of
the $(m,s)$-plane, whereas EM can be preferable only in the small-data,
aggressive-subsampling regime characterised by
\[
m<\varepsilon^{-1},
\qquad
\frac{s}{m}<1-\varepsilon m.
\]

Since in most practical applications the mini-batch size is chosen much smaller
than the total dataset size, it is natural to further examine the asymptotic cost
comparison in the regime $s \ll m$. Note that $s \ll m$ implies $(m-s)/ms \geq \varepsilon$ for small $\varepsilon$. In this case, the SGLD complexity simplifies
substantially as follows: 
\begin{itemize}
	
	\item If $m < \varepsilon^{-2}$, then the cost becomes  $\varepsilon^{-4} \log(\varepsilon^{-1})m^{-1}$
	
	\item If $m > \varepsilon^{-2}$, then the cost becomes $ \varepsilon^{-2} \log(\varepsilon^{-1}) $.
\end{itemize} and the comparison with the Euler--Maruyama estimator can be
represented in the reduced $(\varepsilon,m)$-plane, as shown in
Figure \ref{fig:CostRegion2} (a). The resulting regions indicate that EM is preferable when $m$ is below $\varepsilon^{-1}$, while SGLD becomes cheaper once $m$ exceeds this threshold. Note that the method used to obtain the empirical phase diagrams in Figure \ref{fig:CostRegion2} (b) will be explained in Section \ref{section:numerical cost}.
\begin{figure}[]
	\centering
	\includegraphics[width=1\textwidth]{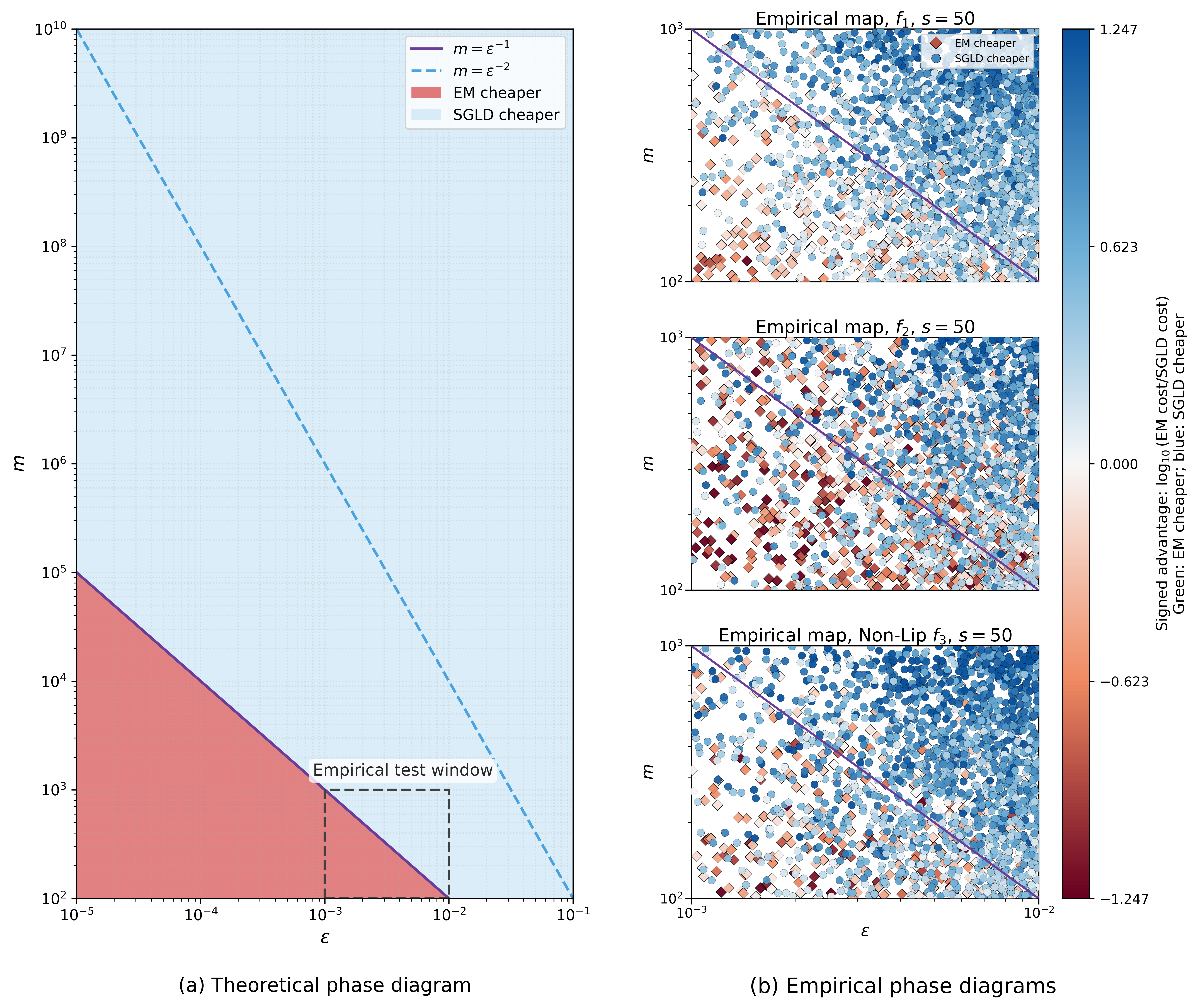}
	\caption{Theoretical and empirical phase diagrams comparing the computational efficiency of EM and SGLD in the $(\varepsilon,m)$ plane ($s\ll m$). Panel~(a) shows the theoretical partition of the parameter space: the red and blue regions indicate the parameter pairs $(\varepsilon,m)$ for which EM or SGLD is predicted to be cheaper, respectively. The dashed rectangle marks the empirical test window used in the numerical experiments. Panel~(b): (b) upper, presents the empirical phase diagram for $f_1(x)=\sin x$ at fixed $s=50$.
		(b) middle, presents the corresponding result for $f_2(x)=\sin x+|x|+1$.
		(b) lower, presents the result for the nonlinear test function $f_3(x)=x^2$.
		In all three empirical panels, each marker represents one sampled parameter pair $(\varepsilon,m)$. The marker colour encodes the signed empirical cost advantage,
		$\log_{10}\!\bigl(\mathrm{Cost(EM)}/\mathrm{Cost(SGLD)}\bigr)$. Negative values indicate an EM advantage, positive values indicate an SGLD advantage, and darker shades denote a larger difference in computational cost.
	} \label{fig:CostRegion2}
\end{figure}

\section{Numerical verification of computational cost}\label{section:numerical cost}

The aim of this subsection is to construct a tractable benchmark model which
retains the structural features required by the theory while allowing for an
accurate empirical evaluation of the computational cost.

\subsection{The numerical set-up}\label{sec:validation}

We consider the following Ornstein–Uhlenbeck (OU) process \begin{equation}\label{eq:Lan_OU}
dX_t=\frac{-(X_t-\mu_p)}{2\sigma_{p}^{2}}\,\mathrm{d}t+\,\mathrm{d}W_{t},
\end{equation}
where $\mu_p\in \R$ and $\sigma_{p}^{2}>0$ are to be determined. It is well-known that SDE \eqref{eq:Lan_OU} admits
a unique invariant measure $\pi$ given by 
$\pi=\mathcal{N}(\mu_{p},\sigma_{p}^{2})$.

We generate $m$ synthetic data points from the hierarchical Gaussian model 
\begin{align}\label{eqn: hierarchical}
\begin{split}
\theta & \sim  \mathcal{N}(0,\sigma_{\theta}^{2}/m),\\
y_{i} \,|\, \theta & \stackrel{i.i.d.}{\sim} \mathcal{N}(\theta,\sigma_{y}^{2})\quad \text{for $i=1,\ldots,m$,}
\end{split}
\end{align}
where positive numbers $\sigma^2_\theta, \sigma^2_y$ are pre-determined, and we
set the posterior distribution to be the invariant measure of \eqref{eq:Lan_OU} by  
\begin{equation}\label{eq:SGaussianPosterior}
\mathcal{N}\left(m^{-1}\big(\sum\limits_{i=1}^{m}y_{i}\big)\Big(\frac{\sigma_{y}^{2}}{\sigma_{\theta}^{2}}+1\Big)^{-1},m^{-1}\left(\frac{1}{\sigma_{\theta}^{2}}+\frac{1}{\sigma_{y}^{2}}\right)^{-1}\right)=: \mathcal{N}(\mu_{p},\sigma_{p}^{2}).
\end{equation}
One can find the derivation of \eqref{eq:SGaussianPosterior} in Appendix \ref{append:posterior}.
The Gaussian hierarchical model in \eqref{eqn: hierarchical}--\eqref{eq:SGaussianPosterior}
is used to generate a dataset-dependent target distribution in a controlled way.
More precisely, the posterior law is Gaussian with explicitly computable mean
$\mu_p$ and variance $\sigma_p^2$, and this posterior is taken as the invariant
measure of the Langevin dynamics. This falls within the general framework of finite-sum drifts \eqref{drift}, with each data
point contributing one term to the drift. At the same time, the resulting target
distribution remains simple enough that the exact value of $\pi(f)$ can be
approximated to very high accuracy by numerical quadrature.

Instead of solving \eqref{eq:Lan_OU} directly, we consider its scaled version 
\begin{equation}\label{eq:Lan_OU_scaled}
dX_t=\frac{-(X_t-\mu_p)}{2m\sigma_{p}^{2}}\,\mathrm{d}t+\frac{1}{\sqrt{m}}\,\mathrm{d}W_{t},
\end{equation}
and one can easily verify that in our general notation $dX_t = a(X_t)\,\mathrm{d}t + \beta \,\mathrm{d}W_t$ and $a(x) = \frac{1}{m} \sum_{i=1}^m b(x,\xi_i)$ introduced in \eqref{generalSDE} and \eqref{drift}, this corresponds to
$$a(x) = - \frac{1}{2m}\frac{x-\mu_p}{\sigma_p^2} =  \sum_{i=1}^{m} \frac{1}{m} \frac{\Big(\frac{\sigma_{y}^{2}}{\sigma_{\theta}^{2}}+1\Big)^{-1}y_i-x}{2\left(\frac{1}{\sigma_{\theta}^{2}}+\frac{1}{\sigma_{y}^{2}}\right)^{-1}}$$ 
with 
$$b(x,y_i) = - \frac{x-\Big(\frac{\sigma_{y}^{2}}{\sigma_{\theta}^{2}}+1\Big)^{-1}y_i}{2\left(\frac{1}{\sigma_{\theta}^{2}}+\frac{1}{\sigma_{y}^{2}}\right)^{-1}}.$$
Clearly Assumption \ref{ass:sol1} and Assumption \ref{ass:sol2} are satisfied with $$L=K=\Big(\frac{1}{\sigma_{\theta}^{2}}+\frac{1}{\sigma_{y}^{2}}\Big)/2,$$
independent of data size $m$. 

The rescaled dynamics \eqref{eq:Lan_OU_scaled} are particularly convenient for
our purposes, because they preserve the invariant distribution while making  the drift satisfy Assumptions \ref{ass:sol1} and~\ref{ass:sol2} with constants $L$ and $K$
independent of $m$, whereas the diffusion coefficient becomes $\beta=1/\sqrt m$.
This is essential for testing the theoretical cost estimates from Theorems \ref{lem:cost_em} and \ref{lem:cost_sgld}, since those estimates rely on the fact that the contraction and
regularity constants remain stable as $m$ varies, while the variance per path
decreases like $m^{-1}$.

The full Euler discretization with $\xi_{k}\stackrel{i.i.d.}{\sim}\mathcal{N}(0,1)$ and $\theta_0=x_0$ in the scaled case is
\begin{equation}\label{eqn:ne1_eulerscale}
\theta_{k+1}=\theta_k-\frac{1}{2}\left(\frac{\theta_k-\mu_{p}}{\sigma_{p}^{2}}\right)\cdot\frac{h}{m}+\sqrt{\frac{h}{m}}\xi_{k},
\end{equation}
and the SGLD chain reads as
\begin{equation}\label{eqn:ne2_slgdscale}
Y_{k+1}=Y_k-\frac{1}{2}\left(\frac{Y_k-\mu_{sgld}}{\sigma_{p}^{2}}\right)\cdot\frac{h}{m}+\sqrt{\frac{h}{m}}\xi_{k},
\end{equation}
with $$\mu_{sgld}=\Big(\frac{\sigma_{y}^{2}}{\sigma_{\theta}^{2}}+1\Big)^{-1}\sum\limits_{i=1}^{s}\frac{1}{s} y_{\tau_{k}^i} $$
where $\tau_k=(\tau_{k}^1,\cdots,\tau_k^{s})$ denote a random subset of $[m]=\{1,\cdots,m\}$ 
generated by sampling without replacement from
$[m]$, independently for each $k$.

We fix $\sigma_\theta=0.2$ and $\sigma_y=0.15$, and take two different functions $f$ for experiments approximating $\int_{\mathbb{R}} f(x) \pi(dx)$,
$$f_1(x):=\sin(x)$$ and $$f_2(x):=\sin(x)+|x|+1.$$ The ground truth can be estimated via trapezoidal rule for approximating the following integration:
$$\mathbb{E}[f_i(X)]=\int_{\mathbb{R}}f_i(x)g_d(x,\mu_p,\sigma^2_p)\,\mathrm{d}x, \quad i\in \{1,2\},$$
where $g_d(x,\mu_p,\sigma^2_p)$ is the density function of $\mathcal{N}(\mu_{p},\sigma_{p}^{2})$.
\begin{remark}
	The time rescaling in \eqref{eq:Lan_OU_scaled} is essential for the
	uniform-in-$m$ framework considered in this paper. Indeed, since
	\begin{equation*}
	\sigma_p^2
	=
	m^{-1}
	\left(
	\frac{1}{\sigma_\theta^2}
	+
	\frac{1}{\sigma_y^2}
	\right)^{-1},
	\end{equation*}
	the mean-reversion coefficient of the unscaled dynamics
	\eqref{eq:Lan_OU} is
	\begin{equation*}
	\frac{1}{2\sigma_p^2}
	=
	\frac{m}{2}
	\left(
	\frac{1}{\sigma_\theta^2}
	+
	\frac{1}{\sigma_y^2}
	\right),
	\end{equation*}
	and therefore grows linearly with $m$. Consequently, the corresponding
	Lipschitz and dissipativity constants are not independent of the
	dataset size.
	
	For the rescaled dynamics \eqref{eq:Lan_OU_scaled}, however,
	\begin{equation*}
	\frac{1}{2m\sigma_p^2}
	=
	\frac{1}{2}
	\left(
	\frac{1}{\sigma_\theta^2}
	+
	\frac{1}{\sigma_y^2}
	\right),
	\end{equation*}
	so the Lipschitz and dissipativity constants remain independent of
	$m$, while the diffusion coefficient is $\beta=m^{-1/2}$.
	Furthermore, the invariant variance of
	\eqref{eq:Lan_OU_scaled} is
	\begin{equation*}
	\frac{m^{-1}}
	{2(2m\sigma_p^2)^{-1}}
	=
	\sigma_p^2,
	\end{equation*}
	and hence the rescaled dynamics have the same invariant distribution
	$\mathcal{N}(\mu_p,\sigma_p^2)$ as the original OU process.
	
	In particular, although $\sigma_p^2$ is of order $m^{-1}$, the
	rescaled dynamics do not become increasingly stiff as $m$ grows,
	because the factor $m^{-1}$ in the drift exactly compensates for the
	scaling of the posterior variance.
\end{remark}

\subsection{Cost estimate}
\leavevmode
In this subsection we describe the numerical protocol used to estimate the
computational cost of the Euler--Maruyama and SGLD estimators for a prescribed
accuracy level $\varepsilon>0$. The construction is directly motivated by the
bias--variance decomposition underlying the theoretical cost analysis in
Section~\ref{sec:discussion}.

For each method, the numerical cost is determined by three quantities: the
simulation horizon $T$, the time step $h$, and the number of independent Monte
Carlo realisations. The horizon $T$ is chosen so as to make the finite-time bias
negligible. The step size $h$ is then calibrated empirically through a
step-doubling test, in which the outputs computed with step sizes $h$ and $h/2$
are compared; this serves as a practical proxy for the discretisation bias, and
the time step is refined until the discrepancy falls below the required tolerance.
Finally, after the discretisation level has been fixed, the variance of a
single-path output is estimated, and the number of Monte Carlo samples is chosen
according to the standard Monte Carlo scaling
$$
N \asymp \varepsilon^{-2}\operatorname{Var}.
$$

The corresponding empirical computational cost is then defined by
$$
\mathrm{Cost}=N\times \frac{T}{h}\times c_{\mathrm{step}},
$$
where $c_{\mathrm{step}}=m$ for the Euler--Maruyama method and
$c_{\mathrm{step}}=s$ for SGLD. This reflects the fact that one step of the full
Euler scheme uses all $m$ data points, whereas one step of SGLD uses only a
mini-batch of size $s$. Hence the empirical protocol is fully consistent with
the cost model adopted in Section~\ref{section:general}.

Algorithm~\ref{alg:cost} summarises this procedure. It should be understood as a
practical implementation of the theoretical complexity-balancing principle,
rather than as an attempt to reproduce the exact constants appearing in the
analysis.

Though the theoretical results in this article exclude the non-Lipschitz case, in the numerical experiments we take both Lipschitz function and non-Lipschitz function. For each function, the numerical setting is the same and as described above.

\begin{algorithm}
	\caption{Empirical selection of the stepsize}\label{alg:stepsize}
	\begin{algorithmic}
		\Procedure{stepsize}{$Z$, $\varepsilon$, $T$, $M=10^3$}
		\Comment{$Z=\theta$ for EM and $Z=Y$ for SGLD; T, the simulation time}
		\State \textbf{Set} $h=4\varepsilon$ \Comment{$h$, the stepsize}
		\State Generate $M$ coupled pairs
		$\bigl(Z_h^{(j)}(T),Z_{h/2}^{(j)}(T)\bigr)$
		\State  \Comment{$Z^{(j)}_{h}(T)$, the $j$-th realisation of $Z$ with a stepsize $h$ evaluated at $T$}
		\While{$
			\dfrac{1}{M}\displaystyle\sum_{j=1}^{M}
			\left|Z_h^{(j)}(T)-Z_{h/2}^{(j)}(T)\right|^2
			>
			\dfrac{\varepsilon^2}{2}
			$}
		\State \textbf{Set} $h=h/2$
		\State Regenerate the $M$ coupled pairs
		$\bigl(Z_h^{(j)}(T),Z_{h/2}^{(j)}(T)\bigr)$
		\EndWhile
		\State \textbf{Output} $\widehat h=h$
		\EndProcedure
	\end{algorithmic}
\end{algorithm}

\begin{algorithm}
	\caption{Empirical estimation of the computational cost for EM and SGLD}
	\label{alg:cost}
	\begin{algorithmic}
		\Procedure{cost}{$n$, $\varepsilon$, $m$, $s$, ans, $M=10^2$}
		\State \textbf{Set} $T=\lceil3\log(\varepsilon^{-1})\rceil$
		
		\For{$i=1,\ldots,n$}
		\State Generate $m$ samples and calculate the posterior distribution
		
		\State \textbf{Set}
		$\widehat h_i^\theta=\Call{stepsize}{\theta,\varepsilon,T,M}$ in Algorithm \ref{alg:stepsize}.
		
		\State Estimate
		\[
		V(\theta)_i
		=
		\frac1M\sum_{j=1}^M
		\left(f\bigl(\theta_{\widehat h_i^\theta}^{(j)}(T)\bigr)\right)^2
		-
		\left(
		\frac1M\sum_{j=1}^M
		f\bigl(\theta_{\widehat h_i^\theta}^{(j)}(T)\bigr)
		\right)^2
		\]
		\State  \Comment{$\theta^{(j)}_{h}(T)$, the $j$-th realisation of full Euler $\theta$ with a stepsize $h$ evaluated at $T$}
		\State \textbf{Set}
		$\widehat N_i=
		\max\left\{1,\left\lceil2V(\theta)_i\varepsilon^{-2}\right\rceil\right\}$
		\Comment{$\hat{N}_i$, the number of realisations used for the $i$th experiment for full EM scheme $\theta$}
		\State Estimate
		\[
		f_i^\theta
		=
		\frac1{\widehat N_i}
		\sum_{j=1}^{\widehat N_i}
		f\bigl(\theta_{\widehat h_i^\theta}^{(j)}(T)\bigr),
		\qquad
		\mathrm{cost}_i^\theta
		=
		\frac{\widehat N_iTm}{\widehat h_i^\theta}
		\]
		\State \Comment{$\mathrm{cost}_i^\theta$, the cost of the $i$-th run in the case of full Euler $\theta$}
		\State \textbf{Set}
		$\widehat h_i^Y=\Call{stepsize}{Y,\varepsilon,T,M}$ in Algorithm \ref{alg:stepsize}.

		\State Estimate
		\[
		V(Y)_i
		=
		\frac1M\sum_{j=1}^M
		\left(f\bigl(Y_{\widehat h_i^Y}^{(j)}(T)\bigr)\right)^2
		-
		\left(
		\frac1M\sum_{j=1}^M
		f\bigl(Y_{\widehat h_i^Y}^{(j)}(T)\bigr)
		\right)^2
		\]
		\State  \Comment{$Y^{(j)}_{h}(T)$, the $j$-th realisation of SGLD method $Y$ with a stepsize $h$ evaluated at $T$}
		\State \textbf{Set}
		$\widehat n_i=
		\max\left\{1,\left\lceil2V(Y)_i\varepsilon^{-2}\right\rceil\right\}$
		\Comment{$\hat{n}_i$, the number of realisations used for the $i$th experiment for SGLD  method $Y$}
		\State Estimate
		\[
		f_i^Y
		=
		\frac1{\widehat n_i}
		\sum_{j=1}^{\widehat n_i}
		f\bigl(Y_{\widehat h_i^Y}^{(j)}(T)\bigr),
		\qquad
		\mathrm{cost}_i^Y
		=
		\frac{\widehat n_iTs}{\widehat h_i^Y}
		\]
		\State  \Comment{$\mathrm{cost}_i^Y$, the cost of the $i$-th run in the case of SGLD method $Y$}
		\EndFor
		
		\State \textbf{Output}
		\[
		\sqrt{\frac1n\sum_{i=1}^n(f_i^\theta-\mathrm{ans})^2},
		\qquad
		\sqrt{\frac1n\sum_{i=1}^n(f_i^Y-\mathrm{ans})^2},
		\]
		and
		\[
		\frac1n\sum_{i=1}^n\mathrm{cost}_i^\theta,
		\qquad
		\frac1n\sum_{i=1}^n\mathrm{cost}_i^Y.
		\]
		\EndProcedure
	\end{algorithmic}
\end{algorithm}

\subsubsection{Non-linear functional: Lipschitz case}
\leavevmode
We adopt the same setting as in Section~\ref{sec:validation} and apply
Algorithm~\ref{alg:cost} to the two Lipschitz observables $f_1$ and $f_2$.
For the fixed-batch experiments, we set $s=50$ and generate $2000$
Sobol-sampled parameter configurations over
\[
m\in[10^2,10^3],
\qquad
\varepsilon\in[10^{-3},10^{-2}].
\]
For each configuration, the RMSE and computational cost are estimated using
$50$ independent Monte Carlo paths. The complete experiment is repeated five
times, and the reported quantities are averaged over these repetitions.

The upper and middle panels of Fig.~\ref{fig:CostRegion2}(b), corresponding to
$f_1$ and $f_2$, respectively, show that SGLD is cheaper over most of the
fixed-batch test window. The marker colour represents the signed empirical cost
advantage
\[
\log_{10}\left(
\frac{\operatorname{Cost}(\mathrm{EM})}
{\operatorname{Cost}(\mathrm{SGLD})}
\right),
\]
so that positive values indicate an SGLD advantage and negative values indicate
an EM advantage. The predominance of positive values at moderate and large
$m$ is broadly consistent with the theoretical phase diagram in
Fig.~\ref{fig:CostRegion2}(a). In the regime $s\ll m$, SGLD replaces the
full-data drift evaluation of cost order $m$ by a mini-batch evaluation of cost
order $s$, and its computational advantage therefore increases with the dataset
size.

However, EM remains competitive in the lower-$m$ region and near the theoretical
transition boundary. Some configurations in these regions yield a negative
signed cost advantage, indicating that EM is empirically cheaper.

Figure~\ref{fig:EmpiricalRegion_linear} examines this transition more directly
by allowing the mini-batch size $s$ to vary jointly with the dataset size $m$.
The discussion in this section concerns the top two rows, corresponding to the
Lipschitz observables $f_1$ and $f_2$, at the target accuracies
$\varepsilon=10^{-3}$ and $\varepsilon=5\times10^{-4}$. For both observables,
SGLD is generally favoured when $s/m$ is small. As $s$ approaches $m$, however,
the computational saving from mini-batching diminishes and EM becomes
increasingly competitive. In the limiting case $s=m$, the mini-batch drift
coincides with the full drift, so SGLD no longer has a per-step cost advantage.
These results therefore complement the fixed-$s$ phase diagrams by displaying
the transition from the small-batch regime to the near-full-batch regime.

The numerical costs reported in Table~\ref{tab:lip1} provide further support for
this interpretation. For both $f_1$ and $f_2$, and for
$s\in\{20,50,100\}$, SGLD becomes substantially cheaper once $m$ is
sufficiently large. The difference is particularly pronounced for
$m=10^5$ and $m=10^7$, where the SGLD cost is lower by several orders of
magnitude. By contrast, at smaller $m$, the difference between the methods is
less pronounced and EM can remain competitive. Overall, the experiments show
that the preferred method depends jointly on $m$ and $s$: SGLD is most
advantageous in the large-data, small-batch regime, whereas EM can be preferable
when the dataset is small or the batch size is close to the full dataset size.

\begin{figure}[H]
	\begin{subfigure}{0.48\textwidth}
		\centering
		\includegraphics[width=\linewidth]{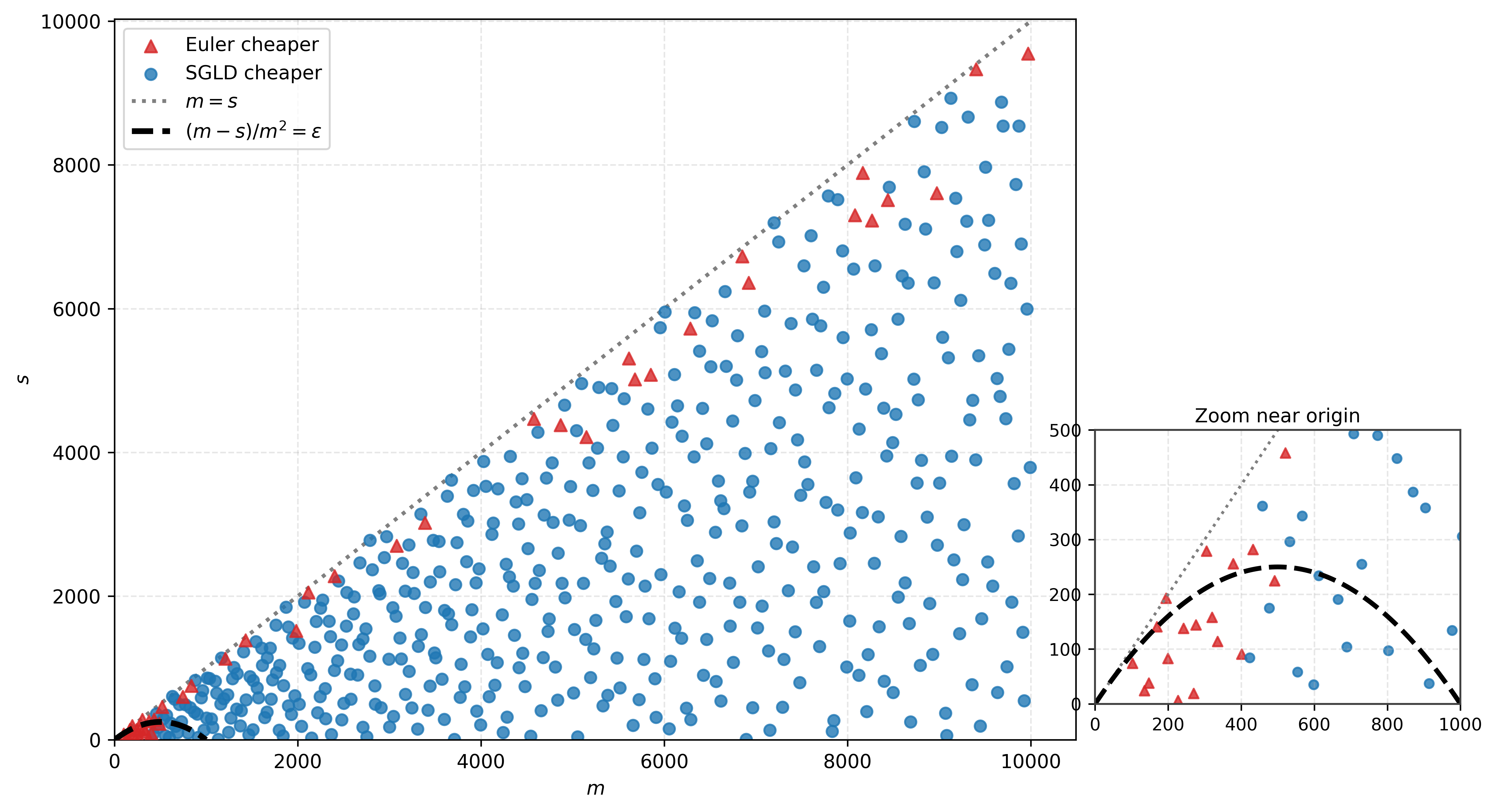}
		\caption{$\varepsilon = 0.001$,
			$f_1$}
	\end{subfigure}
	\hfill
	\begin{subfigure}{0.48\textwidth}
		\centering
		\includegraphics[width=\linewidth]{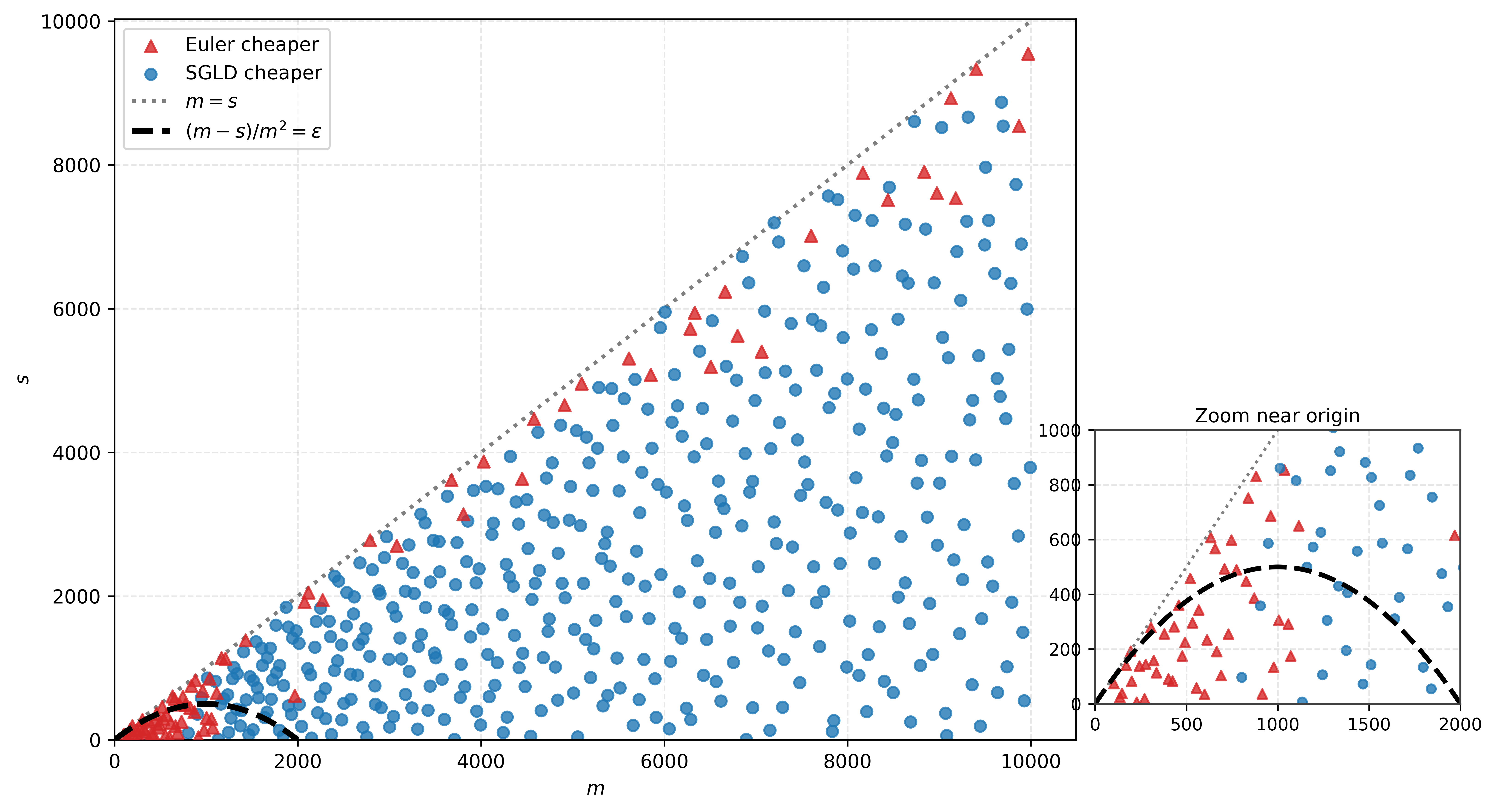}
		\caption{$\varepsilon = 0.0005$,
			$f_1$}
	\end{subfigure}
	\hfill
	\begin{subfigure}{0.48\textwidth}
		\centering
		\includegraphics[width=\linewidth]{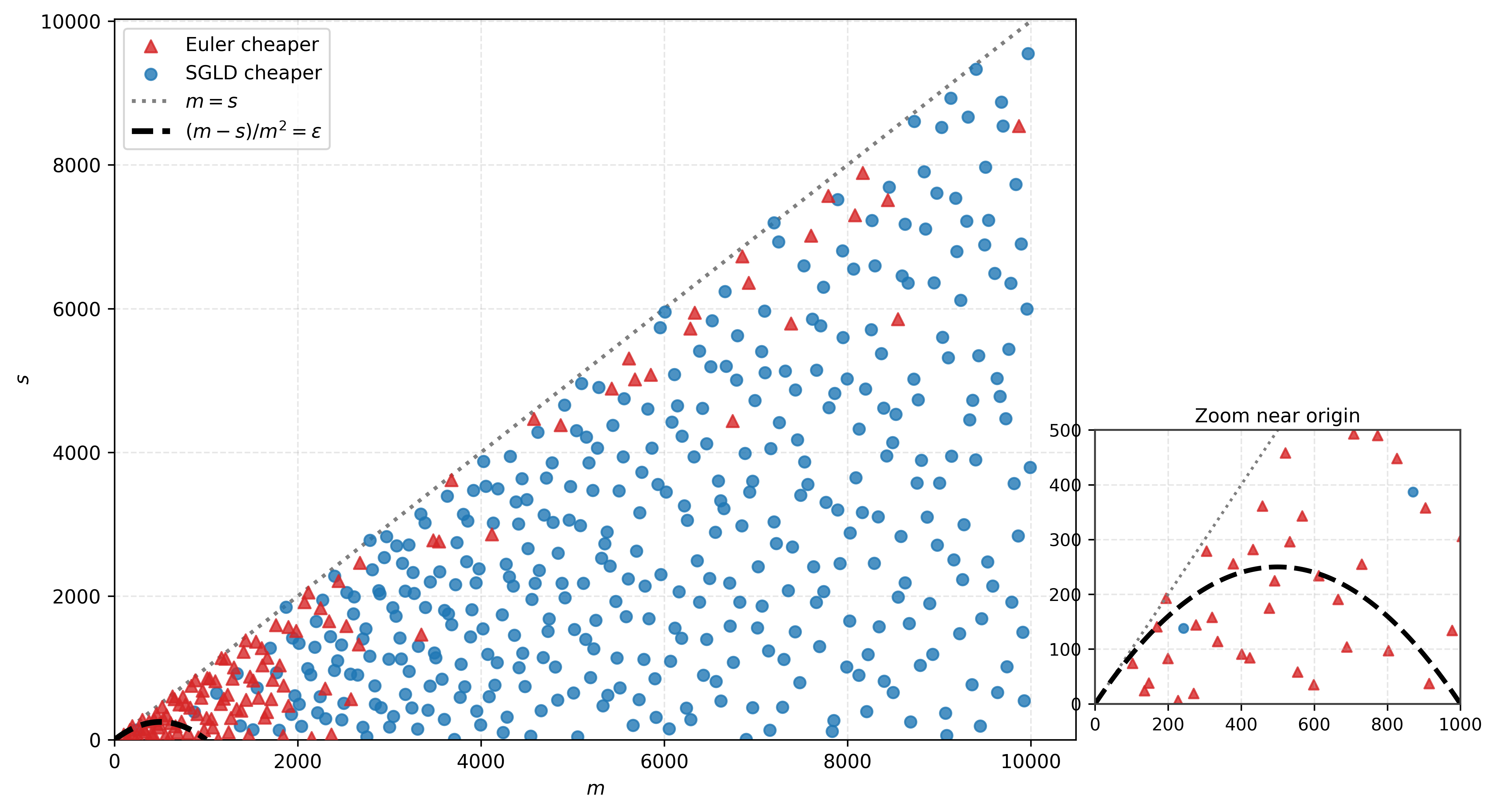}
		\caption{$\varepsilon = 0.001$,
			$f_2$}
	\end{subfigure}
	\hfill
	\begin{subfigure}{0.48\textwidth}
		\centering
		\includegraphics[width=\linewidth]{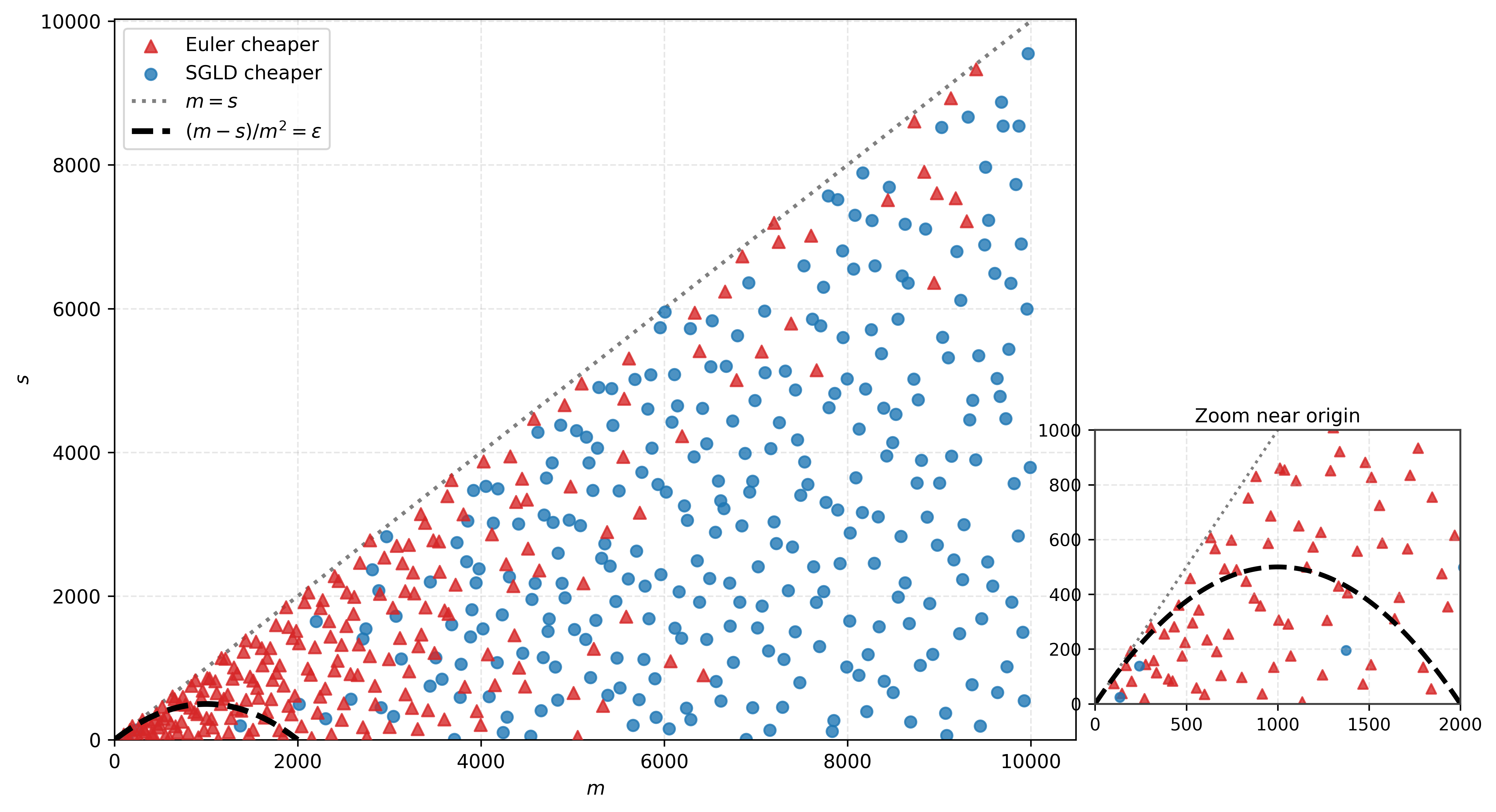}
		\caption{$\varepsilon = 0.0005$,
			$f_2$}
	\end{subfigure}
	\hfill
	
	\begin{subfigure}{0.48\textwidth}
		\centering
		\includegraphics[width=\linewidth]{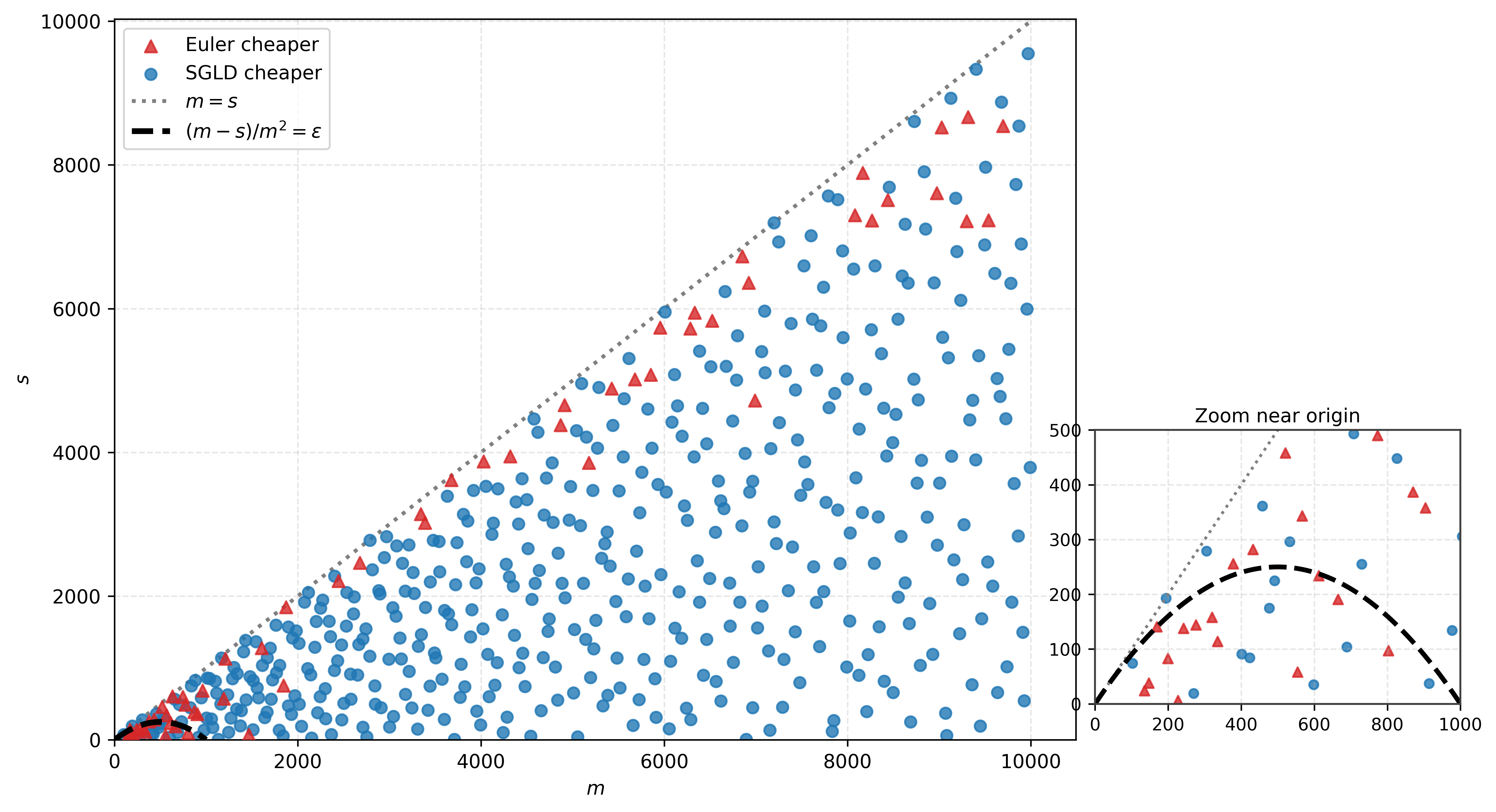}
		\caption{$\varepsilon = 0.001$,
			$f_3$ (Non-Lip)}
	\end{subfigure}
	\hfill
	\begin{subfigure}{0.48\textwidth}
		\centering
		\includegraphics[width=\linewidth]{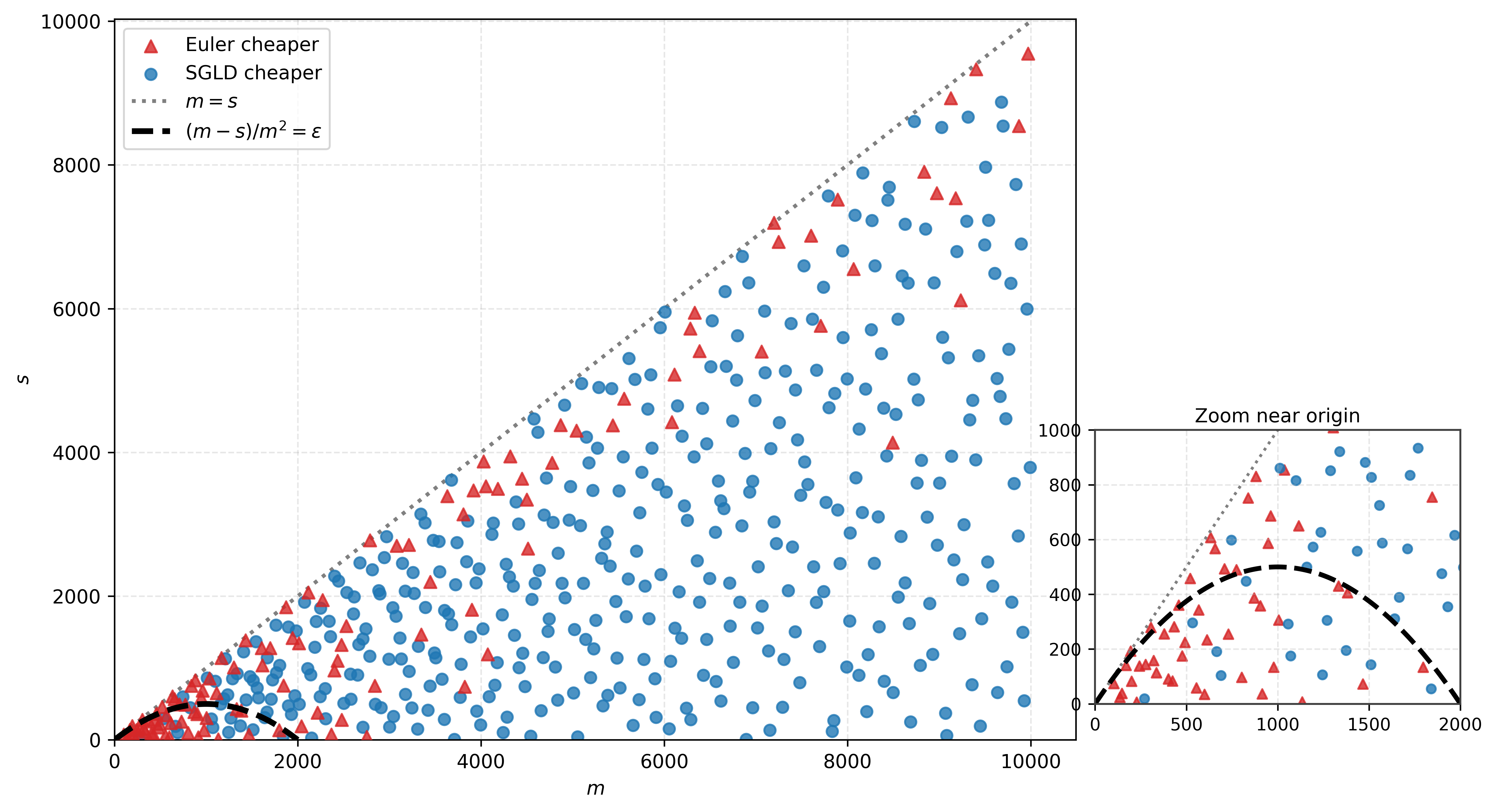}
		\caption{$\varepsilon = 0.0005$, $f_3$ (Non-Lip)}
	\end{subfigure}
	\hfill
	\caption{Empirical winner maps comparing the computational costs of
		Euler--Maruyama (EM) and SGLD for
		$f_1$, $f_2$, and $f_3$.
		The rows correspond to $f_1$, $f_2$, and $f_3$, while the left and right
		columns correspond to $\varepsilon=10^{-3}$ and
		$\varepsilon=5\times10^{-4}$, respectively.
		Each panel contains $500$ Sobol-sampled parameter configurations
		$(m,s)$. For each configuration, the costs were estimated using $25$
		independent Monte Carlo paths, and the experiment was repeated $5$ times;
		the displayed winner was determined from the costs averaged over these
		repetitions. Red triangles indicate that EM is cheaper, whereas blue
		circles indicate that SGLD is cheaper; and the inset enlarges
		the transition region near the origin. }

	\label{fig:EmpiricalRegion_linear}
\end{figure}

\begin{table}[t]
	\centering
	
	\begin{minipage}{0.92\textwidth}

		\captionsetup{
			width=\linewidth,
			justification=justified,
			singlelinecheck=false
		}
		\caption{Computational costs of estimating $\pi(f_1)$, $\pi(f_2)$, and
			$\pi(f_3)$ at the target mean-squared error
			$\varepsilon^2=10^{-6}$, for dataset sizes
			$m\in\{10^2,10^3,10^5,10^7\}$ and mini-batch sizes
			$s\in\{20,50,100\}$. For each parameter configuration, the cost was
			estimated using $100$ independent Monte Carlo paths, and the complete
			experiment was repeated $10$ times. The reported values are the averages
			over these $10$ independent repetitions. In each ordered pair
			$(\mathrm{cost}_1,\mathrm{cost}_2)$, $\mathrm{cost}_1$ and
			$\mathrm{cost}_2$ denote the computational costs of EM and SGLD,
			respectively. The lower cost in each pair is shown in bold.}
		\vspace{-8pt}
		\label{tab:lip1}
		
		\centering
		\renewcommand{\arraystretch}{1.2}
		\setlength{\tabcolsep}{6pt}
		
		\resizebox{\linewidth}{!}{%
			\begin{tabular}{|c|c|c|c|c|}
				\hline
				\multirow{2}{*}{value of $m$}
				& \multirow{2}{*}{choice of $f$}
				& \multicolumn{3}{c|}{value of $s$}
				\\ \cline{3-5}
				& & $s=20$ & $s=50$ & $s=100$
				\\
				\hline
				\hline
				
				\multirow{3}{*}{$m=10^2$}
				& $f_1$
				& $(\mathbf{5.7e{+}07},\ 2.4e{+}08)$
				& $(\mathbf{5.7e{+}07},\ 2.9e{+}08)$
				& $-$
				\\ \cline{2-5}
				
				& $f_2$
				& $(\mathbf{1.1e{+}08},\ 1.4e{+}09)$
				& $(\mathbf{1.1e{+}08},\ 1.4e{+}09)$
				& $-$
				\\ \cline{2-5}
				
				& $f_3$ (Non-Lip)
				& $(\mathbf{5.5e{+}07},\ 4.7e{+}08)$
				& $(\mathbf{5.5e{+}07},\ 3.2e{+}08)$
				& $-$
				\\
				\hline
				\hline
				
				\multirow{3}{*}{$m=10^3$}
				& $f_1$
				& $(5.6e{+}07,\ \mathbf{3.6e{+}07})$
				& $(5.6e{+}07,\ \mathbf{3.3e{+}07})$
				& $(5.6e{+}07,\ \mathbf{2.4e{+}07})$
				\\ \cline{2-5}
				
				& $f_2$
				& $(\mathbf{1.7e{+}08},\ 4.7e{+}08)$
				& $(\mathbf{1.7e{+}08},\ 4.4e{+}08)$
				& $(\mathbf{1.7e{+}08},\ 4.5e{+}08)$
				\\ \cline{2-5}
				
				& $f_3$ (Non-Lip)
				& $(\mathbf{4.2e{+}07},\ 4.7e{+}07)$
				& $(4.2e{+}07,\ \mathbf{3.3e{+}07})$
				& $(4.2e{+}07,\ \mathbf{3.43e{+}07})$
				\\
				\hline
				\hline
				
				\multirow{3}{*}{$m=10^5$}
				& $f_1$
				& $(9.0e{+}07,\ \mathbf{7.0e{+}05})$
				& $(9.0e{+}07,\ \mathbf{6.8e{+}05})$
				& $(9.0e{+}07,\ \mathbf{7.3e{+}05})$
				\\ \cline{2-5}
				
				& $f_2$
				& $(1.9e{+}08,\ \mathbf{4.9e{+}06})$
				& $(1.9e{+}08,\ \mathbf{4.6e{+}06})$
				& $(1.9e{+}08,\ \mathbf{3.5e{+}06})$
				\\ \cline{2-5}
				
				& $f_3$ (Non-Lip)
				& $(1.2e{+}08,\ \mathbf{1.6e{+}06})$
				& $(1.2e{+}08,\ \mathbf{1.9e{+}06})$
				& $(1.2e{+}08,\ \mathbf{1.9e{+}06})$
				\\
				\hline
				\hline
				
				\multirow{3}{*}{$m=10^7$}
				& $f_1$
				& $(6.9e{+}09,\ \mathbf{4.0e{+}05})$
				& $(6.9e{+}09,\ \mathbf{4.4e{+}05})$
				& $(6.9e{+}09,\ \mathbf{4.1e{+}05})$
				\\ \cline{2-5}
				
				& $f_2$
				& $(6.9e{+}09,\ \mathbf{1.9e{+}05})$
				& $(6.9e{+}09,\ \mathbf{1.4e{+}05})$
				& $(6.9e{+}09,\ \mathbf{2.8e{+}05})$
				\\ \cline{2-5}
				
				& $f_3$ (Non-Lip)
				& $(6.9e{+}09,\ \mathbf{1.5e{+}05})$
				& $(6.9e{+}09,\ \mathbf{1.8e{+}05})$
				& $(6.9e{+}09,\ \mathbf{1.4e{+}05})$
				\\
				\hline
				
			\end{tabular}%
		}
		
	\end{minipage}
\end{table}

\subsubsection{Non-linear functional: non-Lipschitz case}
\leavevmode
To check the performance of both methods in a non-Lipschitz case, which is not covered by our theorical results, we adopt the same setting as in Section \ref{sec:validation}, but implement Algorithm \ref{alg:cost} on a non-Lipschitz function $f_3=x^2$.

The corresponding results in the lower panel of
Fig.~\ref{fig:CostRegion2}(b), the bottom row of
Fig.~\ref{fig:EmpiricalRegion_linear}, and Table~\ref{tab:lip1} show a
qualitative pattern similar to that obtained for $f_1$ and $f_2$. SGLD remains
substantially cheaper in the large-$m$, small-$s/m$ regime, while EM can be
competitive for smaller $m$ or when $s$ approaches $m$. This provides numerical
evidence that the practical advantage of SGLD can persist beyond the Lipschitz
setting. However, a full theoretical analysis for general
non-Lipschitz observables falls beyond the scope of the present paper.

\section{Theoretical analysis and proofs}
\label{sec:discussion}

In this section, we provide a rigorous analysis of the computational cost of the EM and SGLD methods, beginning with some auxiliary results on the corresponding SDE.

\subsection{Properties of exact solution}\label{section:solution}
This section is devoted to the analysis of the solution $X$ to the SDE \eqref{generalSDE}. All the proofs of this section are postponed to Appendix \ref{sec:proofs}. 

We begin by recalling a well-known convergence result for the SDE \eqref{generalSDE}, where we denote by $(X_t^x)_{t\geq 0}$ the solution of
SDE~\eqref{generalSDE} with initial condition $X_0^x=x$.

\begin{proposition}\label{est:time}
	Let Assumptions~\ref{ass:sol1} and~\ref{ass:sol2} hold. Then
	SDE~\eqref{generalSDE} admits an invariant probability measure $\pi$
	satisfying
	\begin{equation}\label{eq:invariant_second_moment}
	\int_{\mathbb{R}^d}|x|^2\,\pi(\mathrm{d}x)<\infty.
	\end{equation}
	Moreover, for every $x\in\mathbb{R}^d$ and every $t\geq 0$,
	\begin{equation}\label{eq:pointwise_convergence_to_invariant}
	\sup_{f:\,\operatorname{Lip}(f)\leq 1}
	\left|
	\pi(f)-P_tf(x)
	\right|
	\leq
	\exp(-Kt)
	\int_{\mathbb{R}^d}|x-y|\,\pi(\mathrm{d}y),
	\end{equation}
	where
	\begin{equation*}
	\pi(f):=\int_{\mathbb{R}^d}f(x)\,\pi(\mathrm{d}x),
	\end{equation*}
	and $(P_t)_{t\geq 0}$ is the transition semigroup associated with
	$(X_t)_{t\geq 0}$, defined by
	\begin{equation*}
	P_tf(x):=\mathbb{E}_W[f(X_t^x)].
	\end{equation*}
	
	Consequently, for any initial condition
	$X_0\in L^2(\Omega_W)$ and every $t\geq 0$,
	\begin{equation}\label{eq:random_initial_convergence_to_invariant}
	\sup_{f:\,\operatorname{Lip}(f)\leq 1}
	\left|
	\pi(f)-\mathbb{E}_W[f(X_t)]
	\right|
	\leq
	\exp(-Kt)
	\int_{\mathbb{R}^d}
	\mathbb{E}_W[|X_0-y|]
	\,\pi(\mathrm{d}y).
	\end{equation}
\end{proposition}
The proof of Proposition \ref{est:time} is deferred to Appendix \ref{sec:proofs}.

\begin{lemma}[Bound on the second moment] \label{lm:2}
	Consider SDE \eqref{generalSDE} with the drift satisfying Assumption \ref{ass:sol1} and Assumption \ref{ass:sol2}. Then for all $t \geq 0$, the solution $X_t$ to \eqref{generalSDE} has the following moment bound
	\begin{equation}
	\E_W [\left|X_t\right|^2 ]\le e^{-K t}\E_W[ \left|X_0\right|^2 ] +\left(\frac{| a(0)|^2}{K^2} +\frac{d\beta^2}{K}\right)\left(1-e^{-K t}\right).\label{eq:ito1}    
	\end{equation}
\end{lemma}

Define by $C^{(1)}$ and $C^{(2)}$ the positive constants such that the first and the second moments of $X_t$ are bounded by $C^{(1)}$ and $C^{(2)}$, respectively, for all $t \geq 0$, i.e.,
\begin{equation}
\sup_{t\geq 0} \E_W [\left|X_t\right|^i ] \leq C^{(i)} \text{ for } i \in \{1,2\},
\end{equation}
Such uniform bounds can be obtained due to Lemma \ref{lm:2} (for the second moment) and the Jensen inequality (for the first moment), and they can be made independent  of $m$ by bounding $\beta^2 = \frac{1}{m} \leq 1$.

\begin{lemma}\label{lem:RK2}
	Let Assumption \ref{ass:sol1} and Assumption \ref{ass:sol2} hold for SDE \eqref{generalSDE} with a solution $(X_t)_{t\geq 0}$. Consider the time grid $\Tau_h$ given by \eqref{eqn:timegrid}, i.e., $t_k = kh$ for a fixed $h>0$ and define for each $k$\begin{equation} \label{eq:R}
	\mathcal{R}_k := \int\limits_{t_k}^{t_{k+1}} \left( a(X_s)  - a(X_{t_k}) \right)  \,\mathrm{d}s.  
	\end{equation}
	Then the following bound holds:
	\begin{equation}\label{e:ERbound2} 
	\E_W[|\mathcal{R}_k|^2]\le  4L^4\mathcal{C}_2h^4 + d\beta^2 L^2h^3,
	\end{equation}
	where
	$$\mathcal{C}_2 = \left(\frac{|a(0)|^2}{2L^2}+\E_W[ \left|X_0\right|^2 ] +\left(1+\frac{1}{2L^2}\right)\left(\frac{| a(0)|^2}{K^2} +\frac{d\beta^2}{K}\right) \right). $$
\end{lemma}

We will also need an improved bound on the residual term considered in  \eqref{eq:R}, which can be obtained under an additional Assumption \ref{ass:additional}.

\begin{lemma}\label{lem:RK1}
	Let Assumption \ref{ass:sol1}, Assumption \ref{ass:sol2} and Assumption \ref{ass:additional} hold for SDE \eqref{generalSDE} with a solution $(X_t)_{t\geq 0}$. Recall the time grid $\Tau_h$ given by \eqref{eqn:timegrid}, i.e., $t_k = kh$, and the term $\mathcal{R}_k$ defined in \eqref{eq:R} from Lemma \ref{lem:RK2}. 
	Then for any $h\in (0,1)$, the following bound holds:
	\begin{equation}\label{e:ERbound} \left|\E^W_k[\mathcal{R}_k] \right|^2 \le C_R h^4 \,,
	\end{equation} 
	where, $\E^W_k[\cdot] := \E_W[\cdot | \mathcal{F}^W_{t_k}]$,
	and
	$$C_R = \frac{1}{4}\sum_{i=1}^{d}\left( C_{a_i^{(1)}}(L C^{(1)} + |a(0)|) + C_{a_i^{(2)}}(L^2 C^{(2)} + |a(0)|^2 + \frac{1}{2}d\beta^2) \right)^2 \,.$$
\end{lemma}

\subsection{Analysis of Euler-Maruyama method}\label{section:SMC}

In this section we will derive time-uniform bounds for the Euler-Maruyama (EM) solution $\theta$ based on the discretisation \eqref{discretisation}, including the second moment distance $ \sup_k \E_W[| X_{t_k} - \theta_k |^2 ]$ and variance $\sup_k \V_W(f(\theta_k))$.

\begin{theorem}
	\label{bias:determ}
	Let Assumption \ref{ass:sol1}, Assumption \ref{ass:sol2} and Assumption \ref{ass:additional} hold for SDE \eqref{generalSDE} with a solution $(X_t)_{t\geq 0}$. Recall the constant $\mathcal{C}_2$ and $C_R$ defined in Lemma \ref{lem:RK2} and \ref{lem:RK1}. Consider the solution $(\theta_k)_{k\in \mathbb{N}\cup \{0\}}$ the EM discretization \eqref{discretisation} at stepsize $h$ with $\theta_0 = X_0$. Provided that the stepsize $h$ satisfies $$h\leq \min\Big(\frac{1}{K}, \frac{K}{2L^2}\Big),$$ 
	then, for any $k \geq 1$ and $t_k=kh$,
	\begin{align*}
	\E_W[| X_{t_k} - \theta_k |^2 ]&\le \left(4L^4\mathcal{C}_2h^3 + d\beta^2 L^2h^2 + K^{-1}C_Rh^2\right)\cdot\frac{1-(1+2L^2h^2-Kh)^{k}}{K-2L^2h}\\
	&\le\frac{4L^4\mathcal{C}_2h^3 + d\beta^2 L^2h^2 + K^{-1}C_Rh^2}{K-2L^2h}\\
	&\approx\frac{K^{-1}C_Rh^2}{K-2L^2h}.
	\end{align*}
\end{theorem}
\begin{proof}
	Set the approximation error term $e_k:=X_{t_k} - \theta_k$ and recall the definition of $\mathcal{R}_k$ defined in \eqref{eq:R}. Note that for any $k>0$,
	\begin{align}\label{eqn:decompo}
	\begin{split}
	e_{k+1} = \, e_k +  \int\limits_{t_k}^{t_{k+1}} \left( a(X_s)  - a(\theta_{k}) \right) \,\mathrm{d}s=e_k +  (a(X_{t_k}) - a(\theta_k) )h + \mathcal{R}_k.
	\end{split}
	\end{align}
	Squaring both sides of \eqref{eqn:decompo} yields
	\begin{align*}
	| e_{k+1} |^2  =\,  &\left|e_k\right|^2   + | a(X_{t_k}) - a(\theta_k)|^2 h^2  +  |\mathcal{R}_k|^2\\
	&\, + 2 \langle e_k ,   a(X_{t_k}) - a(\theta_k) \rangle h + 2 \langle a(X_{t_k}) - a(\theta_k),\mathcal{R}_k \rangle h +  2 \langle e_k, \mathcal{R}_k \rangle.  
	\end{align*}
	Recall that $\E^W_k[\cdot] := \E_W[\cdot | \mathcal{F}^W_{t_k}]$. We have 
	\begin{align}\label{eqn:errorterms}
	\begin{split}
	\E^W_k[| e_{k+1} |^2]  =&  \left|e_k\right|^2   + \E^W_k[| a(X_{t_k}) - a(\theta_k)|^2] h^2  +  \E^W_k[|\mathcal{R}_k|^2] \\
	&\,+ 2h\E^W_k \langle e_k ,   a(X_{t_k}) - a(\theta_k) \rangle  + 2 h\E^W_k\langle a(X_{t_k}) - a(\theta_k),\mathcal{R}_k \rangle  \\
	&\,+  2 \langle e_k, \E^W_k[\mathcal{R}_k] \rangle. 
	\end{split}
	\end{align}
	Because of the conditions \eqref{eq:a1} and \eqref{eq:a2} on $a$,  we have bounds on some of the terms above
	\begin{equation*}
	\E_k^W[| a(X_{t_k}) - a(\theta_k)|^2] h^2 \le L^2h^2\left|e_k\right|^2,
	\end{equation*}
	\begin{equation*}
	2 h\E_k^W\langle e_k ,   a(X_{t_k}) - a(\theta_k) \rangle \le - 2Kh\left|e_k\right|^2,
	\end{equation*}
	and 
	\begin{align*}
	2 h\E^W_k\langle a(X_{t_k}) - a(\theta_k),\mathcal{R}_k \rangle &\le \E^W_k[| a(X_{t_k}) - a(\theta_k)|^2] h^2 + \E_k^W[|\mathcal{R}_k|^2]\\
	&\le L^2h^2\left|e_k\right|^2+ \E_k^W[|\mathcal{R}_k|^2]. 
	\end{align*}
	Besides, via Young's inequality for any $\eta >0$ we have that
	\begin{equation*}
	2 \langle e_k, \E^W_k[\mathcal{R}_k]  \rangle \le \left|e_k\right|^2\cdot \eta h+ |\E^W_k[\mathcal{R}_k]|^2\cdot \eta^{-1}h^{-1}.
	\end{equation*}
	Substituting all the bounds above into \eqref{eqn:errorterms} and taking $\eta=K$ give 
	\begin{gather}\label{eq:estek}
	\E^W_k[| e_{k+1} |^2] \le  \left(1+2L^2h^2-Kh\right)|e_k|^2 +2\E^W_k[\left|\mathcal{R}_k\right|^2] + \eta^{-1}h^{-1}\left|\E^W_k[\mathcal{R}_k]\right|^2.
	\end{gather}
	The stepsize is chosen such that
	$2L^2h^2-K h <0$, which is equivalent to $h<K/(2L^2)$. Also note that 
	$$1+2L^2h^2-Kh>1-Kh>0 $$
	as long as we set $h<1/K$. Similar arguments for setting the range of $h$ will be used in almost all the main proofs in this manuscript.
	
	Recall the estimates of $\E_W[|\mathcal{R}_k|^2]$ and $|\E^W_k[\mathcal{R}_k]|^2$ from Lemma \ref{lem:RK1} and Lemma \ref{lem:RK2}. Combining everything together we have 
	\begin{gather*}
	\E_W[| e_{k+1} |^2] \le \left(1+2L^2h^2-Kh\right)\cdot\E_W[|e_k|^2] +4L^4\mathcal{C}_2h^4 + d\beta^2 L^2h^3 + K^{-1}C_Rh^3,
	\end{gather*}
	or having in mind that $e_0=0$, for all $k \geq 1$,
	\begin{align*}
	\E_W[| e_{k} |^2] &\le \left(4L^4\mathcal{C}_2h^3 + d\beta^2 L^2h^2 + K^{-1}C_Rh^2\right)\cdot\frac{1-(1+2L^2h^2-Kh)^{k}}{K-2L^2h}\\
	&\le\frac{4L^4\mathcal{C}_2h^3 + d\beta^2 L^2h^2 + K^{-1}C_Rh^2}{K-2L^2h},
	\end{align*}
	which finishes the proof.
\end{proof}

Let us finish this section by proving a simple bound on the variance of Lipschitz functionals of discretisations \eqref{discretisation}.

\begin{theorem}\label{var:euler:est}
	Let $f$ be a Lipschitz functional with the Lipschitz coefficient $L_f$. Provided that the stepsize $h$ satisfies $$h\leq \min\Big(\frac{1}{2K}, \frac{2K}{L^2}\Big),$$ then for any $k \geq 1$, then for a process $\theta_k$, given by 
	$$\theta_{k+1} = \theta_k + ha(\theta_k) + \beta \sqrt{h} Z_{k+1},$$
	we have
	
	$$\V_W( f(\theta_k)) \le \beta^2\cdot d\cdot L_f^2\cdot \frac{1-(1-2hK+L^2h^2)^{k+1}}{2K-L^2h}\le\frac{\beta^2\cdot d\cdot L_f^2}{2K-L^2h}.$$

\end{theorem}
\begin{proof}
	We will use the following general observation. Consider two independent random variables $X$ and $Y$ with the same distribution and a function $f$ as specified above. Then
	\begin{equation}\label{eqn:varianceequiv}
	\V_W( f(X)) = \frac{1}{2}\E_W [\left( f(X) - f(Y) \right)^2]  \le L_f^2 \E_W [\left| X-Y \right|^2 ]\,.
	\end{equation}
	We now consider two independent Euler discretisations $\theta_{k}$ and $\bar \theta_{k}$ and denote  $\eta_k:=\theta_k-\bar\theta_k$ along with $\E^W_k[\cdot] = \E_W[\cdot | \mathcal{F}_{t_k}]$.
	Then we have
	\begin{align}
	\E_k^W [|\eta_{k+1}|^2] &= |\eta_k|^2 + 2h\E_k^W\langle\eta_k,a(\theta_k)-a(\bar\theta_k)\rangle+\E_k^W[|a(\theta_k)-a(\bar\theta_k)|^2]h^2 + \beta^2dh\notag \\
	&\le |\eta_k|^2\cdot(1-2Kh+L^2h^2)+\beta^2dh,  \label{variance:bound}
	\end{align}
	which gives us the desired result, as $\eta_0=\theta_{0}-\bar \theta_{0}=0$.
\end{proof}

\subsection{Analysis of SGLD}\label{sec:slgd}
In this section, we will derive time-uniform bounds for the SGLD chain $Y$ based on the discretisation \eqref{discretisationSGLD}, including the second moment distance $ \sup_k \E[| \theta_k - Y_k |^2 ]$ and variance $\sup_k \V(f(Y_k))$, where $\theta$ is the numerical solution from the Euler discretisation \eqref{discretisation}.

We begin this section by deriving a second-moment identity for the subsampling-without-replacement estimator; as an immediate corollary, we obtain a bound on the variance of the corresponding estimation error. Both proofs are deferred to Appendix \ref{sec:numerical}.

\begin{lemma}\label{lem:estimatorProduct}
	For a sequence $(a_i)_{i=1}^{m} \subset \mathbb{R}^d$,
	consider
	$$\Lambda = \sum_{i=1}^{m}\frac{1}{m}a_{i}$$
	and its estimator
	$$\Lambda_{1} = \frac{1}{s}\sum_{i=1}^{m}a_{i}U_{i}$$
	based on subsampling without replacement, i.e, with $U_i$ given as in \eqref{eqn:subsapling_pb1} and \eqref{eqn:subsapling_pb2}.  Then
	$$ \quad$$
	$$\mathbb{E}_{U}\left[\Lambda_{1}^{T}\Lambda_{1}\right]=\frac{1}{s}\left(\frac{m-s}{m-1}\right) \sum_{i=1}^{m}\frac{1}{m}a_{i}^{T}a_{i} + \frac{m(s-1)}{s(m-1)}\Lambda^{T}\Lambda.$$
\end{lemma}

\begin{corollary} \label{cor:estimatorDiff}
	Under the same assumptions as Lemma \ref{lem:estimatorProduct}:
	$$\mathbb{E}_{U}[|\Lambda-\hat{\Lambda}_{1}|^{2}] \le \frac{m-s}{s(m-1)}\sum_{i=1}^{m}\frac{1}{m}a_{i}^{T}a_{i}.$$
\end{corollary}

We next establish a second-moment estimate for the SGLD chain $(Y_k)$ under a suitable condition on the stepsize $h$. In particular, this estimate implies a time-uniform bound on $E[|Y_k|^2]$, which will be used repeatedly in the subsequent error analysis.

\begin{lemma}\label{lemma:momentSLGD}
	Let Assumptions \ref{ass:sol1} and \ref{ass:sol2} hold. Consider the SGLD chain $(Y_k)$ defined by \eqref{discretisationSGLD}. Then, for sufficiently small $h$ satisfying
	\begin{equation}\label{eqn:hchoice}
	h\leq \min\Big(\frac{2}{K}, \frac{K}{4L^2}\Big),
	\end{equation}
	we have
	\begin{align*}
	\E [|Y_{k+1}|^2] \le \, &\E [|Y_0|^2] \left(  1 - \frac{Kh}{2} +2L^2 h^2 \right)^{k+1} \\
	&\, + \frac{2\max_{i=1,\ldots,m}|b(0,\xi_i)|^2 h + d \beta^2 h + \frac{1}{2K}|a(0)|^2 }{\frac{K}{2} - 2L^2h}.
	\end{align*} 
\end{lemma}
\begin{proof}
	We have from Eqn. \eqref{discretisationSGLD} that
	\begin{align}\label{eq:Y_increment}
	\begin{split}
	\E [|Y_{k+1}|^2] = &\E [|Y_k|^2] + \E [|h\hat{b}_s(Y_k,U_k) + \beta \sqrt{h} Z_{k+1} |^2] \\
	&\, + 2 \E \langle Y_k , h\hat{b}_s(Y_k,U_k) + \beta \sqrt{h} Z_{k+1} \rangle \\
	\leq & \E [|Y_k|^2]+h^2 \E \left[ |\hat{b}_s(Y_k,U_k)|^2 \right] +d\beta^2 h^2 \\
	&\, + 2 \E \langle Y_k , h\hat{b}_s(Y_k,U_k)\rangle,
	\end{split}
	\end{align}
	where we have used the fact that $Z_{k+1}$ is independent of $Y_k$. It remains to bound two terms $\E \left[ |\hat{b}_s(Y_k,U_k)|^2 \right]$ and $\E \langle Y_k , h\hat{b}_s(Y_k,U_k)\rangle$.
	
	For $\E \left[ |\hat{b}_s(Y_k,U_k)|^2 \right]$, note that it holds that
	\begin{equation}\label{eqn:samplingterms}
	\E \left[ |\hat{b}_s(Y_k,U_k)|^2 \right] \le  2\E \left[ |\hat{b}_s(Y_k,U_k) - \hat{b}_s(0,U_k)|^2 \right] +  2\E \left[ |\hat{b}_s(0,U_k)|^2 \right].  
	\end{equation}
	For the first term on the right hand of Eqn. \eqref{eqn:samplingterms}, we have that
	\begin{align*}
	&\mathbb{E}\!\left[\mathbb{E}_{k}^{U}\left[\left| \hat{b}_s(Y_k,U_k) - \hat{b}_s(0,U_k)\right|^{2}\right]\right] \\
	&\;= \mathbb{E}\!\left[\mathbb{E}_{k}^{U}\left[\left|\frac{1}{s}\sum_{i=1}^{m}\big(b_i(Y_k)-b_i(0)\big) U_k^{i}\right|^{2}\right]\right] \\
	&\;\le \mathbb{E}\!\Bigg[
	\frac{1}{s}\left(\frac{m-s}{m-1}\right)\sum_{i=1}^{m}\frac{1}{m}\,|b_i(Y_k)-b_i(0)|^{2}
	+\frac{m}{s}\left(\frac{s-1}{m-1}\right)\Big|\sum_{i=1}^{m}\frac{1}{m}\,(b_i(Y_k)-b_i(0))\Big|^{2}
	\Bigg] \\
	&\;\le \frac{1}{s}\left(\frac{m-s}{m-1}\right) L^{2}\,\mathbb{E}[|Y_k|^{2}]
	\;+\; \frac{m(s-1)}{s(m-1)} L^{2}\,\mathbb{E}[|Y_k|^{2}] \\
	&\;= L^{2}\,\mathbb{E}[|Y_k|^{2}],
	\end{align*}
	where the last equality is due to Lemma 
	\ref{lem:estimatorProduct} and the first inequality is due to Assumption \ref{ass:sol1}. 
	
	Similarly, applying Lemma 
	\ref{lem:estimatorProduct} to the second term on the right hand side of Eqn. \eqref{eqn:samplingterms} yields
	$$\E \left[ |\hat{b}_s(0,U_k)|^2 \right]\leq \max_{i=1,\ldots,m}|b(0,\xi_i)|^2.$$
	In total, we have that
	$$  \E \left[ |\hat{b}_s(Y_k,U_k)|^2 \right] \le  2 L^{2}\,\mathbb{E}[|Y_k|^{2}] +  2\max_{i=1,\ldots,m}|b(0,\xi_i)|^2.  $$
	For $\E \langle Y_k , h\hat{b}_s(Y_k,U_k)\rangle$, note that 
	$$\E \langle Y_k , h\hat{b}_s(Y_k,U_k)\rangle=\E\E_k^U \langle Y_k , h\hat{b}_s(Y_k,U_k)\rangle=\E \langle Y_k , h a(Y_k)\rangle,$$
	we shall apply the consequence \eqref{eqn:inequalitya} of Assumption \ref{ass:sol1} to get
	\begin{align*}
	\E \langle Y_k , ha(Y_k)\rangle&\leq -\frac{K}{2}h\E [|Y_k|^2]+\frac{h}{2K} |a(0)|^2.
	\end{align*}
	
	Pulling altogether yields
	\begin{align*} 
	\E [|Y_{k+1}|^2] \le \, &\big( 1 - \frac{Kh}{2} +2L^2 h^2 \big)\E [|Y_k|^2]  \\
	&\, + \left(  2\max_{i=1,\ldots,m}|b(0,\xi_i)|^2 h^2 + d \beta^2 h + \frac{h}{2K}|a(0)|^2 \right).
	\end{align*}
\end{proof}
The next corollary, the proof of which is deferred to Appendix \ref{sec:numerical}, translates the moment bound from Lemma \ref{lemma:momentSLGD} into a bound on the subsampling error along the SGLD chain.
\begin{corollary}\label{cor:slgddifference}
	Let the conditions of Lemma \ref{lemma:momentSLGD} hold and consider the chain $(Y_k)$ from \eqref{discretisationSGLD}. Then
	\begin{equation*}
	\E [|a(Y_k) - \hat{b}(Y_k)|^2] \leq  \frac{m-s}{s(m-1)} \mathcal{L} \,, 
	\end{equation*}
	where
	\begin{equation*}
	\mathcal{L} := L^2 \left( \E [|Y_0|^2] + \frac{2\max_{i=1,\ldots,m}|b(0,\xi_i)|^2 h + d \beta^2 h + \frac{1}{2K}|a(0)|^2 }{\frac{K}{2} - 2L^2h}\right) + \max_{i=1,\ldots,m} |b_i(0,\xi_i)|^2 \,.
	\end{equation*}
\end{corollary}
We now complete this section with two further estimates for the SGLD scheme: first, a bound on the mean-square difference between the Euler discretisation and the SGLD chain, and second, a variance bound for Lipschitz functionals of the SGLD iterates. The proofs of both results are postponed to Appendix \ref{sec:numerical}.
\begin{theorem}
	\label{bias:subsamp}
	Let Assumptions \ref{ass:sol1} and \ref{ass:sol2} hold. Consider the difference $\delta_k := \theta_k - Y_k$ with $\delta_0=0$, where $(\theta_k)$ and $(Y_k)$ are solutions of \eqref{discretisation} and \eqref{discretisationSGLD} respectively. 
	Provided that the stepsize $h$ satisfies \eqref{eqn:hchoice},
	then for any $k \geq 0$,
	\begin{gather*}
	\E_U[| \delta_{k+1} |^2] \le (1-2hK+ 2h^2L^2)\E_U[|\delta_k|^2] + 2\mathcal{L}h^2\frac{m-s}{(m-1)s},
	\end{gather*}
	and hence for any $k \geq 1$, 
	\begin{equation*}
	\E[| \delta_{k} |^2] \le \frac{m-s}{s(m-1)}\frac{\mathcal{L}h}{(K- hL^2)} \,.
	\end{equation*}
\end{theorem}

\begin{theorem}\label{thm:varianceSLGD}
	Let $f$ be a Lipschitz functional with Lipschitz coefficient $L_f$. Provided that the stepsize $h$ satisfies \eqref{eqn:hchoice}, then for an SGLD chain $(Y_k)$ given by \eqref{discretisationSGLD},
	we have 
	$$\V \big(f(Y_k)\big)\le L_f^2\cdot \left(4\beta^2d + 6h\left(\frac{m-s}{ms}\right)\mathcal{L}\right)\frac{1-(1-2hK+3L^2h^2)^{k+1}}{2K-3L^2h}.$$
\end{theorem}

\subsection{Proofs for the cost results}

We are finally ready to present the proofs of Theorems \ref{lem:cost_em} and \ref{lem:cost_sgld}.

\begin{proof}[Proof of Theorem \ref{lem:cost_em}]
	We begin with the standard MSE decomposition:
	\begin{align}\label{eq:mse_em}
	\begin{split}
	&\mathbb{E}_W\Bigl[\bigl|\pi(f)-\widehat{\pi}^{\mathrm{EM}}_{k,N_p}(f)\bigr|^2\Bigr]\\
	&\quad\leq
	2\bigl|\pi(f)-\mathbb{E}_W[f(X_{t_k})]\bigr|^2
	+ 2L_f^2 \mathbb{E}_W\bigl[|X_{t_k}-\theta_k|^2\bigr]
	+ 2N_p^{-1}\,\mathbb{V}_W\bigl(f(\theta_k)\bigr).       
	\end{split}
	\end{align}
	We choose parameters so that each term on the right-hand side is of order
	$\varepsilon^2$.
	
	For the first term on the RHS of \eqref{eq:mse_em}, Proposition~\ref{est:time} gives exponential convergence to equilibrium,
	so there exists a constant $\mathcal M>0$ dependent on $X_0$, $\pi$ and $L_f$ such that
	$$
	\bigl|\pi(f)-\mathbb{E}_W[f(X_{t_k})]\bigr|^2
	\leq
	\mathcal M^2 e^{-2Kt_k}.
	$$
	Hence it is sufficient to take
	\begin{equation}
	\label{eq:tk_choice_em}
	t_k \geq \frac{1}{K}\log(\mathcal M\varepsilon^{-1}).
	\end{equation}
	
	For the third term on the RHS of \eqref{eq:mse_em}, Theorem~\ref{var:euler:est} yields, for
	$$
	h \leq \min\!\left(\frac{1}{2K},\frac{2K}{L^2}\right),
	$$
	the bound
	$$
	\mathbb{V}_W\bigl(f(\theta_k)\bigr)
	\leq
	\frac{\beta^2 d L_f^2}{2K-L^2h}
	\leq
	\frac{\beta^2 d L_f^2}{K}.
	$$
	Therefore it is enough to choose
	\begin{equation}
	\label{eq:Np_choice_em}
	N_p
	\geq
	\max\!\left(\frac{\beta^2 d L_f^2}{K}\,\varepsilon^{-2},\,1\right).
	\end{equation}
	
	For the second term on the RHS of \eqref{eq:mse_em}, Theorem~\ref{bias:determ} implies that, for
	$$
	h \leq \min\!\left(\frac{1}{K},\frac{K}{4L^2}\right),
	$$
	we have
	\begin{align*}
	2L_f^2 \mathbb{E}_W\bigl[|X_{t_k}-\theta_k|^2\bigr]
	&\leq
	2L_f^2 \frac{K^{-1} C_R h^2}{K-2L^2h} \leq
	\frac{4L_f^2 C_R}{K^2} h^2.
	\end{align*}
	Thus it is sufficient to impose
	\begin{equation}
	\label{eq:h_choice_em}
	h
	\leq
	\frac{K\varepsilon}{2L_f\sqrt{C_R}}.
	\end{equation}
	Equivalently,
	$$
	\frac{1}{h}
	\geq
	\max\!\left(
	\frac{2L_f\sqrt{C_R}}{K\varepsilon},
	\frac{4L^2}{K},
	2K
	\right).
	$$
	
	Recall that the total cost is proportional to
	$$
	\mathrm{Cost}(\mathrm{EM})
	\asymp
	N_p \times t_k \times \frac{1}{h} \times m.
	$$
	Using \eqref{eq:tk_choice_em}, \eqref{eq:Np_choice_em}, and \eqref{eq:h_choice_em}, we obtain
	$$
	\mathrm{Cost}(\mathrm{EM})
	\asymp
	\max \left( \varepsilon^{-2} \frac{\beta^2 \cdot d \cdot L_f^2}{K} , 1 \right) \cdot \frac1K\log (\mathcal{M}\varepsilon^{-1}) \cdot \max\!\left(
	\frac{2L_f\sqrt{C_{R}}}{K\varepsilon},\;
	\frac{4L^{2}}{K},\;
	2K
	\right) \cdot m.
	$$
	
	Recall that in our scaling $\beta = m^{-1/2}$ while $K$, $L$, and $C_R$ are independent of $m$.
	Hence
	$$
	\mathrm{Cost}(\mathrm{EM})
	\asymp
	\max\!\bigl(\varepsilon^{-2}m^{-1},1\bigr)\,
	\log(\varepsilon^{-1})\,
	\varepsilon^{-1}\,m.
	$$
	This immediately gives
	$$
	\mathrm{Cost}(\mathrm{EM})
	\asymp
	\begin{cases}
	\varepsilon^{-3}\log(\varepsilon^{-1}), & \text{if } m < \varepsilon^{-2},\\[1ex]
	m\,\varepsilon^{-1}\log(\varepsilon^{-1}), & \text{if } m \geq \varepsilon^{-2}.
	\end{cases}
	$$
	This completes the proof.
\end{proof}

\begin{proof}[Proof of Theorem \ref{lem:cost_sgld}]
	As in the Euler case, we decompose the MSE:
	\begin{align}\label{eq:mse_sgld}
	\begin{split}
	&\mathbb{E}\Bigl[\bigl|\pi(f)-\widehat{\pi}^{\mathrm{SGLD}}_{k,N_p}(f)\bigr|^2\Bigr]\\
	&\quad\leq
	2\bigl|\pi(f)-\mathbb{E}[f(X_{t_k})]\bigr|^2
	+ 2L_f^2 \mathbb{E}[|X_{t_k}-Y_k|^2]
	+ 2N_p^{-1}\,\mathbb{V}\bigl(f(Y_k)\bigr).        
	\end{split}
	\end{align}
	
	The first term on the RHS of \eqref{eq:mse_sgld} is exactly the same as in the Euler case, so it is enough to choose
	\begin{equation}
	\label{eq:tk_choice_sgld}
	t_k \geq \frac{1}{K}\log(\mathcal{M}\varepsilon^{-1}).
	\end{equation}
	
	For the variance term of \eqref{eq:mse_sgld}, Theorem~\ref{thm:varianceSLGD} yields
	$$
	\mathbb{V}\bigl(f(Y_k)\bigr)
	\leq
	L_f^2
	\frac{4\beta^2 d + 6h \left(\frac{m-s}{ms}\right)\mathcal L}{2K-3L^2h}
	\leq
	\frac{L_f^2}{K}
	\left(
	4\beta^2 d + 6h \frac{m-s}{ms}\mathcal L
	\right),
	$$
	provided
	$$
	h \leq \min\!\left(\frac{2}{K},\frac{K}{4L^2}\right).
	$$
	Hence it is sufficient to choose
	\begin{equation}
	\label{eq:Np_choice_sgld}
	N_p
	\geq
	\max\!\left(
	\frac{L_f^2}{K}
	\left(
	4\beta^2 d + 6h \frac{m-s}{ms}\mathcal L
	\right)\varepsilon^{-2},
	\,1
	\right).
	\end{equation}

	Next, for the bias term, note that
	$ \mathbb{E}\!\left[\lvert X_{t_k}-Y_h\rvert^{2}\right]
	\le
	2\,\mathbb{E}\!\left[\lvert X_{t_k}-\theta_k\rvert^{2}\right]
	+
	2\,\mathbb{E}\!\left[\lvert \theta_k-Y_h\rvert^{2}\right].
	$
	The first term has been bounded by a quantity of order $\varepsilon^{2}$ in the proof of Theorem \ref{lem:cost_em}; it remains to bound the second one.
	From Theorem \ref{bias:subsamp},
	$
	4L_f^{2}\,\mathbb{E}\!\left[\lvert \theta_k-Y_h\rvert^{2}\right]
	\le
	8L_f^{2}\,\frac{m-s}{s(m-1)}\,\frac{\mathcal{L}\cdot h}{K},
	$
	if $h$ satisfies (\ref{eqn:hchoice}). Therefore, we have the following condition:
	\begin{align*}
	8L_f^{2}\,\frac{m-s}{s(m-1)}\,\frac{\mathcal{L}\cdot h}{K}
	\le
	\varepsilon^{2}
	\quad\Longleftrightarrow\quad
	\frac{1}{h}
	\ge
	\varepsilon^{-2}\,K^{-1}\,8L_f^{2}\mathcal{L}\,\frac{m-s}{s(m-1)}. 
	\end{align*}
	In total,
	$$\frac{1}{h}
	>
	\max\!\left(
	\frac{2L_f\sqrt{C_R}}{K\varepsilon},\;
	\frac{4L^{2}}{K},\;
	2K,\;
	\varepsilon^{-2}\,K^{-1}\,8L_f^{2}\,\frac{\mathcal{L}(m-s)}{s(m-1)}
	\right).$$
	
	Similarly as in the proof of Theorem \ref{lem:cost_em}, we now note that all the constants appearing above (with the exception of $\beta = 1 /\sqrt{m}$) are independent of $m$. Hence our bound on the number of paths from \eqref{eq:Np_choice_sgld} becomes
	\begin{equation*}
	N_p > \max \left( \varepsilon^{-2} \left( \frac{1}{m} +  \varepsilon^2 \right) , 1  \right) \,.
	\end{equation*}
	Hence we need
	\begin{equation*}
	N_p > \max \left( \varepsilon^{-2} m^{-1} + 1 , 1  \right) \,.
	\end{equation*}
	However, for the purpose of our cost analysis, we are only interested in comparing leading order terms, hence the requirement above is equivalent to
	\begin{equation*}
	N_P > \max \left( \varepsilon^{-2} m^{-1} , 1  \right) \,,
	\end{equation*}
	i.e., it is identical to the requirement on the number of paths that we obtained in the proof of Theorem \ref{lem:cost_em}.
	
	Finally, the condition on $1/h$ becomes
	\begin{equation*}
	\frac{1}{h} > \max \left( \varepsilon^{-1} , \varepsilon^{-2} \left( \frac{m-s}{(m-1)s} \right) \right) 
	\end{equation*}
	and hence the overall cost becomes
	\begin{align*}
	\mathrm{Cost}(\mathrm{SGLD})
	&\asymp
	N_p \times t_k \times \frac{1}{h} \times s\\
	&\asymp\max \left( \varepsilon^{-2} m^{-1}, 1 \right)
	\log(\varepsilon^{-1})
	\max \left( \varepsilon^{-1}, \varepsilon^{-2}\frac{m-s}{(m-1)s} \right)\cdot s
	\\
	&\asymp\max\left( \varepsilon^{-2} m^{-1}, 1 \right)
	\log(\varepsilon^{-1})
	\max \left( \varepsilon^{-1}, \varepsilon^{-2}\frac{m-s}{m s} \right)\cdot s .
	\end{align*}
	
	It remains to distinguish the two possibilities for each maximum.
	
	\medskip
	
	\noindent
	\textbf{Case 1:} $m < \varepsilon^{-2}$, so
	$$
	\max\!\bigl(\varepsilon^{-2}m^{-1},1\bigr)=\varepsilon^{-2}m^{-1}.
	$$
	Then:
	\begin{itemize}
		\item if $\dfrac{m-s}{ms} < \varepsilon$, then
		$$
		\max\!\left(\varepsilon^{-1},\,\varepsilon^{-2}\frac{m-s}{ms}\right)=\varepsilon^{-1},
		$$
		and therefore
		$$
		\mathrm{Cost}(\mathrm{SGLD})
		\asymp
		\varepsilon^{-2}m^{-1}\log(\varepsilon^{-1})\,\varepsilon^{-1}s
		=
		\varepsilon^{-3}\log(\varepsilon^{-1})\,\frac{s}{m};
		$$
		
		\item if $\dfrac{m-s}{ms} \geq \varepsilon$, then
		$$
		\max\!\left(\varepsilon^{-1},\,\varepsilon^{-2}\frac{m-s}{ms}\right)
		=
		\varepsilon^{-2}\frac{m-s}{ms},
		$$
		and hence
		\begin{align}
		\mathrm{Cost}(\mathrm{SGLD})
		&\asymp
		\varepsilon^{-2}m^{-1}\log(\varepsilon^{-1})
		\left(\varepsilon^{-2}\frac{m-s}{ms}\right)s \\
		&=
		\varepsilon^{-4}\log(\varepsilon^{-1})\,\frac{m-s}{m^2}.
		\end{align}
	\end{itemize}
	
	\noindent
	\textbf{Case 2:} $m \geq \varepsilon^{-2}$, so
	$$
	\max\!\bigl(\varepsilon^{-2}m^{-1},1\bigr)=1.
	$$
	Then:
	\begin{itemize}
		\item if $\dfrac{m-s}{ms} < \varepsilon$, then
		$$
		\mathrm{Cost}(\mathrm{SGLD})
		\asymp
		\log(\varepsilon^{-1})\,\varepsilon^{-1}s;
		$$
		
		\item if $\dfrac{m-s}{ms} \geq \varepsilon$, then
		\begin{align}
		\mathrm{Cost}(\mathrm{SGLD})
		&\asymp
		\log(\varepsilon^{-1})
		\left(\varepsilon^{-2}\frac{m-s}{ms}\right)s \\
		&=
		\varepsilon^{-2}\log(\varepsilon^{-1})\,\frac{m-s}{m}.
		\end{align}
	\end{itemize}
	
	This proves the four stated regimes.
\end{proof}

\section{Additional numerical experiments}\label{section: additional numerical}

In this section we present additional numerical experiments that verify our estimates from Section \ref{sec:discussion}.

We use the same setting as in the numerical experiments in Section \ref{section:numerical cost}.

\subsection{Illustration of Theorem \ref{bias:determ}  and Theorem \ref{bias:subsamp}}
\leavevmode

Theorem~\ref{bias:determ} gives strong order one for the Euler
method, while Theorem~\ref{bias:subsamp}, together with the triangle
inequality, gives strong order one half for the SGLD method when the
mini-batch size satisfies $s<m$. We investigate these convergence
rates for the scaled SDE introduced in
Section~\ref{section:numerical cost}, for two dataset sizes
$m\in\{32,64\}$, using $10^3$ sample paths.

For each dataset size $m$, we take the terminal time $3m$.
For each refinement level $k=4,\ldots,9$, we use the stepsize
$h_k = 2^{-k}$,
and compute the corresponding reference solution using
$  h_{\mathrm{ref}} = 2^{-12}$ .

Consequently, the coarse and reference solutions are evaluated after
\begin{equation*}
n_k=\frac{T_m}{h_k}
\qquad\text{and}\qquad
n_{\mathrm{ref}}=\frac{T_m}{h_{\mathrm{ref}}}
\end{equation*}
time steps, respectively. The coarse and reference paths are
synchronously coupled through their Brownian increments, and the
strong errors are evaluated at the common terminal time $T_m$.
We fit a linear regression model to the error data to estimate the
empirical strong convergence orders, as shown in
Figure~\ref{fig:strongorder}.

\begin{figure}
	\centering
	\begin{subfigure}{0.385\textwidth}
		\centering
		\includegraphics[width=\textwidth]{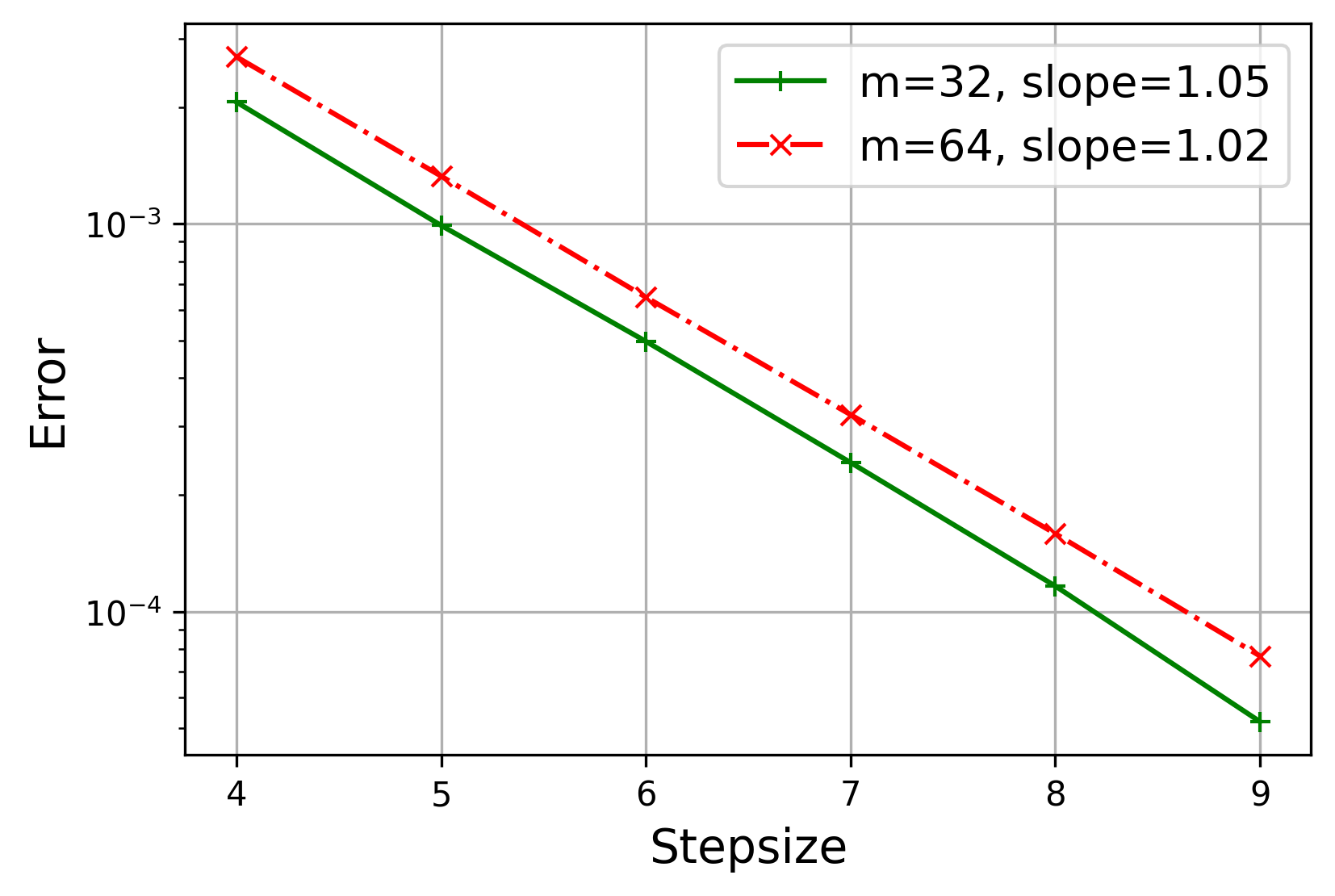}
		\caption{Euler on  on $m=32,64$. \label{fig:eulerstrongm}}
	\end{subfigure}
	\begin{subfigure}{0.4\textwidth}
		\centering
		\includegraphics[width=\textwidth]{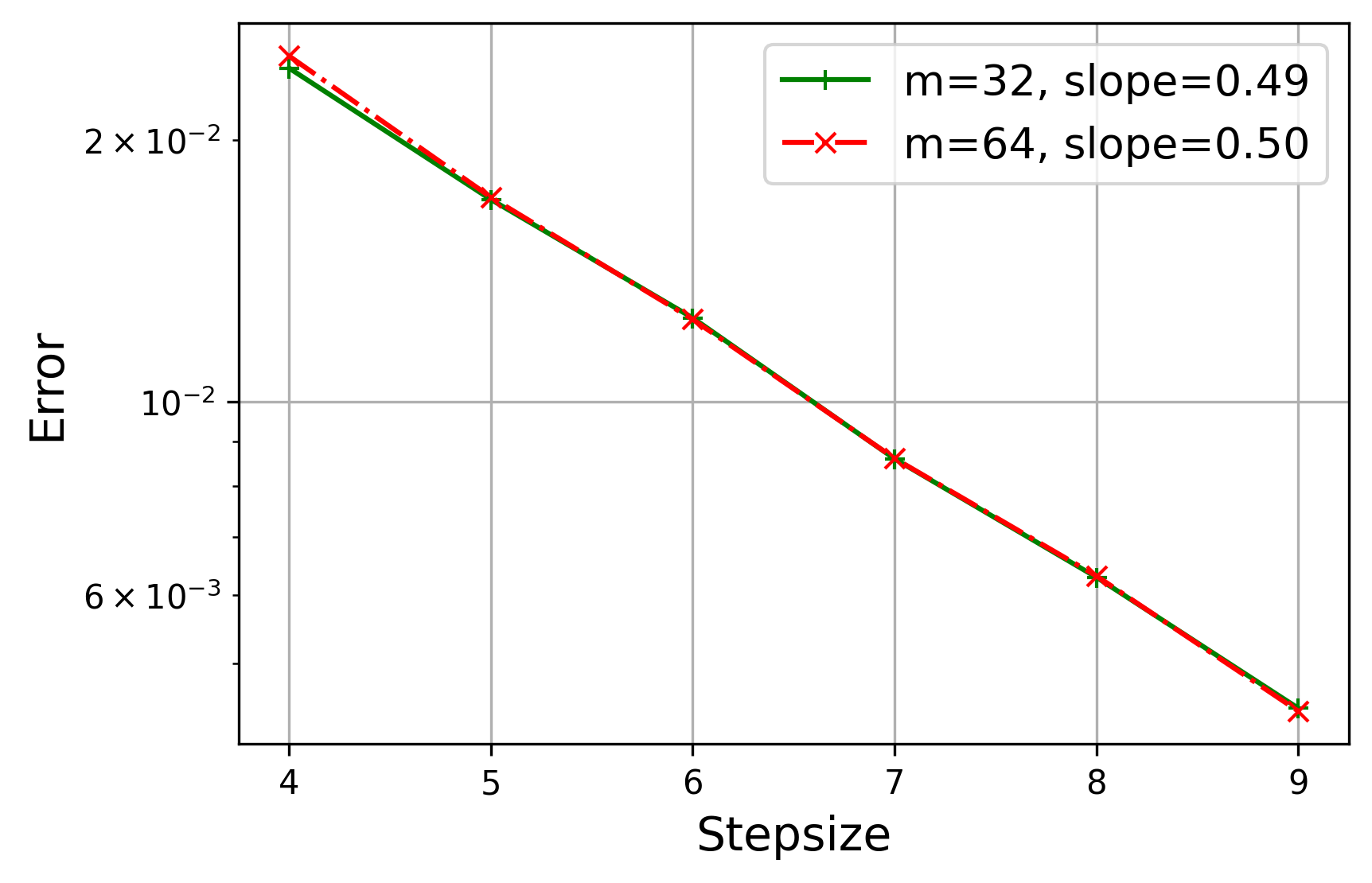}
		\caption{SGLD on  on $m=32,64$. \label{fig:sgdstrongm}}
	\end{subfigure}
	\caption{The strong convergence of both Euler and SGLD methods. \label{fig:strongorder}}
\end{figure}

\subsection{Illustration of Lemma \ref{var:euler:est} and Theorem \ref{thm:varianceSLGD} }
\leavevmode
Lemma \ref{var:euler:est} and Theorem \ref{thm:varianceSLGD} give the upper bounds of the estimates for the variance of a single realisation from a full EM and from a SGLD method respectively. For the full Euler, we have
$$\V_W( f(\theta_k)) \le\frac{\beta^2\cdot d\cdot L_f^2}{2K-L^2h}\propto \frac{1}{m},$$
and similarly for fixed $s$
$$\V \big(f(Y_k)\big)\le L_f^2\cdot \left(4\beta^2d + 6h\left(\frac{m-s}{ms}\right)\mathcal{L}\right)\frac{1-(1-2hK+3L^2h^2)^{k+1}}{2K-3L^2h}\propto \frac{1}{m}+\frac{h}{s}.$$

We will first justify these upper bounds are decreasing with $m$ at an order one on different $f$ for full Euler.  We fix the number of experiments $10^4$, $h=2^{-11}$ and $s=16$,  and test the variance for each $m=2^k$, $k=5,\ldots,12$. One can observe that the variance is decreasing roughly at an order one for different f and for each method in Figure \ref{fig:variance}.

\begin{figure}
	\centering
	\begin{subfigure}{0.40\textwidth}
		\centering
		\includegraphics[width=\textwidth]{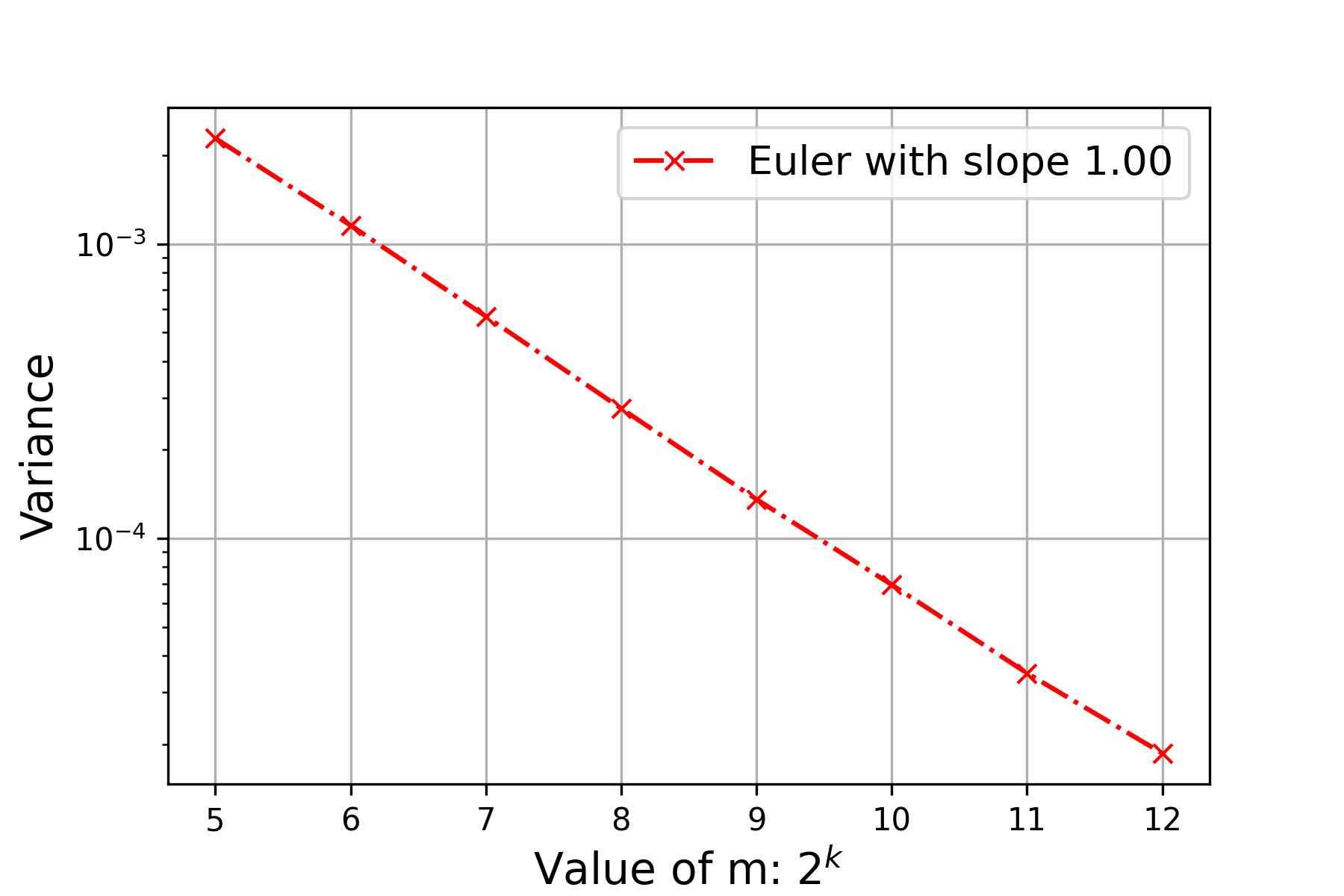}
		\caption{$f_1$. \label{fig:eulervariancef1}}
	\end{subfigure}
	\begin{subfigure}{0.4\textwidth}
		\centering
		\includegraphics[width=\textwidth]{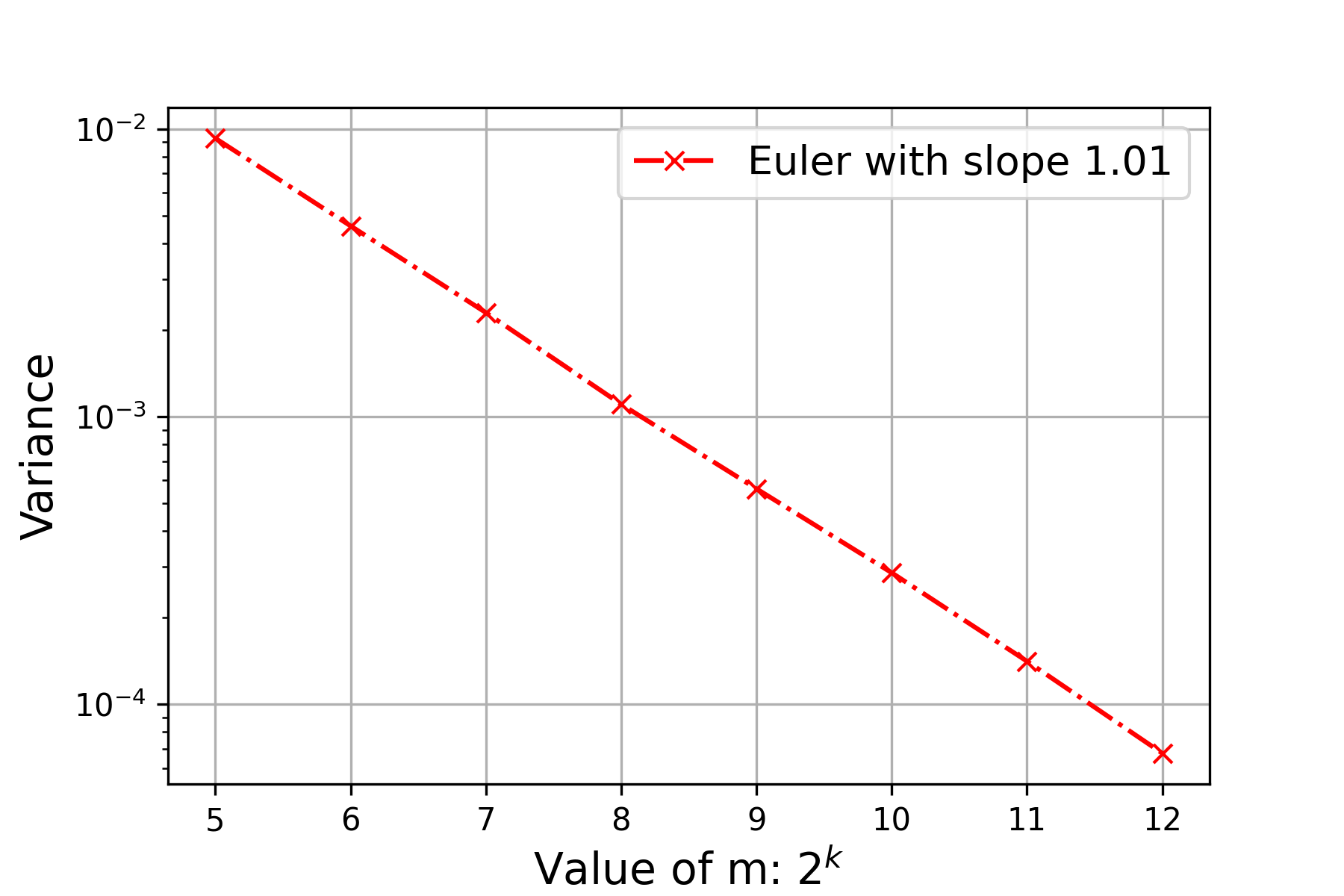}
		\caption{$f_2$.\label{eulervariancef2}}  
	\end{subfigure}
	\caption{The variance of a single simulation for the Euler method. \label{fig:variance} }
\end{figure}

We then examine the dependence of the SGLD variance on the mini-batch
size $s$. We fix the number of experiments at $10^4$, the stepsize at
$h=2^{-7}$, and the dataset size at $m=2^{14}$, and estimate the
variance for $s=2^k$, $k=1,\ldots,8$. As shown in
Figure~\ref{fig:variance_s}, the variance decreases as the mini-batch
size $s$ increases for the different choices of $f$.
\begin{figure}
	\centering
	\includegraphics[width=0.5\textwidth]{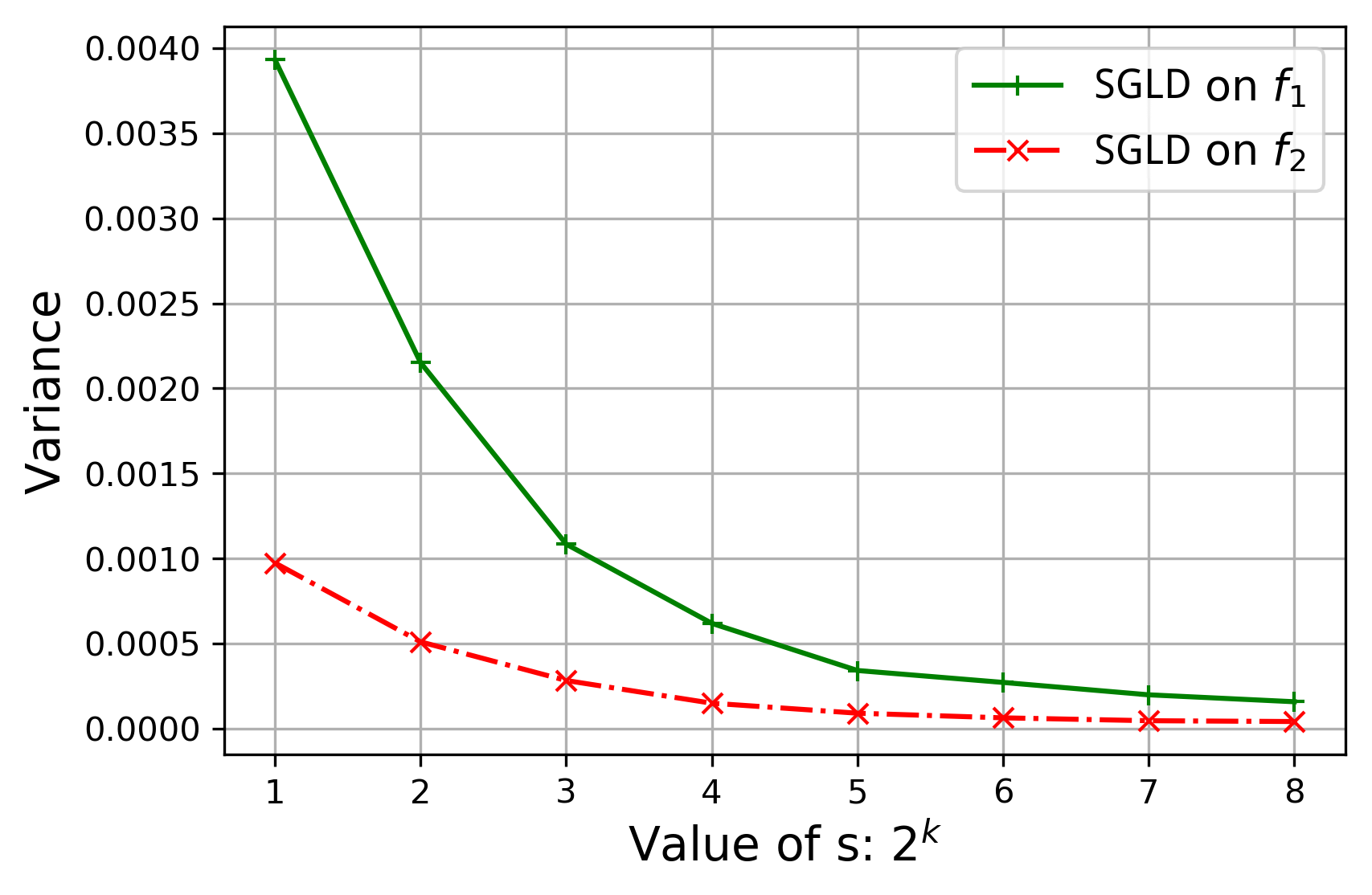}
	\caption{The variance of a single simulation for the SGLD method with respect to $s$. \label{fig:variance_s}}
\end{figure}

Finally, we investigate how the single-path variance depends on the
stepsize in Figure \ref{fig:variance_h}. To do this, we test it on different $m\in \{32,256,4096\}$ and $f$ for each $h=2^{-k}$, $k=7,\ldots,11$. For the full EM scheme, the variance is
only weakly affected by the stepsize over the tested range and approaches
a nearly constant level as \(h\) decreases. For SGLD, however, the
stepsize dependence becomes increasingly visible as \(m\) grows. This is
consistent with the additional subsampling contribution $h/s$
in the SGLD variance bound. When \(s\) is fixed and \(m\) is large, this
term can dominate the diffusion contribution of order \(m^{-1}\).
Consequently, the SGLD variance decreases as the stepsize is refined. For
smaller \(m\), the \(m^{-1}\) contribution is comparatively large, so the
SGLD variance is less sensitive to \(h\). These observations agree with
the theoretical prediction that the SGLD variance has the leading-order
structure
\[
\operatorname{Var}(f(Y_k))
\lesssim
\frac{1}{m}
+
\frac{h}{s},
\]
and approaches an \(h\)-independent variance floor once the subsampling
contribution becomes negligible.

\begin{figure}
	\centering
	\includegraphics[width=\textwidth]{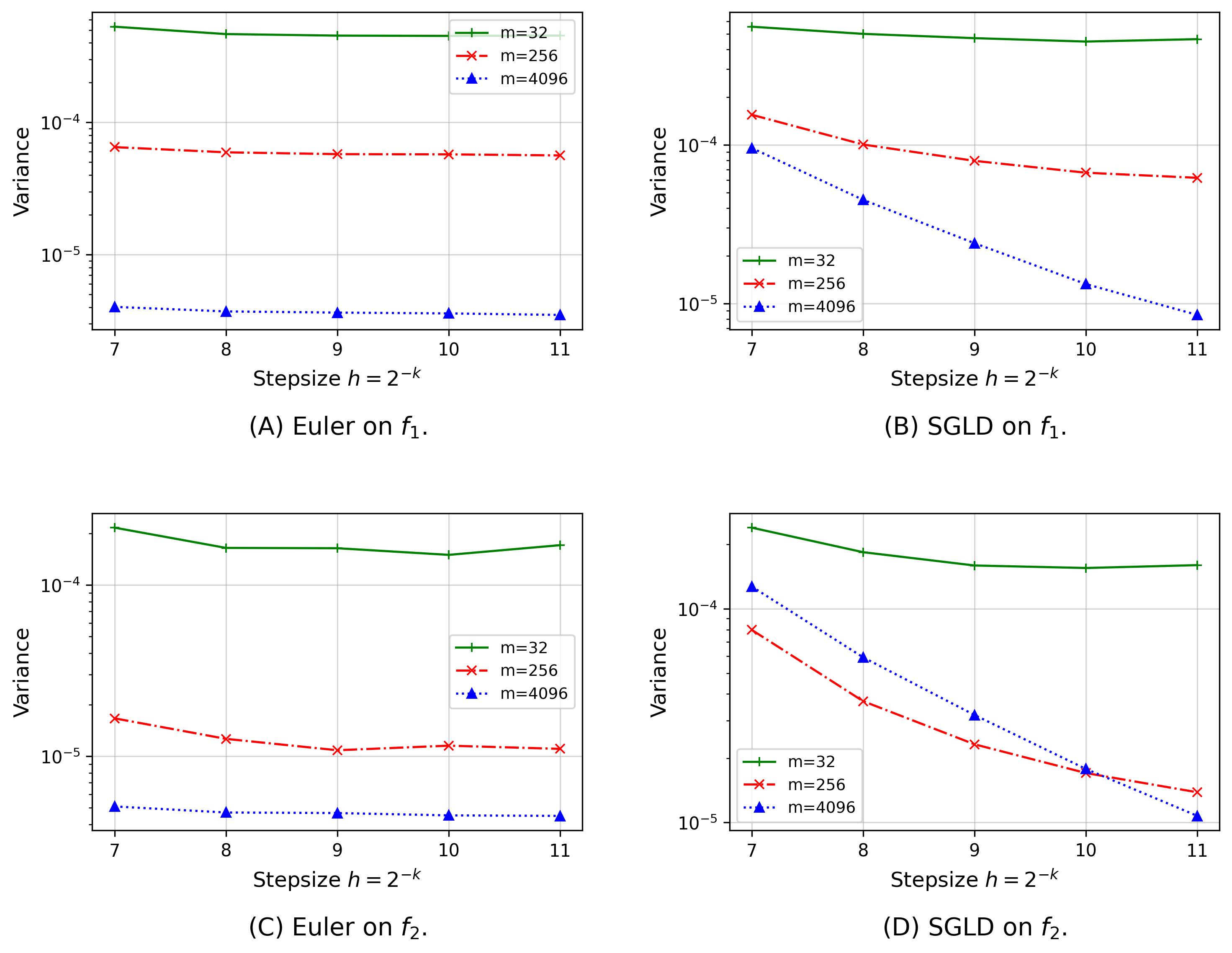}
	\caption{The variance of a single simulation for both Euler and SGLD methods with respect to the stepsize. \label{fig:variance_h}}
\end{figure}

\section*{Acknowledgment}
The code used to reproduce the numerical experiments and figures in this paper is publicly available  via \href{https://github.com/Danqi-ki/ON-THE-COMPUTATIONAL-COST-OF-STOCHASTIC-GRADIENT-LANGEVIN-DYNAMICS}{the github repository}.

\bibliographystyle{plain}

\bibliography{References}

\appendix

\section{Proofs in Section \ref{section:solution}}\label{sec:proofs}

\begin{proof}[The proof of \eqref{eqn:inequalitya}]
	Indeed, taking $y=0$ in \eqref{eq:a2} yields
	\begin{equation*}
	\langle x,a(x)-a(0)\rangle
	\leq
	-K|x|^2.
	\end{equation*}
	Consequently,
	\begin{align*}
	\langle x,a(x)\rangle
	=
	\langle x,a(x)-a(0)\rangle
	+
	\langle x,a(0)\rangle
	\leq
	-K|x|^2
	+
	|x|\,|a(0)|.
	\end{align*}
	By Young's inequality,
	\begin{align*}
	\langle x,a(x)\rangle
	&\leq
	-K|x|^2
	+
	\frac{K}{2}|x|^2
	+
	\frac{1}{2K}|a(0)|^2
	=
	-\frac{K}{2}|x|^2
	+
	\frac{1}{2K}|a(0)|^2,
	\end{align*}
	which proves \eqref{eqn:inequalitya}.
\end{proof}
\begin{proof}[Proof of Proposition~\ref{est:time}]
	By~\eqref{eq:a1}, the drift $a$ is globally Lipschitz. Therefore,
	for every square-integrable initial condition,
	SDE~\eqref{generalSDE} admits a unique strong solution.
	
	We first establish a contraction estimate. For any
	$x,y\in\mathbb{R}^d$, let $(X_t^x)_{t\geq 0}$ and
	$(X_t^y)_{t\geq 0}$ be the solutions of~\eqref{generalSDE}
	starting from $x$ and $y$, respectively, and driven by the same
	Brownian motion. Since the diffusion coefficient is constant, the
	Brownian terms cancel, and hence
	\begin{equation}\label{eq:synchronous_difference}
	X_t^x-X_t^y
	=
	x-y+
	\int_0^t
	\bigl(a(X_s^x)-a(X_s^y)\bigr)
	\,\mathrm{d}s.
	\end{equation}
	It follows from~\eqref{eq:a2} that
	\begin{align}
	\frac{\mathrm{d}}{\mathrm{d}t}|X_t^x-X_t^y|^2
	&=
	2\left\langle
	X_t^x-X_t^y,
	a(X_t^x)-a(X_t^y)
	\right\rangle
	\leq
	-2K|X_t^x-X_t^y|^2.
	\end{align}
	Therefore, Gronwall's inequality yields
	\begin{equation}\label{eq:synchronous_contraction}
	|X_t^x-X_t^y|
	\leq
	\exp(-Kt)|x-y|,
	\qquad t\geq 0,
	\quad \mathbb{P}_W\text{-a.s.}
	\end{equation}
	Consequently, for every Lipschitz function
	$f:\mathbb{R}^d\to\mathbb{R}$,
	\begin{align}\label{eq:semigroup_lipschitz_contraction}
	\begin{split}
	|P_tf(x)-P_tf(y)|
	&=
	\left|
	\mathbb{E}_W
	\bigl[
	f(X_t^x)-f(X_t^y)
	\bigr]
	\right|
	\\
	&\leq
	\operatorname{Lip}(f)
	\mathbb{E}_W[|X_t^x-X_t^y|]
	\\
	&\leq
	\operatorname{Lip}(f)
	\exp(-Kt)|x-y|.
	\end{split}
	\end{align}
	
	We next prove the existence of an invariant measure. Let $\mathcal{P}_1(\mathbb{R}^d)$ denote the space of Borel
	probability measures on $\mathbb{R}^d$ with finite first moment, that is,
	\begin{equation*}
	\mathcal{P}_1(\mathbb{R}^d)
	:=
	\left\{
	\mu:
	\mu \text{ is a Borel probability measure on }\mathbb{R}^d,
	\quad
	\int_{\mathbb{R}^d}|x|\,\mu(\mathrm{d}x)<\infty
	\right\}.
	\end{equation*}
	For $\mu,\nu\in\mathcal{P}_1(\mathbb{R}^d)$, their
	$1$-Wasserstein distance is defined by
	\begin{equation}\label{eq:W1_definition}
	\mathcal{W}_1(\mu,\nu)
	:=
	\inf_{\gamma\in\Gamma(\mu,\nu)}
	\int_{\mathbb{R}^d\times\mathbb{R}^d}
	|x-y|
	\,\gamma(\mathrm{d}x,\mathrm{d}y),
	\end{equation}
	where $\Gamma(\mu,\nu)$ denotes the collection of all couplings of
	$\mu$ and $\nu$, namely,
	\begin{equation*}
	\Gamma(\mu,\nu)
	:=
	\left\{
	\gamma\in\mathcal{P}(\mathbb{R}^d\times\mathbb{R}^d):
	\gamma(A\times\mathbb{R}^d)=\mu(A),
	\quad
	\gamma(\mathbb{R}^d\times A)=\nu(A)
	\right\},
	\end{equation*}
	for every Borel set $A\subseteq\mathbb{R}^d$.
	
	For $x\in\mathbb{R}^d$, let $(X_t^x)_{t\geq 0}$ denote the
	solution of SDE~\eqref{generalSDE} with initial condition
	$X_0^x=x$. The transition kernel associated with
	SDE~\eqref{generalSDE} is defined by
	\begin{equation}\label{eq:transition_kernel}
	p_t(x,A)
	:=
	\mathbb{P}_W(X_t^x\in A),
	\qquad
	x\in\mathbb{R}^d,
	\quad
	A\in\mathcal{B}(\mathbb{R}^d).
	\end{equation}
	The transition semigroup $(P_t)_{t\geq 0}$ acts on measurable
	functions according to
	\begin{equation}\label{eq:transition_semigroup}
	P_tf(x)
	:=
	\mathbb{E}_W[f(X_t^x)]
	=
	\int_{\mathbb{R}^d}f(y)\,p_t(x,\mathrm{d}y).
	\end{equation}
	
	For a probability measure $\mu$ on $\mathbb{R}^d$, we denote by
	$\mu P_t$ the action of the transition semigroup on measures. Using
	the duality pairing
	\begin{equation*}
	\langle \mu,f\rangle
	:=
	\int_{\mathbb{R}^d}f(x)\,\mu(\mathrm{d}x),
	\end{equation*}
	this action is defined by
	\begin{equation}\label{eq:dual_semigroup_action}
	\langle \mu P_t,f\rangle
	=
	\langle \mu,P_tf\rangle
	\end{equation}
	for every bounded measurable function
	$f:\mathbb{R}^d\to\mathbb{R}$. Equivalently,
	\begin{align*}
	\langle \mu P_t,f\rangle
	&=
	\int_{\mathbb{R}^d}
	\int_{\mathbb{R}^d}
	f(y)\,p_t(x,\mathrm{d}y)
	\,\mu(\mathrm{d}x),
	\end{align*}
	where $p_t(x,\cdot)$ is the transition kernel associated with
	SDE~\eqref{generalSDE}.
	
	Under Assumption~\ref{ass:sol1}, the solution can be chosen so that
	the map
	\begin{equation*}
	\Phi_t:
	\mathbb{R}^d\times\Omega_W
	\longrightarrow
	\mathbb{R}^d,
	\qquad
	\Phi_t(x,\omega):=X_t^x(\omega),
	\end{equation*}
	is measurable. Therefore,
	\begin{equation}\label{eq:measure_pushforward_solution}
	\mu P_t
	=
	(\Phi_t)_{\#}
	\bigl(\mu\otimes\mathbb{P}_W\bigr),
	\end{equation}
	where, for a measurable map $\Phi$ and a measure $\lambda$, the
	pushforward measure $\Phi_{\#}\lambda$ is defined by
	\begin{equation*}
	(\Phi_{\#}\lambda)(A)
	:=
	\lambda\bigl(\Phi^{-1}(A)\bigr),
	\qquad
	A\in\mathcal{B}(\mathbb{R}^d).
	\end{equation*}
	Indeed, for any $A\in\mathcal{B}(\mathbb{R}^d)$,
	\begin{align}
	\begin{split}
	(\Phi_t)_{\#}
	\bigl(\mu\otimes\mathbb{P}_W\bigr)(A)
	&=
	\bigl(\mu\otimes\mathbb{P}_W\bigr)
	\left(
	\left\{
	(x,\omega):
	X_t^x(\omega)\in A
	\right\}
	\right)\\
	&=
	\int_{\mathbb{R}^d}
	\mathbb{P}_W(X_t^x\in A)
	\,\mu(\mathrm{d}x)
	\\
	&=
	\int_{\mathbb{R}^d}
	p_t(x,A)
	\,\mu(\mathrm{d}x)
	\\
	&=
	(\mu P_t)(A).
	\end{split}
	\end{align}
	
	The synchronous coupling estimate~\eqref{eq:synchronous_contraction}
	implies a contraction of the transition semigroup in the
	$1$-Wasserstein distance. Indeed, let
	$\mu,\nu\in\mathcal{P}_1(\mathbb{R}^d)$ and fix an arbitrary coupling
	$\gamma\in\Gamma(\mu,\nu)$. Consider the product probability space
	\begin{equation*}
	\left(
	\mathbb{R}^d\times\mathbb{R}^d\times\Omega_W,
	\mathcal{B}(\mathbb{R}^d)
	\otimes\mathcal{B}(\mathbb{R}^d)
	\otimes\mathcal{F}^W,
	\gamma\otimes\mathbb{P}_W
	\right).
	\end{equation*}
	Let the coordinate random variables be
	\begin{equation*}
	\xi(x,y,\omega):=x,
	\qquad
	\eta(x,y,\omega):=y.
	\end{equation*}
	Then
	\begin{equation*}
	\mathcal{L}(\xi)=\mu,
	\qquad
	\mathcal{L}(\eta)=\nu,
	\qquad
	\mathcal{L}(\xi,\eta)=\gamma,
	\end{equation*}
	and $(\xi,\eta)$ is independent of the Brownian motion $W$.

	Let $(X_t^\xi)_{t\geq 0}$ and $(X_t^\eta)_{t\geq 0}$ be the
	solutions of~\eqref{generalSDE} with initial conditions $\xi$ and
	$\eta$, respectively, driven by the same Brownian motion. More
	precisely,
	\begin{align*}
	X_t^\xi
	&=
	\xi+
	\int_0^t a(X_s^\xi)\,\mathrm{d}s
	+
	\beta W_t,\\
	X_t^\eta
	&=
	\eta+
	\int_0^t a(X_s^\eta)\,\mathrm{d}s
	+
	\beta W_t.
	\end{align*}
	By the definition of the action of the transition semigroup on
	measures,
	\begin{equation*}
	\mathcal{L}(X_t^\xi)=\mu P_t,
	\qquad
	\mathcal{L}(X_t^\eta)=\nu P_t.
	\end{equation*}
	Consequently, the joint law
	\begin{equation*}
	\gamma_t
	:=
	\mathcal{L}(X_t^\xi,X_t^\eta)
	\end{equation*}
	belongs to $\Gamma(\mu P_t,\nu P_t)$.
	
	Using the definition of the $1$-Wasserstein distance and then the
	synchronous contraction estimate
	\eqref{eq:synchronous_contraction}, we obtain
	\begin{align}
	\begin{split}
	\mathcal{W}_1(\mu P_t,\nu P_t)
	&\leq
	\int_{\mathbb{R}^d\times\mathbb{R}^d}
	|u-v|
	\,\gamma_t(\mathrm{d}u,\mathrm{d}v)
	\\
	&=
	\mathbb{E}_{\gamma\otimes\mathbb{P}_W}
	\left[
	|X_t^\xi-X_t^\eta|
	\right]
	\\
	&\leq
	\exp(-Kt)
	\mathbb{E}_{\gamma\otimes\mathbb{P}_W}
	\left[
	|\xi-\eta|
	\right]
	\\
	&=
	\exp(-Kt)
	\int_{\mathbb{R}^d\times\mathbb{R}^d}
	|x-y|
	\,\gamma(\mathrm{d}x,\mathrm{d}y).
	\label{eq:wasserstein_contraction_for_gamma}
	\end{split}
	\end{align}
	Since $\gamma\in\Gamma(\mu,\nu)$ was arbitrary, taking the infimum
	over all couplings $\gamma$ gives
	\begin{equation}\label{eq:wasserstein_contraction}
	\mathcal{W}_1(\mu P_t,\nu P_t)
	\leq
	\exp(-Kt)\mathcal{W}_1(\mu,\nu).
	\end{equation}
	Furthermore, an application of Lemma \ref{lm:2} gives
	\begin{align}
	\label{eq:first_moment_transition}
	\begin{split}
	\mathbb{E}_W[|X_t^x|]
	&\leq
	\exp(-Kt)|x|
	+
	\left(
	\frac{|a(0)|^2}{K^2}
	+
	\frac{d\beta^2}{K}
	\right)^{1/2}. 
	\end{split}
	\end{align}
	It follows from the definition of $\mu P_t$ that
	\begin{align*}
	\int_{\mathbb{R}^d}
	|u|\,(\mu P_t)(\mathrm{d}u)
	&=
	\int_{\mathbb{R}^d}
	\mathbb{E}_W[|X_t^x|]
	\,\mu(\mathrm{d}x)
	\\
	&\leq
	\exp(-Kt)
	\int_{\mathbb{R}^d}
	|x|\,\mu(\mathrm{d}x)
	+
	\left(
	\frac{|a(0)|^2}{K^2}
	+
	\frac{d\beta^2}{K}
	\right)^{1/2}.
	\end{align*}
	Since $\mu\in\mathcal{P}_1(\mathbb{R}^d)$, the right-hand side is
	finite. Therefore,
	\begin{equation*}
	\mu P_t\in\mathcal{P}_1(\mathbb{R}^d).
	\end{equation*}
	
	Consider the mapping
	\begin{equation*}
	\mathcal{T}:
	\mathcal{P}_1(\mathbb{R}^d)
	\longrightarrow
	\mathcal{P}_1(\mathbb{R}^d),
	\qquad
	\mathcal{T}(\mu):=\mu P_1.
	\end{equation*}
	By~\eqref{eq:wasserstein_contraction},
	\begin{equation*}
	\mathcal{W}_1(\mathcal{T}(\mu),\mathcal{T}(\nu))
	\leq
	\exp(-K)\mathcal{W}_1(\mu,\nu).
	\end{equation*}
	Since $\exp(-K)<1$ and
	$\mathcal{P}_1(\mathbb{R}^d)$ is complete under
	$\mathcal{W}_1$, the Banach fixed-point theorem implies that there
	exists a unique $\pi\in\mathcal{P}_1(\mathbb{R}^d)$ such that
	\begin{equation}\label{eq:invariant_at_one}
	\pi P_1=\pi.
	\end{equation}
	
	We now show that $\pi$ is invariant for the whole semigroup. For any
	$t\geq 0$, the semigroup property and~\eqref{eq:invariant_at_one}
	give
	\begin{align}
	(\pi P_t)P_1
	&=
	\pi P_{t+1}
	=
	(\pi P_1)P_t
	=
	\pi P_t.
	\end{align}
	Thus $\pi P_t$ is also a fixed point of the mapping $\mathcal{T}$.
	By uniqueness of the fixed point,
	\begin{equation}\label{eq:full_invariance}
	\pi P_t=\pi,
	\qquad t\geq 0.
	\end{equation}
	
	It remains to verify that $\pi$ has a finite second moment. By the
	fixed-point argument,
	\begin{equation*}
	\delta_0P_n
	\longrightarrow
	\pi
	\qquad\text{in }\mathcal{W}_1
	\quad\text{as }n\to\infty.
	\end{equation*}
	In particular, $\delta_0P_n$ converges weakly to $\pi$. Therefore,
	by the lower semicontinuity of $x\mapsto |x|^2$ and
	Lemma \ref{lm:2},
	\begin{align}
	\int_{\mathbb{R}^d}|x|^2\,\pi(\mathrm{d}x)
	&\leq
	\liminf_{n\to\infty}
	\int_{\mathbb{R}^d}|x|^2\,(\delta_0P_n)(\mathrm{d}x)
	\nonumber\\
	&=
	\liminf_{n\to\infty}
	\mathbb{E}_W[|X_n^0|^2]
	\nonumber\\
	&\leq
	\frac{|a(0)|^2}{K^2}
	+
	\frac{d\beta^2}{K}.
	\label{eq:invariant_moment_bound}
	\end{align}
	This proves~\eqref{eq:invariant_second_moment}.
	
	Finally, let $f:\mathbb{R}^d\to\mathbb{R}$ satisfy
	$\operatorname{Lip}(f)\leq 1$. By the invariance of $\pi$,
	\begin{equation*}
	\pi(f)
	=
	\int_{\mathbb{R}^d}P_tf(y)\,\pi(\mathrm{d}y).
	\end{equation*}
	Hence, using~\eqref{eq:semigroup_lipschitz_contraction},
	\begin{align}
	|\pi(f)-P_tf(x)|
	&=
	\left|
	\int_{\mathbb{R}^d}
	\bigl(P_tf(y)-P_tf(x)\bigr)
	\,\pi(\mathrm{d}y)
	\right|
	\nonumber\\
	&\leq
	\int_{\mathbb{R}^d}
	|P_tf(y)-P_tf(x)|
	\,\pi(\mathrm{d}y)
	\nonumber\\
	&\leq
	\exp(-Kt)
	\int_{\mathbb{R}^d}|x-y|\,\pi(\mathrm{d}y).
	\end{align}
	Taking the supremum over all $f$ satisfying
	$\operatorname{Lip}(f)\leq 1$ proves
	\eqref{eq:pointwise_convergence_to_invariant}.
	
	For a general initial condition $X_0\in L^2(\Omega_W)$, the Markov
	property gives
	\begin{equation*}
	\mathbb{E}_W[f(X_t)]
	=
	\mathbb{E}_W[P_tf(X_0)].
	\end{equation*}
	Therefore,
	\begin{align}
	\left|
	\pi(f)-\mathbb{E}_W[f(X_t)]
	\right|
	&=
	\left|
	\mathbb{E}_W
	\left[
	\pi(f)-P_tf(X_0)
	\right]
	\right|
	\nonumber\\
	&\leq
	\mathbb{E}_W
	\left[
	|\pi(f)-P_tf(X_0)|
	\right]
	\nonumber\\
	&\leq
	\exp(-Kt)
	\mathbb{E}_W
	\left[
	\int_{\mathbb{R}^d}
	|X_0-y|
	\,\pi(\mathrm{d}y)
	\right]
	\nonumber\\
	&=
	\exp(-Kt)
	\int_{\mathbb{R}^d}
	\mathbb{E}_W[|X_0-y|]
	\,\pi(\mathrm{d}y).
	\end{align}
	Taking the supremum over all $f$ satisfying
	$\operatorname{Lip}(f)\leq 1$ proves
	\eqref{eq:random_initial_convergence_to_invariant}.
\end{proof}

\begin{proof}[Proof of Lemma \ref{lm:2}]
	For each $n\in\mathbb{N}$, define the stopping time
	\begin{equation*}
	\tau_n
	:=
	\inf\left\{
	t\geq 0:
	|X_t|\geq n
	\right\}.
	\end{equation*}
	Under Assumption \ref{ass:sol1}, the solution is non-explosive, and hence
	$\tau_n\uparrow\infty$ almost surely as $n\to\infty$.
	
	Applying It\^o's formula to $e^{Ks}|X_s|^2$ on the random interval
	$[0,t\wedge\tau_n]$ gives
	\begin{align*}
	e^{K(t\wedge\tau_n)}
	|X_{t\wedge\tau_n}|^2
	={}&
	|X_0|^2
	+
	\int_0^{t\wedge\tau_n}
	e^{Ks}
	\left(
	K|X_s|^2
	+
	2\langle X_s,a(X_s)\rangle
	+
	d\beta^2
	\right)
	\,\mathrm{d}s
	\\
	&+
	2\beta
	\int_0^{t\wedge\tau_n}
	e^{Ks}
	\langle X_s,\mathrm{d}W_s\rangle.
	\end{align*}
	The stopped stochastic integral is a square-integrable martingale, since
	\begin{align*}
	\E_W\left[
	\int_0^{t\wedge\tau_n}
	4\beta^2 e^{2Ks}|X_s|^2
	\,\mathrm{d}s
	\right]
	\leq
	4\beta^2n^2
	\int_0^t e^{2Ks}\,\mathrm{d}s
	<\infty.
	\end{align*}
	Consequently, its expectation is zero.
	
	By \eqref{eqn:inequalitya},
	\begin{equation*}
	2\langle x,a(x)\rangle+K|x|^2
	\leq
	\frac{|a(0)|^2}{K}.
	\end{equation*}
	Taking expectations in the stopped It\^o identity therefore yields
	\begin{align*}
	\E_W\left[
	e^{K(t\wedge\tau_n)}
	|X_{t\wedge\tau_n}|^2
	\right]
	\leq{}&
	\E_W\left[|X_0|^2\right]
	+
	\left(
	\frac{|a(0)|^2}{K}
	+
	d\beta^2
	\right)
	\int_0^t e^{Ks}\,\mathrm{d}s
	\\
	={}&
	\E_W\left[|X_0|^2\right]
	+
	\left(
	\frac{|a(0)|^2}{K^2}
	+
	\frac{d\beta^2}{K}
	\right)
	\left(e^{Kt}-1\right).
	\end{align*}
	Since $\tau_n\uparrow\infty$ almost surely, we have
	\begin{equation*}
	e^{K(t\wedge\tau_n)}
	|X_{t\wedge\tau_n}|^2
	\longrightarrow
	e^{Kt}|X_t|^2
	\qquad
	\text{almost surely}.
	\end{equation*}
	Fatou's lemma then implies
	\begin{align*}
	e^{Kt}\E_W\left[|X_t|^2\right]
	\leq{}&
	\E_W\left[|X_0|^2\right]
	+
	\left(
	\frac{|a(0)|^2}{K^2}
	+
	\frac{d\beta^2}{K}
	\right)
	\left(e^{Kt}-1\right).
	\end{align*}
	Multiplying by $e^{-Kt}$ gives
	\begin{align*}
	\E_W\left[|X_t|^2\right]
	\leq{}&
	e^{-Kt}\E_W\left[|X_0|^2\right]
	+
	\left(
	\frac{|a(0)|^2}{K^2}
	+
	\frac{d\beta^2}{K}
	\right)
	\left(1-e^{-Kt}\right),
	\end{align*}
	which proves \eqref{eq:ito1}.
\end{proof}

\begin{proof}[Proof of Lemma~\ref{lem:RK2}] Note that by condition \eqref{eq:a1} of $a$ and the integral form of the true solution we have that
	\begin{align*}
	|\mathcal{R}_k| & = \left|\int\limits_{t_k}^{t_{k+1}} \left( a(X_s)  - a(X_{t_k}) \right) \,\mathrm{d}s \right| \le  \int\limits_{t_k}^{t_{k+1}} L | X_s  - X_{t_k}  |\,\mathrm{d}s\\
	&\le  \int\limits_{t_k}^{t_{k+1}} L \left| \int\limits_{t_k}^s  a(X_r)\,\mathrm{d}r + \beta(W_s - W_{t_k}) \right|\,\mathrm{d}s  
	\\&\le  \int\limits_{t_k}^{t_{k+1}} \left[ \left(L  \int\limits_{t_k}^s | a(X_r) |\,\mathrm{d}r\right) 
	+ \beta L|W_s - W_{t_k} | \right] \,\mathrm{d}s.
	\end{align*}
	This implies 
	\begin{align*}
	|\mathcal{R}_k|^2 &\le \left(\int\limits_{t_k}^{t_{k+1}} \left(L  \int\limits_{t_k}^s | a(X_r) |\,\mathrm{d}r\right) + \beta L|W_s - W_{t_k} | \,\mathrm{d}s\right)^2\\
	&\le 2L^2\left(\int\limits_{t_k}^{t_{k+1}} \left( \int\limits_{t_k}^s | a(X_r) |\,\mathrm{d}r\right)\,\mathrm{d}s\right)^2
	+ 2\beta^2 L^2\left(\int\limits_{t_k}^{t_{k+1}}\left|W_s - W_{t_k} \right|\,\mathrm{d}s\right)^2 \\
	&\le 2L^2h^2\int\limits_{t_k}^{t_{k+1}}  \int\limits_{t_k}^s | a(X_r) |^2\,\mathrm{d}r\,\mathrm{d}s
	+ 2\beta^2 L^2h\int\limits_{t_k}^{t_{k+1}}\left|W_s - W_{t_k} \right|^2\,\mathrm{d}s \,,
	\end{align*}
	where the last step is derived by the H\"older inequality.
	Due to Fubini's theorem
	\begin{align}
	\E_W\big[\E^W_k[|\mathcal{R}_k|^2]\big] &\le2 L^2 h^2  \E_W\left[\int\limits_{t_k}^{t_{k+1}}   \int\limits_{t_k}^s \E^W_k| a(X_r) |^2\,\mathrm{d}r \,\mathrm{d}s \right]
	+ d\beta^2 L^2h^3.\label{eq:bias:change}
	\end{align}
	Note that  condition \eqref{eq:a1} of $a$ leads to 
	\begin{equation}\label{eqn:agrowth}
	|a(x)|\leq L|x|+|a(0)|, \forall x \in \mathbb{R}^d,
	\end{equation}
	which gives
	\begin{align}\label{eqn:biaschange2}
	\begin{split}
	&\int\limits_{t_k}^{t_{k+1}}   \int\limits_{t_k}^s \E^W_k[| a(X_r) |^2]\,\mathrm{d}r \,\mathrm{d}s\\ 
	& \le 2\int\limits_{t_k}^{t_{k+1}}   \int\limits_{t_k}^s \left( |a(0)|^2+ L^2\E^W_k\big[\left|X_r\right|^2\big] \right) \,\mathrm{d}r \,\mathrm{d}s \\
	&= h^2|a(0)|^2 + 2L^2\int\limits_{t_k}^{t_{k+1}}   \int\limits_{t_k}^s \E_k^W\big[\left|X_r\right|^2\big]\,\mathrm{d}r \,\mathrm{d}s \,.
	\end{split}
	\end{align}
	From Lemma \ref{lm:2} we know that for any $r\in [t_k,t_{k+1})$, it holds 
	\begin{align*}
	\E_W\big[\E^W_k [\left|X_{r}\right|^2 ]\big] &\leq  e^{-K (r-t_k)}\E_W[ \left|X_{t_k}\right|^2 ] +\left(\frac{| a(0)|^2}{K^2} +\frac{d\beta^2}{K}\right)\left(1-e^{-K (r-t_k)}\right) \,.
	\end{align*}
	Finally, using L'H\^opital's rule, it is not hard to bound
	\begin{align*}
	&\int\limits_{t_{k}}^{t_{k+1}}  \int\limits_{t_{k}}^{s} e^{- K (r-t_k) }\,\mathrm{d}r\,\mathrm{d}s= \cfrac{Kh-1+\exp(-Kh)}{K^2} < h^2/2.    
	\end{align*}
	Now substituting the last two estimates above to \eqref{eqn:biaschange2} gives
	\begin{align*}
	&\E_W\left[\int\limits_{t_k}^{t_{k+1}}   \int\limits_{t_k}^s \E^W_k[| a(X_r) |^2]\,\mathrm{d}r \,\mathrm{d}s \right]\\
	&\le h^2|a(0)|^2 + 2L^2\E_W[ \left|X_{t_k}\right|^2 ]\int\limits_{t_{k}}^{t_{k+1}}  \int\limits_{t_{k}}^{s} e^{- K (r-t_k) }\,\mathrm{d}r\,\mathrm{d}s\\
	&\quad \, +\left(\frac{| a(0)|^2}{K^2} +\frac{d\beta^2}{K}\right)\int\limits_{t_{k}}^{t_{k+1}}  \int\limits_{t_{k}}^{s}\left(1-e^{-K (r-t_k)}\right)\,\mathrm{d}r\,\mathrm{d}s\\
	&\le h^2 \left(|a(0)|^2+\frac{| a(0)|^2}{K^2} +\frac{d\beta^2}{K}+2L^2\E_W[ \left|X_{t_k}\right|^2 ]\right)\\
	&=h^2 \left(\mathcal{C}_1+2L^2\E_W[ \left|X_{t_k}\right|^2 ]\right).
	\end{align*}
	Gathering the terms together for \eqref{eq:bias:change}, we obtain
	\begin{align*}
	\E_W[|\mathcal{R}_k|^2]&\le 2L^2h^4 \left(\mathcal{C}_1+2L^2\E_W[ \left|X_{t_k}\right|^2 ]\right) + d\beta^2 L^2h^3\\
	&\le  2L^2h^4\left(\mathcal{C}_1+2L^2e^{-K t_k}\E_W[ \left|X_0\right|^2 ] +2L^2\left(\frac{| a(0)|^2}{K^2} +\frac{d\beta^2}{K}\right)\left(1-e^{-K t_k}\right) \right)\\
	&\quad \, \, \,+ d\beta^2 L^2h^3\\
	&\le  2L^2h^4\left(|a(0)|^2+2L^2\E[ \left|X_0\right|^2 ] +(2L^2+1)\left(\frac{| a(0)|^2}{K^2} +\frac{d\beta^2}{K}\right) \right) + d\beta^2 L^2h^3\\
	&= 4L^4\mathcal{C}_2h^4 + d\beta^2 L^2h^3,
	\end{align*}
	where $$\mathcal{C}_2 = \left(\frac{|a(0)|^2}{2L^2}+\E_W[ \left|X_0\right|^2 ] +\left(1+\frac{1}{2L^2}\right)\left(\frac{| a(0)|^2}{K^2} +\frac{d\beta^2}{K}\right) \right).$$
	is bounded from below by $d\beta^2/K$.
\end{proof}

\begin{proof}[Proof of Lemma \ref{lem:RK1}]
	Indeed,
	\begin{equation}
	\label{e:EReq}\left|\E^W_k[\mathcal{R}_k]\right|^2 = \sum_{i=1}^d\left(\int\limits_{t_k}^{t_{k+1}} \E^W_k [a_i(X_s) - a_i(X_{t_k})] \,\mathrm{d}s\right)^2
	\end{equation}
	and, due to the Taylor formula, for each $i \in \{ 1, \ldots , d \}$ we have, using the multi-index notation,

	\begin{align*}
	&\E^W_k \left[ a_i(X_s) - a_i(X_{t_k})  \right] \\
	&=\E^W_k \langle \nabla a_i(X_{t_k})  , X_s - X_{t_k} \rangle\\
	&\quad \, \,\, + \E^W_k\int_0^1 (1-t) (X_s - X_{t_k})^T (\nabla^2 a_i(X_{t_k} + t(X_s - X_{t_k}))) (X_s - X_{t_k}) \,\mathrm{d}t  \\
	&\le |\nabla a_i(X_{t_k})| \cdot |\E^W_k [X_s - X_{t_k}]|\\
	&\quad \, \, \,+\frac{1}{2} \E^W_k \left[ \| \nabla^2 a_i (X_{t_k} + t(X_s - X_{t_k})) \|_{\operatorname{op}}  |X_s - X_{t_k}|^2 \right] \\
	&\le C_{a_i^{(1)}} |\E^W_k [X_s - X_{t_k}]| +\frac{ C_{a_i^{(2)}}}{2} \E^W_k [   |X_s - X_{t_k}|^2 ].
	\end{align*}
	It remains to bound $|\E^W_k [X_s - X_{t_k}]|$ and $ \E^W_k [   |X_s - X_{t_k}|^2 ]$. Recall that $$X_s - X_{t_k} = \int_{t_k}^s a(X_r) \,\mathrm{d}r + \beta(W_s - W_{t_k}).$$
	For the former one, we have that
	\begin{align*}
	|\E^W_k [X_s - X_{t_k}]| \le \E^W_k \left[ \int_{t_k}^s |a(X_r)| \,\mathrm{d}r \right] \le (s-t_k)(LC^{(1)} + |a(0)|) \,,
	\end{align*}
	where the last step is due to \eqref{eqn:agrowth} and the first moment bound $C^{(1)}$ on $X$. For the latter one, we have that
	\begin{align*}
	\E^W_k [ |X_s - X_{t_k}|^2] &\le \E^W_k \left[ 2\left| \int_{t_k}^s a(X_r) \,\mathrm{d}r\right|^2 + 2 \beta^2 |W_s - W_{t_k}|^2 \right] \\
	&\le 2 \E^W_k \left[ \int_{t_k}^s|a(X_r)|^2 \,\mathrm{d}r \right] + 2\beta^2 d(s-t_k) \\
	&\le 4(s-t_k)(L^2 C^{(2)} + |a(0)|^2) + 2\beta^2 d(s-t_k) \,,
	\end{align*}
	where  the last step is due to \eqref{eqn:agrowth} and the first moment bound $C^{(2)}$ on $X$.
	
	Combining all our bounds, we obtain a bound on $\E_k \left[ a_i(X_s) - a_i(X_{t_k})  \right]$ that can be plugged into \eqref{e:EReq}. Then, after integrating with respect to $s$, we obtain \eqref{e:ERbound}.
\end{proof}

\section{Proofs in Section \ref{sec:slgd} }\label{sec:numerical}

\begin{proof}[The proof of Lemma \ref{lem:estimatorProduct}.]
	\begin{align}
	\mathbb{E}_{U}\!\left[\left|\frac{1}{s}\sum_{i=1}^{m}a_{i}U_{i}\right|^{2}\right]
	&= \frac{1}{s^{2}}\,
	\mathbb{E}_{U}\!\left[\left|\sum_{i=1}^{m}a_{i}U_{i}\right|^{2}\right] \notag\\ 
	&= \frac{1}{s^{2}}\left[
	\sum_{i=1}^{m} a_{i}^{T}a_{i}\,\mathbb{E}_{U}[U_{i}^{2}]
	\;+\;
	\sum_{i=1}^{m}\sum_{\substack{j=1\\ j\ne i}}^{m}
	a_{i}^{T}a_{j}\,\mathbb{E}_{U}[U_{i}U_{j}]
	\right].\label{eq:U_expansion}
	\end{align}
	Note that
	$$\mathbb{E}_{U}[U_{i}^{2}] = \mathbb{P}(U_{i}=1) \cdot 1 = \frac{s}{m}$$
	$$\mathbb{E}_{U}[U_{i}U_{j}] = \mathbb{P}(U_{i}=1, U_{j}=1) \cdot 1 = \frac{\binom{s}{2}}{\binom{m}{2}} = \frac{s(s-1)}{m(m-1)}$$
	Plugging the two expectations into (\ref{eq:U_expansion}) gives
	\begin{align*}
	\mathbb{E}_{U}\!\left[\left|\frac{1}{s}\sum_{i=1}^{m}a_{i}a_{i}\right|^{2}\right]
	&= \frac{1}{s^{2}}\left(
	\sum_{i=1}^{m}\frac{s}{m}a_{i}^{T}a_{i}
	+\frac{s(s-1)}{m(m-1)}\sum_{\substack{i=1}}^{m}\sum_{\substack{j=1\\ j\ne i}}^{m}
	a_{i}^{T}a_{j}
	\right) \notag\\
	&=\frac{1}{s^{2}}\sum_{i=1}^{m}\frac{s}{m}a_{i}^{T}a_{i}
	+\frac{s-1}{sm(m-1)}
	\sum_{\substack{i=1}}^{m}\sum_{j=1}^{m}
	a_{i}^{T}a_{j}
	-\frac{s-1}{sm(m-1)}
	\sum_{i=1}^{m} a_i^T a_i
	\notag\\
	&=\frac{1}{s}\!\left(1-\frac{s-1}{m-1}\right)
	\sum_{i=1}^{m}\frac{1}{m}a_{i}^{T}a_{i}
	+\frac{m(s-1)}{s(m-1)}\!
	\left(\sum_{i=1}^{m}\frac{1}{m}a_{i}^{T}\right)
	\left(\sum_{j=1}^{m}\frac{1}{m}a_{j}\right)
	\notag\\
	&=\frac{1}{s}\!\left(\frac{m-s}{m-1}\right)
	\sum_{i=1}^{m}\frac{1}{m}a_{i}^{T}a_{i}
	+\frac{m(s-1)}{s(m-1)}\,\Lambda^{T}\Lambda .
	\end{align*}
\end{proof}

\begin{proof}[The proof of Corollary \ref{cor:estimatorDiff}.]
	From \ref{lem:estimatorProduct}, we have that
	\begin{align*}
	\mathbb{E}_{U}\!\left[|\Lambda-\hat{\Lambda}_{1}|^{2}\right]
	&= \mathbb{E}_{U}\!\left[\hat{\Lambda}_{1}^{T}\hat{\Lambda}_{1}\right]
	- \Lambda^{T}\Lambda \notag\\[6pt]
	&= \frac{1}{s}\!\left(\frac{m-s}{m-1}\right)
	\sum_{i=1}^{m}\frac{1}{m}a_{i}^{T}a_{i}
	+ \frac{m(s-1)}{s(m-1)}\Lambda^{T}\Lambda
	- \Lambda^{T}\Lambda \notag\\[6pt]
	&= \frac{1}{s}\!\left(\frac{m-s}{m-1}\right)
	\sum_{i=1}^{m}\frac{1}{m}a_{i}^{T}a_{i}
	- \frac{m-s}{s(m-1)}\Lambda^{T}\Lambda \notag\\[6pt]
	&\le \frac{1}{s}\!\left(\frac{m-s}{m-1}\right)
	\sum_{i=1}^{m}\frac{1}{m}a_{i}^{T}a_{i}. 
	\end{align*}
\end{proof}
\begin{proof}[The proof of Corollary \ref{cor:estimatorDiff}.]
	From Corollary 
	\ref{cor:estimatorDiff} we have
	$$\E[|a(Y_k)-\hat{b}_s(Y_k,U_k)|^2]\le  \frac{m-s}{sm(m-1)}\sum_{i=1}^{m}\E [\left| b_{i}(Y_k,\xi_i)\right| ^{2}].$$\\
	Using Assumption \ref{ass:sol1}, we have for any $i\in1,\ldots,m$
	$$\E[\left| b_{i}(Y_k,\xi_i)\right| ^{2}]\le L^2\E[|Y_k|^2] + \left| b_{i}(0,\xi_i)\right| ^{2},$$
	and using
	Lemma \ref{lemma:momentSLGD} gives the final assertion.
	
\end{proof}

\begin{proof}[The proof of Theorem \ref{bias:subsamp}.] Note that
	\begin{align*}
	\delta_{k+1} =& \theta_k + ha(\theta_k) + \beta \sqrt{h} Z_{k+1} - Y_k - h\hat{b}_s(Y_k) - \beta \sqrt{h} Z_{k+1}\\
	=& \delta_k + h(a(\theta_k) - \hat{b}_s(Y_k,U_k)),
	\end{align*}
	and 
	\begin{align*}
	&\E_U[ \langle\delta_k,a(\theta_k) - \hat{b}_s(\theta_k,U_k)\rangle]\\
	&\,=\E_UE_k^U[ \langle\delta_k,a(\theta_k) - \hat{b}_s(\theta_k,U_k)\rangle]\\
	&\, =\E_U\big[ \langle\delta_k,a(\theta_k) - E_k^U[\hat{b}_s(\theta_k,U_k)]\rangle\big]=0.
	\end{align*}
	
	Hence we have
	\begin{align*}
	\E_U[\left|\delta_{k+1}\right|^2 ]
	=& \E_U[|\delta_k|^2] + \E_U[ h^2|a(\theta_k) - \hat{b}_s(Y_k,U_k)|^2]\\
	&\,+\E_U[ 2h\langle\delta_k,a(\theta_k) - \hat{b}_s(Y_k,U_k)\rangle ]\\
	=& \E_U[|\delta_k|^2] + \E_U[ h^2|a(\theta_k) - \hat{b}_s(Y_k,U_k)|^2]\\
	&\,+\E_U[ 2h\langle\delta_k,a(\theta_k) - \hat{b}_s(\theta_k,U_k)\rangle] \\
	&\, +\E_U[ 2h\langle\delta_k,\hat{b}_s(\theta_k) - \hat{b}_s(Y_k,U_k)\rangle] \\
	\le & \E_U[|\delta_k|^2] + \E_U[ h^2|a(\theta_k) - \hat{b}_s(Y_k,U_k)|^2]-2hK\E_U[|\delta_k|^2] \\
	\le & \E_U[|\delta_k|^2] + 2\E_U[ h^2|a(\theta_k) - a(Y_k)|^2] \\
	& \, + 2\E_U[ h^2|a(Y_k)- \hat{b}_s(Y_k,U_k)|^2]-2hK\E_U[|\delta_k|^2],
	\end{align*}
	where we use Assumption \ref{ass:sol2} to get the last second line.
	
	Thus by Condition \eqref{eq:a1} and Corollary \ref{cor:slgddifference},
	\begin{align*}
	\E_U[\left|\delta_{k+1}\right|^2] 
	&\le (1-2hK+ 2h^2L^2)\E_U[|\delta_k|^2]+2h^2\frac{m-s}{s(m-1)}\mathcal{L} \,,
	\end{align*}
	which finishes the proof.
\end{proof}
\begin{proof}[The proof of Theorem \ref{thm:varianceSLGD}.]
	Consider two independent SLGD chains $Y_{k}$ and $\bar Y_{k}$. Let us denote $\eta_k:=Y_k-\bar Y_k$. Following the argument in Theorem \ref{var:euler:est}, in particular \eqref{eqn:varianceequiv}, it remains to check $\E [|\eta_{k+1}|^2 ]$.
	Then we have
	\begin{align*}
	&\E \left[\E_k[|\eta_{k+1}|^2 ]\right]\\
	&\, \le \E[|\eta_k|^2] + 2h\E\langle\eta_k,\E^U_k[ \hat{b}(Y_k,U_k)]-\E_U[ \hat{b}(\bar Y_k,U_k)]\rangle\notag\\
	&\qquad +\E\left[\E^U_k[|\hat{b}(Y_k,U_k)-\hat{b}(\bar Y_k,U_k)|^2]\right]h^2 + 4\beta^2dh  \\
	&\, \le (1-2Kh)\E[|\eta_k|^2] +3h^2\E[|a(Y_k)-a(\bar Y_k)|^2]\\
	&\qquad +6h^2\E\left[\E^U_k[|\hat{b}(Y_k,U_k)-a(\bar Y_k)|^2]\right]  + 4\beta^2dh  \\
	&\, \le (1-2hK+3L^2h^2) \E[|\eta_k|^2] +4\beta^2dh\\
	&\qquad +6\E\left[\E_U[|a(Y_k)-\hat{b}(Y_k)|^2]\right]h^2  \\
	&\,\le  (1-2hK+3L^2h^2) \E[|\eta_k|^2] +4\beta^2dh+ 6h^2\frac{m-s}{s(m-1)} \mathcal{L} ,
	\end{align*}
	where, we use Assumption \ref{ass:sol2} to get the second inequality, Eqn. \eqref{ass:sol1} to get the third, and    Corollary \ref{cor:slgddifference} to get the last line.
\end{proof}

\section{The derivation of \eqref{eq:SGaussianPosterior} }\label{append:posterior}

To derive the posterior distribution \eqref{eq:SGaussianPosterior}, let
$y_{1:m}:=(y_1,\ldots,y_m)$. By Bayes' formula, the posterior density
of $\theta$ is proportional to the product of the prior density and
the likelihood:
\begin{align}    \label{eq:posterior_density_derivation}
\begin{split}
p(\theta\mid y_{1:m})
&\propto
p(\theta)
\prod_{i=1}^m p(y_i\mid\theta)
\\
&\propto
\exp\left(
-\frac{m\theta^2}{2\sigma_\theta^2}
\right)
\prod_{i=1}^m
\exp\left(
-\frac{(y_i-\theta)^2}{2\sigma_y^2}
\right)\\
&=
\exp\left(
-\frac{1}{2}
\left[
\frac{m}{\sigma_\theta^2}\theta^2
+
\frac{1}{\sigma_y^2}
\sum_{i=1}^m (y_i-\theta)^2
\right]
\right).  
\end{split}
\end{align}
Expanding the quadratic term in $\theta$ gives
\begin{align}
\frac{m}{\sigma_\theta^2}\theta^2
+
\frac{1}{\sigma_y^2}
\sum_{i=1}^m (y_i-\theta)^2
&=
m\left(
\frac{1}{\sigma_\theta^2}
+
\frac{1}{\sigma_y^2}
\right)\theta^2
-
\frac{2}{\sigma_y^2}
\left(
\sum_{i=1}^m y_i
\right)\theta
+
\frac{1}{\sigma_y^2}
\sum_{i=1}^m y_i^2.
\label{eq:posterior_quadratic_expansion}
\end{align}
Define
\begin{equation*}
A
:=
m\left(
\frac{1}{\sigma_\theta^2}
+
\frac{1}{\sigma_y^2}
\right),
\qquad
B
:=
\frac{1}{\sigma_y^2}
\sum_{i=1}^m y_i.
\end{equation*}
Then
\begin{equation*}
A\theta^2-2B\theta
=
A\left(\theta-\frac{B}{A}\right)^2
-
\frac{B^2}{A}.
\end{equation*}
Since the terms independent of $\theta$ are absorbed into the
normalising constant, it follows that
\begin{equation*}
p(\theta\mid y_{1:m})
\propto
\exp\left(
-\frac{A}{2}
\left(\theta-\frac{B}{A}\right)^2
\right).
\end{equation*}
Therefore,
\begin{equation}
\theta\mid y_{1:m}
\sim
\mathcal{N}\left(
\frac{B}{A},
\frac{1}{A}
\right).
\label{eq:posterior_before_simplification}
\end{equation}
The posterior variance is
\begin{align}
\frac{1}{A}
&=
\frac{1}{
	m\left(
	\frac{1}{\sigma_\theta^2}
	+
	\frac{1}{\sigma_y^2}
	\right)
}
=
m^{-1}
\left(
\frac{1}{\sigma_\theta^2}
+
\frac{1}{\sigma_y^2}
\right)^{-1}.
\end{align}
The posterior mean is
\begin{align}
\frac{B}{A}
=
\frac{
	\frac{1}{\sigma_y^2}
	\sum_{i=1}^m y_i
}{
	m\left(
	\frac{1}{\sigma_\theta^2}
	+
	\frac{1}{\sigma_y^2}
	\right)
}
=
\frac{1}{m}
\left(
\sum_{i=1}^m y_i
\right)
\frac{
	\frac{1}{\sigma_y^2}
}{
	\frac{1}{\sigma_\theta^2}
	+
	\frac{1}{\sigma_y^2}
}
=
\frac{1}{m}
\left(
\sum_{i=1}^m y_i
\right)
\left(
\frac{\sigma_y^2}{\sigma_\theta^2}
+
1
\right)^{-1}.
\end{align}
Consequently,
\begin{equation*}
\theta\mid y_{1:m}
\sim
\mathcal{N}\left(
m^{-1}
\left(
\sum_{i=1}^m y_i
\right)
\left(
\frac{\sigma_y^2}{\sigma_\theta^2}
+
1
\right)^{-1},
\,
m^{-1}
\left(
\frac{1}{\sigma_\theta^2}
+
\frac{1}{\sigma_y^2}
\right)^{-1}
\right),
\end{equation*}
which is precisely the distribution defined in
\eqref{eq:SGaussianPosterior}.

\end{document}